\documentclass[12pt]{amsart}%
\usepackage{amsmath}
\usepackage{amssymb}
\usepackage{amsthm}
\usepackage{amscd}
\usepackage[dvipdfmx]{graphicx}
\usepackage[mathscr]{eucal}
\usepackage{bm}
\usepackage{xcolor}
\DeclareMathOperator{\sinc}{sinc}
\makeindex
\newtheorem{Thm}{Theorem}[section]
\theoremstyle{definition}
\newtheorem{Theorem}[Thm]{Theorem}
\newtheorem{Lemma}[Thm]{Lemma}
\newtheorem{Corollary}[Thm]{Corollary}
\newtheorem{Proposition}[Thm]{Proposition}
\newtheorem{Definition}[Thm]{Definition}
\newtheorem{Example}[Thm]{Example}
\newtheorem{Exercise}{Exercise}
\theoremstyle{remark}
\newtheorem{Remark}{Remark}
\font\sy=cmsy10

\font\ym=msbm10

\newcommand{\Z}{\text{\ym Z}}

\newcommand{\R}{\text{\ym R}}

\newcommand{\C}{\text{\ym C}}

\newcommand{\cD}{{\hbox{\sy D}}}
\newcommand{\cE}{{\hbox{\sy E}}}

\newcommand{\cI}{{\hbox{\sy I}}}
\newcommand{\cL}{{\hbox{\sy L}}}

\renewcommand{\Im}{\,\hbox{Im}\,}
\title[Elementary Integration]
{An Elementary but Logical Approach to Integration}
\dedicatory{To the memory of Yamagami Nobuko}
\author{Shigeru Yamagami}
\begin{document}
\maketitle


\begin{abstract}
  We present a replacement for traditional Riemann integrals in undergraduate calculus,
  which supplements naive precalculus and at the same time opens
  a way to more sophisticated theories such as Lebesgue integration. 
\end{abstract}

\tableofcontents


Lots of textbooks on calculus adopt traditional approaches to integration based on the so-called Riemann integral.
Some authors (Bourbaki, for example) are critical in this and
trying to pay much attention to linkage to the advanced theory of Lebesgue integration, 
at least in the case of single variable.  

 For realistic applications, however, we need a naive understanding of integration in the form of an approximation by (or a limit of)
 a large number sum of small quantities, which should be theorefore retained in any approach. 

 Defects are culminated in the description of repeated calculus of improper integrals.
 Theoretically, integrability of multi-variable functions is required there 
 but no practical and useful criterion is supplied in elementary courses. Thus, even if repeated integrals are possible in a safe manner,
 they can not be logically related to the total integrals.
 Of course, in Lebesgue integration, this can be dealt with by the Fubini-Tonelli theorem but at much cost for sophistication. 
 More elementary but effective formulation is desirable even for practical integration.

 We here systematically use monotone-limit extensions of elementary quadrature, which are of intermediate character in Daniell's approach to
 Lebesgue integration but it works fairly well in concrete integrals and provides preliminaries to advanced theories as well 
 with good experiences for further achievement. 

 The author is grateful to Amazon reviewer susumukuni for many useful comments on a draft version 
 as well as constant impetus during the preparation of this monograph, 
 to Tsuyoshi Kajiwara and Tomohiro Hayashi for illuminating discussions on the subject.


 \section{Sets and Functions}
 Set notation: Given two sets $X$ and $Y$, their product $X\times Y$ is the set of all ordered pairs $(x,y)$ ($x \in X$, $y \in Y$),
 $Y^X$ is the set of all maps of $X$ into $Y$, and the power set $2^X$ of $X$ is the set of all subsets of $X$.
 So $\R^X$ denotes the set of real-valued functions on $X$, for example.
 
 When $X$ and $Y$ are finite sets with their numbers of elements denoted by $|X|$ and $|Y|$, we have
 $|X\times Y| = |X|\,|Y|$, $|Y^X| = |Y|^{|X|}$ and $|2^X| = 2^{|X|}$. 

 Multiple products are defined in a similar fashion and identified in an associative manner:
 $(X\times Y)\times Z = X\times Y\times Z = X\times(Y\times Z)$.
 When multiple product is performed on a single set $Y$ repeatedly, $Y\times \dots\times Y$ ($n$-times repetition) is denoted by $Y^n$.
 Thus $\R^n$ denotes the set of $n$-tuples of real numbers. 
 If $n = |X|$ with $X$ a finite set and elements of $X$ are listed by $x_1,\dots,x_n$, $Y^X$ is naturally identified with $Y^n$.

 {\small
 \begin{Remark}
   Throughout this monograph, the notation $|\cdot|$ is used in a multiple way:
   For sets, it denotes the size of its extent. For numbers and numerical vectors, it expresses the length. 
 \end{Remark}}

Given a set $X$ and a condition $P$ on $x \in X$, we denote by $[P]$ the subset of $X$ consisting of $x \in X$ which satisfies $P$.
\index{+supportcondition@$[P]$ the set of a property $P$} 
 As an example, if $f$ and $g$ are real functions on a set $X$, $[f < g] = \{ x \in X; f(x) < g(x)\}$.
 When $P$ holds for any $x \in A$ ($A$ being a subset of $X$), i.e., $A \subset [P]$, we shall also write $P$ ($x \in A$). 
 
The order relation in $\R$ is extended to real-valued functions as well:
For functions $f,g:X \to \R$, we write $f \leq g$ if $f(x) \leq g(x)$ ($x \in X$).
It is convenient to extend the ordered set $\R$ by adding formal elements $\pm\infty$ which are upper and lower bounds of $\R$ respectively.
This is in fact not so formal because $\R$ is order isomorphic to an open interval $(-1,1)$ by a monotone bicontinuous bijection
$h: \R \to (-1,1)$ such as $h(t) = t/(1+|t|)$ or $h(t) = (2/\pi) \arctan t$ so that
the \textbf{extended real line}\index{extended real line} $[-\infty,\infty]$ corresponds to the closed interval $[-1,1]$. 

 A sequence $(f_n)$ of real-valued functions is said to be \textbf{increasing} (\textbf{decreasing})\index{increasing}\index{decreasing}
 if $f_n \leq f_{n+1}$ ($f_n \geq f_{n+1}$) for $n \geq 1$.
 When $f$ is the limit function of $(f_n)$, i.e., $\displaystyle f(x) = \lim_{n \to \infty} f_n(x)$ for each $x \in X$, we write 
 $f_n \uparrow f$ ($f_n\downarrow f$). 


Complex or real are used as an adjective on functions to \textit{indicate their ranges}.

For a (scalar-valued) function $f$ defined on a set $X$ and a subset $A \subset X$, we make an overall use of the notation
\index{+norm@$\parallel \cdot \parallel_A$ sup norm over $A$}
\[
  \| f\|_A = \sup\{ |f(x)|; x \in A\}, 
\]
which satisfies the so-called seminorm condition: $\| \alpha f\|_A = |\alpha| \| f\|_A$ and $\| f+g\|_A \leq \| f\|_A + \| g\|_A$.
When $A$ is obvious, we simply write $\| f\|$.

 For a complex function $f$ defined on a set $X$, it gives rise to a map $2^X \to 2^\C$ by $A \mapsto \{ f(a); a \in A\}$.
 Although logically ambiguous when both $A \subset X$ and $A \in X$ occur, it is customary to write $f(A) = \{ f(a); a \in A\}$
 (called the image of $A$ under $f$). Likewise a map $2^\C \to 2^X$ is defined by $B \mapsto f^{-1}(B) = \{ x \in X; f(x) \in B\}$.
 Note that the inverse image $f^{-1}(B)$ of $B$ is also expressed by $[f \in B]$. 

 A function $f$ is said to be \textbf{positive}\index{positive function} if $f(X) \subset [0,\infty)$.
 Thus a positive function may take $0$ as its value. If you need a function satisfying $f(x) > 0$ ($x \in X$), we say that
 $f$ is strictly positive. Since we occasionally work with complex functions, we shall avoid `non-negative' to indicate our `positive'.
  
 Given a subset $A \subset X$, its \textbf{indicator}\index{indicator} is a function $1_A$ defined by
 $1_A(x) = 1$ or $0$ according to $x \in A$ or not. Thus $1_A \in \{ 0,1\}^X \subset \R^X$ and
 the correspondence $2^X \ni A \mapsto 1_A \in \{0,1\}^X$ is bijective. 
 
 Based on this fact, we shall identify sets and their indicators in case of no confusion.

 \begin{Example}
   Let $(A_i)$ be a family of sets. Then $\sum_i A_i$ denotes a set if and only if $\bigcup_i A_i$ is a disjoint union, i.e.,
   $\bigcup_i A_i = \bigsqcup_i A_i$.

   Let $(f_i)$ be a family of functions on a set $X$ and $\bigsqcup_i A_i \subset X$,
   then $\sum_i A_if_i$ is a function described by
   \[
     \Bigl(\sum_iA_if_i\Bigr)(x) =
     \begin{cases}
       f_i(x) &\text{if $x \in A_i$ for some $i$,}\\
       0 &\text{otherwise.}
     \end{cases}
   \]
 \end{Example}
 

 We say that a function $f$ is \textbf{supported}\index{supported} by a set $A$ if $Af = f$.

 {\small
 \begin{Remark}
 To avoid possible confusion such as $r[a,b]$ and $[ra,rb]$ for intervals and $r>0$, we prefer $Af$ to $fA$. 
\end{Remark}}

\begin{Exercise}
  $A \cap B = AB = A \wedge B$, $X \setminus A = X - A$ and
  $A\cap B + A \cup B = A+B$. 
\end{Exercise}

\begin{Example}
To get more insight on its usage and conveniency, we take up the sieve formula (the inclusion-exclusion principle)
in combinatorics. 

Given finitely many subsets $A_1, \dots, A_n$ of $X$,
de Morgan's law is expressed by
$X - (A_1 \cup \dots \cup A_n) = (X -A_1)\cdots(X-A_n)$, which is combined with its algebraic expansion 
\[
 X - (A_1+\dots + A_n) + \sum_{i<j} A_i \cap A_j + \cdots + (-1)^nA_1\cdots A_n
\]
to obtain the identity 
\[
  A_1 \cup \dots \cup A_n = A_1+\dots+A_n - \sum_{i<j} A_iA_j + \dots + (-1)^{n-1} A_1\cdots A_n. 
\]
When $A_1,\dots,A_n$ are all finite sets, we can evaluate these by counting measure to get to the sieve formula.
\end{Example}

A function $f$ on a set $X$ is said to be \textbf{simple}\index{simple function} if it satisfies the following equivalent conditions.
\begin{enumerate}
\item
  $f$ is a linear combination of finitely many subsets of $X$.
\item
  The range $f(X)$ is a finite set of scalars. 
\end{enumerate}

\begin{Exercise}
Check the equivalence of (i) and (ii). 
\end{Exercise}

\begin{Definition}\label{ll}
  A real vector space $L$ consisting of real-valued functions on a set $X$ is called 
a \textbf{linear lattice}\index{linear lattice} or a vector lattice\index{vector lattice} if 
\[
f, g \in L\ \Longrightarrow\ f \vee g, f \wedge g \in L, 
\]
where \index{+max@$\vee$ max} \index{+min@$\wedge$ min}
\[
(f \vee g)(x) = \max\{f(x), g(x)\}, 
\quad 
(f \wedge g)(x) = \min\{ f(x), g(x)\}.
\]

From the identity $2(s\diamond t) = s+t \mp |s-t|$ ($s,t \in \R$), the condition is equivalent to
$|f| \in L$ ($f \in L$), i.e., $L$ is closed under taking absolute-value functions. 

Given a linear lattice $L$, we define the \textbf{positive part}\index{positive part} of $L$ by 
$L^+ = \{ f \in L; f \geq 0 \}$, which generates $L$ linearly in view of
$f = (0\vee f) + (0\wedge f) = (0\vee f) - 0\vee(-f)$ ($f \in L$). 
\end{Definition}

\begin{figure}[h]
  \centering
 \includegraphics[width=0.4\textwidth]{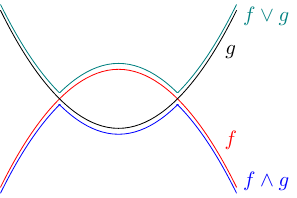}
 \caption{Lattice Operation}
\end{figure}

\begin{Definition}
  A linear functional $I: L \to \R$ on a linear lattice $L$ is said to be
  \textbf{positive}\index{positive functional} if $I(f) \geq 0$ ($f \in L^+$).
  A positive linear functional, simply a positive functional, is \textbf{continuous}\index{continuous} if
  $f_n \in L^+$ satisfies $f_n \downarrow 0$, then $I(f_n) \downarrow 0$.
  A continuous positive functional is called a \textbf{preintegral}\index{preintegral} (or a Daniell integral). 
  



An \textbf{integral system}\index{integral system} on a set $X$ is defined to be 
a couple $(L,I)$ of a linear lattice $L$ on $X$ and a preintegral $I$ on $L$. 
\end{Definition}

\begin{Exercise} A linear functional $I$ on a linear lattice $L$ is positive if and only if
  $|I(f)| \leq I(|f|)$ ($f \in L$). 
\end{Exercise}

By a \textbf{division}\index{division} in a set $X$,
we shall mean a finite family $\cD$ of mutually disjoint non-empty subsets of $X$ and, given a division $\cD$, let
$\R\cD = \sum_{D \in \cD} \R D$ be the set of linear combinations of sets in $\cD$, which is an algebra and a linear lattice at the same time
(called an \textbf{algebra-lattice}\index{algebra-lattice}).
The algebra $\R \cD$ has a unit element given by $[\cD] = \bigsqcup_{D \in \cD} D$ (called the \textbf{support}\index{support} of $\cD$).

Since $\cD$ is linearly independent as a family in $\R \cD$,
the vector space $\R \cD$ is naturally isomorphic to $\R^\cD$ as an algebra-lattice  
and any positive function $\mu$ on $\cD$ is extended to a positive functional
$I$ on $\R \cD$, which is obviously continuous. Thus there is a one-to-one correspondence (by restriction and extension)
between positive functions on $\cD$ and preintegrals on $\R \cD$. 

Among divisions in $X$, we define an order relation $\cD \prec \cE$ by $\R\cD \subset \R \cE$,
which is equivalently described by the condition that each $D \in \cD$ is a disjoint union of $E \in \cE$ included in $D$. 
We say that $\cE$ is a \textbf{subdivision}\index{subdivision} of $\cD$ if $\cD \prec \cE$ and $[\cD] = [\cE]$.
Let $I$ and $J$ be preintegrals on $\R \cD$ and $\R \cE$ respectively. Then $J$ is an extension of $I$ if and only if
they are related by $I(D) = \sum_{E \subset D} J(E)$ ($D \in \cD$, $E \in \cE$).

\bigskip
We assume and follow standard terminology and notations on the topology in $\R^d$.
For example, given a subset $D \subset \R^d$, $\partial D$ denotes the boundary of $D$.
\index{boundary}\index{+boundary@$\partial D$ boundary of $D$}
Non-standard is the notation and the meaning for the support of a function $f$ defined on a subset $A$ of $\R^d$:
The closure of $[f \not= 0]$ in $\R^d$ is called the \textbf{support}\index{support} of $f$ and is denoted by $[f]$.
\index{+support@$[f]$ the support of a function $f$}
In other words, our support of $f$ is the ordinary support of the zero extension of $f$ to $\R^d$. 

For a subset $U$ of $\R^d$, $C(U)$ denotes the set of continuous functions on $U$. 
When $U$ is open, a continuous function $f \in C(U)$ \index{+C@$C(U)$ continuous} is said to have a \textbf{compact support}\index{compact support}
if $[f]$ is bounded and $[f] \subset U$.
The set of continuous functions on $U$ having compact supports is denoted by $C_c(U)$. \index{+Cc@$C_c(U)$ continuous and compact}


\section{Definite and Indefinite Integrals}
Here we discuss definite integrals of functions of a single variable.

A \textbf{step function}\index{step function} is by definition a linear combination of bounded intervals in $\R$.
Let $S(\R)$ be the linear space of step functions. \index{+step-one@$S(\R)$ step function}
For a bounded interval $D$ with $a \leq b$ endpoints, its width (or length) $b-a$ is denoted by $|D|$.
\index{+width@$\lvert D \rvert$ width of $D$}


Given a finite\footnote{We exclusively deal with finite partitions and the adjective `finite' is henceforth omitted.}
\textbf{partition}\index{partition} $\Delta = \{ t_0 < t_1 < \dots < t_n\}$ in $\R$,
open intervals $(t_0,t_1), \dots, (t_{n-1},t_n)$ and one-point intervals $[t_0,t_0],\dots,[t_n,t_n]$
(called interval parts)\index{interval parts} 
are linearly independent in $S(\R)$ and, if we denote by $\R\Delta$ the set of their linear spans,
$\R\Delta$ is an algebra-lattice with
the width function on interval parts in $\Delta$ linearly extended to a positive linear functional $I_\Delta$.
It is immediate to see that if $\Delta'$ is a refinement of $\Delta$, then $\R\Delta \subset \R\Delta'$ and
$I_{\Delta'}$ extends $I_\Delta$. 

Given finitely many intervals $D_1,\dots, D_m$,
we can find a finite partition $\Delta$ so that each $D_j$ is a sum of interval parts in $\Delta$.
Note that an interval part $\Delta_i$ of $\Delta$ satisfies
$D_j\Delta_i = \Delta_i$ or $0$ according to $\Delta_i \subset D_j$ (denoted by $i \prec j$) or not.
Moreover we have an expression $D_j = \sum_{i \prec j} \Delta_i$ together with an obvious equality $|D_j| = \sum_{i \prec j} |\Delta_i|$.

\begin{Lemma}
  The step function space $S(\R)$ is an algebra-lattice. 
  The width function is extended to a positive linear functional $I$ on $S(\R)$,
  which is referred to as the \textbf{width integral}\index{width integral}. 
\end{Lemma}

\begin{proof}
  Since interval parts in $\Delta$ are idempotents in the function algebra $\R^\R$, their linear combinations constitute
  an algebra-lattice, which is inherited from $S(\R)$. 

  To see that the width function is well-extended to a positive linear functional,
  let $\sum_j \alpha_j D_j = \sum_k \beta_k E_k$ with $D_j$ and $E_k$ bounded intervals. Choose a partition $\Delta$
  so that $D_j, E_k \in \R\Delta$. Then
  \begin{multline*}
    \sum_j \alpha_j |D_j| = \sum_{i \prec j} \alpha_j |\Delta_i|
    = I_\Delta(\sum_j \alpha_j D_j)\\
    = I_\Delta(\sum_k \beta_k E_k) 
    = \sum_{i \prec k} \beta_k I_\Delta(\Delta_i)
    = \sum_k \beta_k |E_k|.
  \end{multline*}
\end{proof}

The width function is now extended to a set $A \in S(\R)$ by $|A| = I(A)$.
Note that such an $A$ is exactly a union of finitely many bounded intervals.
The following is a key toward integral extensions. 

\begin{Lemma}\label{key}
 Let $\bigsqcup_{n \geq 1} D_n$ be a decomposition of an open interval $(a,b)$ into countably many bounded intervals. Then
  $b-a = \sum_{n=1}^\infty |D_n|$. 
\end{Lemma}

\begin{proof}
  Intuitively this seems obvious because it just prevents any leakage from the summation
  and you may take it for granted\footnote{This resembles changing paths 
    in contour integrals of complex analysis, which is intuitively obvious but not logically at all.}
  to see further developments.
  The proof itself is, however, not difficult once you know the Heine-Borel covering theorem\footnote{The covering theorem itself is in fact
    established as sophistication of the proof of this lemma.}:

  Since $\sum_{j=1}^n D_j \leq (a,b)$ as functions on $\R$, taking the width integral gives
  $\sum_{j=1}^n |D_j| \leq b-a$ and then $\sum_{n \geq 1} |D_n| \leq b-a$.

  To get the reverse inequality, 
  given $\epsilon>0$, by replacing each $D_n$ with a slightly large open interval $U_n$ satisfying $|U_n| \leq |D_n| + \epsilon/2^n$,
  we obtain an open covering $(U_n)$ of $[a+\epsilon,b-\epsilon]$ and then find a finite subcover
  $(U_{n_j})_{1 \leq j \leq k}$ by the Heine-Borel (\ref{HB}) so that
  \[
    [a+\epsilon,b-\epsilon] \leq \bigcup_{j=1}^k U_{n_j} \leq \sum_{j=1}^k U_{n_j}
  \]
  is evaluated by the width integral to get
  \[
    b-a - 2\epsilon \leq \sum_{j=1}^k |U_{n_j}| \leq \sum_{n \geq 1} |U_n| \leq \sum_{n \geq 1} |D_n| + \sum_{n \geq 1} \frac{\epsilon}{2^n}
    = \sum_{n \geq 1} |D_n| + \epsilon. 
  \]
 Thus $b-a \leq \sum_{n \geq 1} |D_n| + 3\epsilon$. 
\end{proof}

\begin{Corollary}\label{monoC}
Monotone continuity holds for the width integral. 
\end{Corollary}

\begin{proof}
  We first claim that, if $\bigsqcup_{n \geq 1} A_n$ is a decomposition of a set $A \in S(\R)$ into $A_n \in S(\R)$, then 
  $|A| = \sum_{n=1}^\infty |A_n|$. 
  In fact, $A$ is a finite disjoint union of open intervals $(a,b)$ and points.
  For points, the width integral satisifes the equality by $0 = \sum_n 0$ and, for open intervals, the assertion in the lemma
  gives $|(a,b)| = \sum_n |(a,b) \cap A_n|$. (Note here that $(a,b) \cap A_n$ is a disjoint union of finitely many bounded intervals.)
  Summing these up, we obtain the claim. 
  
  Let $(h_n)_{n \geq 1}$ be a decreasing sequence of step functions satisfying $h_n \downarrow 0$.
  We show that the width integral fulfills $I(h_n) \downarrow 0$.

    Since both $[h_1>0]$ and $[h_n \leq \epsilon]h_n$ are step functions and satisfy
  $[h_n \leq \epsilon] h_n \leq [h_1>0] \epsilon$ for any $\epsilon > 0$, evaluation by the width integral gives
  \[
    I(h_n) = I([h_n \leq \epsilon]h_n) + I([h_n > \epsilon] h_n) \leq  \epsilon \Bigl|[h_1 > 0]\Bigr| + \| h_1\| \Bigl|[h_n > \epsilon]\Bigr| 
  \]
  and the continuity is reduced to showing $\Bigl|[h_n> \epsilon]\Bigr| \downarrow 0$ as $n \to \infty$.
  
  To see this, we rewrite $[h_n > \epsilon] \downarrow \emptyset$ to the decomposition 
  \[
    [h_m > \epsilon] = \bigsqcup_{n \geq m} \Bigl([h_n > \epsilon] \setminus [h_{n+1} > \epsilon]\Bigr)
    = \bigsqcup_{n \geq m} [h_n > \epsilon] \cap [h_{n+1} \leq \epsilon]
  \]
  for each $m \geq 1$. Since $[h_n > \epsilon]$ and $[h_n > \epsilon] \cap [h_{n+1} \leq \epsilon]$ belong to $S(\R)$,
  we can apply the above claim to have 
  \[
  \Bigl|[h_m > \epsilon]\Bigr| = \sum_{n \geq m} \Bigl|[h_n > \epsilon] \cap [h_{n+1} \leq \epsilon]\Bigr|, 
  \]
  which approaches $0$ as $m \to \infty$ because $\sum_{n \geq 1} \Bigl|
  [h_n > \epsilon] \cap [h_{n+1} \leq \epsilon]\Bigr| = \Bigl|[h_1>\epsilon]\Bigr| < \infty$.
  %
  %
\end{proof}

Thus the width integral on step functions is a preintegral and gives an integral system on $\R$.

To see how the above reassembling lemma is non-trivial, consider the following generalization due to Stieltjes:
Let $\phi: \R \to \R$ be a (weakly)  increasing function. Remark first that jumping points $c$ satisfying $\phi(c-0) < \phi(c+0)$ are countable
because given a finite interval $[a,b]$ $\{ c \in [a,b]; \phi(c+0) - \phi(c-0) \geq 1/n\}$ is a finite set for every $n=1,2,\dots$.

Now the Stieltjes mass is assigned to finite interval parts by
\[
  |(a,b)|_\phi = \phi(b-0) - \phi(a+0),
  \quad 
  |\{ c\}|_\phi = \phi(c+0) - \phi(c-0)
\]
and linearly extended to a positive functional $I_\phi$ on $S(\R)$, which is called the \textbf{Stieltjes integral}\index{Stieltjes integral}.
Note here that values at jumping points are irrelevant in this construction and it is customary to impose left or right continuity on 
$\phi$ so that $\phi$ is uniquely determined from the Stieltjes integral. 

We here claim that, given a decomposition $R = \sqcup R_j$ of a bounded interval $R$ into a countable sequence $(R_j)$ of subintervals,
\[
  |R|_\phi = \sum_j |R_j|_\phi.
\]
Again non-trivial is the inequality $|R|_\phi \leq \sum_j |R_j|_\phi$. 

For a bounded non-open interval, we can move boundary points slightly outer to make it open but the Stieltjes mass difference remain small.
This is possible from the limiting definition of the Stieltjes mass:
Let $R_j = (a_j,b_j]$ for example. Then
\[
  |R_j|_\phi = \lim_{\epsilon \to +0} (\phi(b_j+\epsilon) - \phi(a_j+0)) = \lim_{\epsilon \to +0} |(a_j,b_j+\epsilon)|_\phi. 
\]
Consequently, given $\epsilon>0$, we can find an open set $R_j^\epsilon$ including $R_j$ so that 
$|R_j^\epsilon|_\phi \leq |R_j|_\phi + \epsilon/2^j$ (we take $R_j^\epsilon = R_j$ if $R_j$ is open). 

First consider $R = [a,b]$. By the Heine-Borel covering theorem (\ref{HB}), for a sufficently large $n \geq 1$,
$R \subset \bigcup_{j=1}^n R_j^\epsilon \leq \sum_{j=1}^n R_j^\epsilon$ and hence
\[
  |R|_\phi = I_\phi(R) \leq \sum_{j=1}^n I_\phi(R_j^\epsilon) = \sum_{j=1}^n |R_j^\epsilon|_\phi
  \leq \sum_{j=1}^n (|R_j|_\phi + \epsilon/2^j) \leq \sum_{j=1}^n |R_j|_\phi + \epsilon. 
\]
Since $\epsilon>0$ is arbitrary, this gives $|R|_\phi \leq \sum_j |R_j|_\phi$ and the claim holds. 

When $R = (a,b)$, add $\{ a\}$, $\{ b\}$ to $(R_j)$ and apply the reassembling formula for $[a,b]$ to have
\[
|\{ a\}|_\phi + |\{ b\}|_\phi + |(a,b)|_\phi =  |[a,b]|_\phi = |\{a\}|_\phi + |\{ b\}|_\phi + \sum_{j=1}^\infty |R_j|_\phi, 
\]
which shows that the claim is true for $R = (a,b)$.
Similarly for $R = (a,b]$ and $R = [a,b)$. 

Once the reassembling formula for mass is established, we can repeat the argument in Corollary~\ref{monoC} to see that
$I_\phi$ is continuous, i.e., the Stieltjes integral is a preintegral on $S(\R)$. 

{\small
\begin{Remark}
  In contrast to Stieltjes integrals, values on finitely many points are irrelevant in the width integral.
  Based on this fact, it is often convenient to work with open-closed intervals (or closed-open intervals)
  instead of full intervals as witnessed in the Cauchy-Riemann-Darboux approach afterward. 
\end{Remark}}

Next we enlarge a linear lattice $L$ by monotone sequential limits as a preparation to integral extensions.

\begin{Definition}
Given a linear lattice $L$ on $X$, we set 
\begin{align*}
L_\uparrow &= 
\{ f: X \to (-\infty,\infty]; \exists\, 
\text{a sequence} f_n \in L, f_n \uparrow f \},\\
L_\downarrow &= 
\{ f: X \to [-\infty,\infty); \exists\, 
\text{a sequence} f_n \in L, f_n \downarrow f \}
\end{align*}
and $L_\uparrow^+ = \{ f \in L_\uparrow; f \geq 0 \}$.
Functions in $L_\uparrow$ ($L_\downarrow$) are referred to as \textbf{upper} (\textbf{lower}) functions respectively.
\index{lower function}\index{upper function}
\index{+Lupper@$L_\uparrow$ upper extension of $L$}
\index{+Llower@$L_\downarrow$ lower extension of $L$}

Notice that any monotone sequence $(f_n)$ in $L$ has a limit in $L_\updownarrow$,  
where the notation $L_\updownarrow$  is used to stand for $L_\uparrow$ or $L_\downarrow$. 
For $L = S(\R)$, we write $S_\updownarrow(\R)$ instead of $L_\updownarrow$.
\index{+steplower@$S_\downarrow$ lower extension of $S$}
\index{+stepupper@$S_\uparrow$ upper extension of $S$}
\end{Definition}

The following are immediate from these definitions.

\begin{Proposition}\label{immediate}~ 
  \begin{enumerate}
  \item 
$L_\downarrow = - L_\uparrow$ and
$L \subset L_\uparrow \cap L_\downarrow$. 
\item
  $L_\updownarrow$ are semilinear lattices in the sense that,  
for $\alpha, \beta \in \R_+ $ and $f, g \in L_\updownarrow$, we have 
$\alpha f + \beta g, f \vee g, f \wedge g \in L_\updownarrow$.
Consequently $L_\uparrow \cap L_\downarrow$ is a linear lattice. 
\item
 Moreover if $L$ is an algebra (i.e., being closed under multiplication), $L_\uparrow^+ L_\uparrow^+ \subset L_\uparrow^+$ and
 $L_\uparrow \cap L_\downarrow$ is also an algebra.
 Here we adopt the convention that
 \[
   r\cdot\infty =
   \begin{cases}
     \infty &(0 < r \leq \infty)\\
     0 &(r=0)
   \end{cases}
 \]
 in the multiplication of $L_\uparrow^+ L_\uparrow^+$. 
  \end{enumerate}
\end{Proposition}

\begin{proof}
  (i) and (ii) are immediate from the definition.
  
  (iii) Let $f,g \in L_\uparrow^+$ and choose sequences $(f_n)$ and $(g_n)$ in $L^+$ so that $f_n \uparrow f$ and $g_n \uparrow g$.
  Then $(f_ng_n)$ is an increasing sequence in $L^+$ and $f_ng_n \uparrow fg$ under the convention that $0\cdot\infty = 0$, i.e.,
  $fg \in L_\uparrow^+$.
  
  Since $L_\uparrow \cap L_\downarrow$ is a linear lattice, each element in $L_\uparrow \cap L_\downarrow$ is a difference of
  two positive functions in $L_\uparrow \cap L_\downarrow$ and multiplicativity of $L_\uparrow \cap L_\downarrow$ is reduced to
  showing that $fg \in L_\uparrow \cap L_\downarrow$ for positive functions $f,g$ in $L_\uparrow \cap L_\downarrow$.
  As $fg \in L_\uparrow^+$ being already checked, the problem is further reduced to showing that $fg \in L_\downarrow$.
  This time we choose $(f_n)$ and $(g_n)$ in $L$ so that $f_n \downarrow f$ and $g_n \downarrow g$. Note here that $f_n, g_n \in L^+$ because
  $f$ and $g$ are positive functions. Thus $(f_ng_n)$ is a decreasing sequence in $L^+$ and $f_ng_n \downarrow fg$, i.e., $fg \in L_\downarrow$. 
\end{proof}

\begin{Exercise}
  Check the properties in (i) and (ii). 
\end{Exercise}

We say that a function $f:\R \to \R$ is doubly bounded\index{doubly bounded} if it is bounded and has a bounded support. 

\begin{Lemma}~ \label{doubly}
  \begin{enumerate}
  \item
    For $f \in S_\updownarrow(\R)$, $0\diamond f$ is doubly bounded. 
    Consequently functions in $S_\uparrow(\R) \cap S_\downarrow(\R)$ are doubly bounded as well. 
  \item
    A function $f:\R \to [0,\infty)$ belongs to $S_\uparrow(\R)$ 
    if it is continuous on an open interval $(a,b)$ and satisifies $(a,b)f = f$. 
    Here $-\infty\leq a < b \leq \infty$. 
  \end{enumerate}
\end{Lemma}

\begin{proof}
  Non-trivial is (ii). 
  Choose $a_n \downarrow a$ and $b_n \uparrow b$ so that $[a_n,b_n] \subset (a,b)$. 
  By using positivity of $f$ and dividing $(a,b)$ into subintervals finer and finer,
  we can find a double sequence $f_{n,k}$ in $S^+(\R)$ so that $f_{n,k} = [a_n,b_n]f_{n+1,k}$,
  $f_{n,k} \leq f_{n,k+1}$ and $\displaystyle \lim_{k \to \infty} f_{n,k} = [a_n,b_n] f$ 
  thanks to the Darboux approximation (see the end of this section)
  based on uniform continuity of $[a_n,b_n]f$. 

  Now the diagonal sequence $f_n = f_{n,n}$ in $S^+(\R)$ satisfies $f_n \uparrow f$ and we are done.
\end{proof}


\begin{Corollary}\label{regulated}
  If $f$ is a doubly bounded function having finitely many points of discontinuity, then $f \in S_\uparrow(\R) \cap S_\downarrow(\R)$.
 %
\end{Corollary}

\begin{proof}
  By assumption, we can choose a partition $a = t_0 < t_1 < \cdots < t_l = b$ so that $(a,b) f = f$ and the points of discontinuity of $f$
  are contained in $\{ t_0, \dots, t_n\}$. 
  Then, in the expression 
  \[
  (a,b)(f + \| f\|) = \sum_{j=0}^{l-1} (t_j,t_{j+1})(f + \| f\|)  + \sum_{j=0}^l [t_j,t_j] (f(t_j) + \| f\|), 
  \]
  we apply (ii) to see that it belongs to $S_\uparrow(\R)$, whence $f \in S_\uparrow(\R)$ as
  a sum of $(a,b)(f+\| f\|)$ and $-(a,b) \| f\| \in S(\R) \subset S_\uparrow(\R)$.

  Likewise, $-f \in S_\uparrow(\R)$, i.e., $f \in S_\downarrow(\R)$. 
\end{proof}

Lots of functions belong to $S_\uparrow \cup S_\downarrow$ but of course not always.

\begin{Example}~
  \begin{enumerate}
\item
  We see $(0,r)\bigl(\pm 1 + \sin(1/x)\bigr) \in S_\updownarrow(\R)$ for $0 < r \leq \infty$
  and then $(0,r) \sin(1/x) \in S_\uparrow(\R) \cap S_\downarrow(\R)$ for $0 < r < \infty$. 
\item
  Let $C$ be a dense subset of an open (non-empty) interval $(a,b)$ and assume that $(a,b) \setminus C$ is also dense in $(a,b)$.
  Then neither $S_\uparrow(\R)$ nor $S_\downarrow(\R)$ contains $C$ as an indicator function.

  In fact, let $(f_n)$ be a decreasing sequence in $S(\R)$ satisfying $C \leq f_n$ ($n \geq 1$).
  Since $f_n$ is continuous except for finitely many points,
  the density of $C$ in $(a,b)$ is used to see $f_n \geq (a,b) \geq C$ but $C \not = (a,b)$,
  showing that $C \not= \lim f_n$ and hence $C \not\in S_\downarrow(\R)$.

  Likewise, $(a,b) \setminus C \not\in S_\downarrow(\R)$, i.e., $C - (a,b) = - ((a,b) \setminus C) \not\in S_\uparrow(\R)$ and then
  $C = (a,b) + (C - (a,b)) \not\in S_\uparrow(\R)$ in view of $(a,b) \in S(\R)$. 
\item
  Both $\sin x$ and $x/(1+|x|)$ do not belong to $S_\uparrow(\R) \cup S_\downarrow(\R)$
  simply because their positive and negative parts are not doubly bounded.
\end{enumerate}
\end{Example}

\begin{figure}[h]
  \centering
 \includegraphics[width=0.5\textwidth]{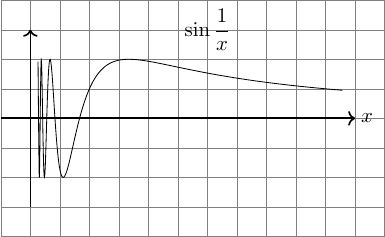}
 \caption{Rapid Oscillation}
\end{figure}

\begin{Exercise}
Any countable dense subset of $(a,b)$ satisfies the condition in (ii). Hint: $(a,b)$ is not countable. 
\end{Exercise}

\begin{Exercise}
  For a sequence $(a_n)$ satisfying $a_n > a_{n+1}$ ($n \geq 1$) and $a_n \downarrow 0$, show that
  a comb function\index{comb function} $\displaystyle \sum_{n \geq 1} [a_{2n},a_{2n-1}]$ is in $S_\uparrow(\R) \cap S_\downarrow(\R)$. 
\end{Exercise}

\begin{Exercise}
  A monotone (increasing or decreasing) function $f: \R \to [-\infty,\infty]$ belongs to $S_\updownarrow(\R)$ if and only if
  $\pm f \geq 0$. Hint: Level approximation in Appendix D. 
\end{Exercise}

{\small
\begin{Remark}
  We notice that so many functions belong to ``$S_\uparrow(\R) - S_\downarrow(\R)$'' but a big issue here is that
  $f_\uparrow - f_\uparrow$ ($f_\updownarrow \in S_\updownarrow(\R)$) is not always well-defined due to the possibility $\infty - \infty$.
  Later we discuss a remedy for this. 
\end{Remark}}

We shall now extend a preintegral $I$ from $L$ to $L_\updownarrow$.

\begin{Lemma}
Let $(f_n)$, $(g_n)$ be increasing sequences in a linear lattice $L$ satisfying the inequality 
\[
\lim_n f_n \leq \lim_n g_n 
\] 
as $(-\infty,\infty]$-valued limit functions
(neither $\lim_n f_n$ nor $\lim_n g_n$ being assumed to be in $L$). Then we have 
\[
\lim_n I(f_n) \leq \lim_n I(g_n). 
\]
\end{Lemma}

\begin{proof}
From the assumption, 
$f_m \leq \lim_{n \to \infty} g_n$ and hence 
$f_m = \lim_{n \to \infty} f_m \wedge g_n$. 
By applying the continuity of $I$ to 
$(f_m - f_m \wedge g_n) \downarrow 0$, we have 
\[
I(f_m) = \lim_{n \to \infty} I(f_m \wedge g_n) 
\leq \lim_{n \to \infty} I(g_n)
\]
and then the limit on $m$ gives the assertion. 
\end{proof}

\begin{Definition}
The previous lemma allows us to define a functional 
$I_\uparrow: L_\uparrow \to (-\infty,\infty]$ by 
\[
I_\uparrow(f) = \lim_{n \to \infty} I(f_n), 
\quad f_n \uparrow f,\ f_n \in L. 
\]
Likewise, 
$I_\downarrow: L_\downarrow \to [-\infty,\infty)$ 
is defined by 
\[
I_\downarrow(f) = \lim_{n \to \infty} I(f_n), 
\quad f_n \downarrow f,\ f_n \in L.
\]
\index{+I-upper@$I_\uparrow$ upper extension of $I$}
\index{+I-lower@$I_\downarrow$ lower extension of $I$}
\end{Definition}

Here are immediate properties of these extensions: 

\begin{Proposition}\label{outer}~ 
  \begin{enumerate}
\item
$I_\downarrow(-f) = - I_\uparrow(f)$ for $f \in L_\uparrow$ 
(recall that $-L_\uparrow = L_\downarrow$). 
\item 
  Functionals $I_\uparrow$ and $I_\downarrow$ coincide on $L_\uparrow \cap L_\downarrow$ and extend $I$, i.e.,
  $I_\uparrow(f) = I_\downarrow(f) \in \R$ for $f \in L_\uparrow \cap L_\downarrow$ and
  $L_\uparrow(f) = I(f) = I_\downarrow(f)$ for $f \in L$.  
\item
Functionals $I_\uparrow$ and $I_\downarrow$ are semilinear, 
i.e., 
for $\alpha, \beta \in \R_+$ and $f, g \in L_\updownarrow$, 
\[
I_\updownarrow (\alpha f + \beta g) 
= \alpha I_\updownarrow(f) + \beta I_\updownarrow(g). 
\]
\item
If $f, g \in L_\updownarrow$ satisfy $f \leq g$, then 
$I_\updownarrow(f) \leq I_\updownarrow(g)$.
\end{enumerate}
Thus $I_\updownarrow$ ($I_\uparrow$ or $I_\downarrow$) is a positive functional on the linear lattice $L_\uparrow \cap L_\downarrow$. 
\end{Proposition}

\begin{proof} We just indicate the coincidence in (ii): 
If $g_n \uparrow f$ and $h_n \downarrow f$ with $g_n, h_n \in L$, then $h_n - g_n \downarrow 0$ and hence
$I(h_n) - I(g_n) \downarrow 0$ by continuity of $I$. Thus
$I_\uparrow(f) = \lim I(g_n) = \lim I(h_n) = I_\downarrow(f)$.
%
\end{proof}

\begin{Exercise}
Check other properties. 
\end{Exercise}

The monotone extensions\index{monotone extension} are now applied to the width integral,
which are conventionally denoted by 
\[
  \int f(t)\, dt \in \R \cup \{ \pm \infty\} \quad
  (f \in S_\updownarrow(\R)). 
\]
Here arises no ambiguity thanks to the coincidence $I_\uparrow = I_\downarrow$ on $S_\uparrow(\R) \cap S_\downarrow(\R)$. 
Note that it gives a positive linear functional on $S_\uparrow(\R) \cap S_\downarrow(\R)$.

Now let $[a,b]$ be a closed interval and $f:\R \to \R$ satisfy
\[
  [a,b]f \in S_\uparrow(\R) \cap S_\downarrow(\R) \iff (a,b)f \in S_\uparrow(\R) \cap S_\downarrow(\R).
  \]
The integral of $[a,b]f$ is called 
the \textbf{definite integral}\index{definite integral} of $f$ on $[a,b]$ and denoted by
\[
 \int_a^b f(t)\, dt.
\]

The definite integral is clearly linear and monotone in $f$, whence it satisfies the integral inequality:
\[
  \left| \int_a^b f(t)\, dt \right|
  \leq \int_a^b |f(t)|\, dt \leq (b-a) \| f\|_{[a,b]}.
\]
Consequently, if a sequence $(f_n)$ and $f$ satisfy $[a,b] f_n \in S_\uparrow(\R) \cap S_\downarrow(\R)$,
$[a,b]f \in S_\uparrow(\R) \cap S_\downarrow(\R)$ and $[a,b]f_n \to [a,b]f$ uniformly on $[a,b]$, then
\[
  \lim_{n \to \infty} \int_a^b f_n(t)\, dt = \int_a^b f(t)\, dt. 
\]

In the definition of definite integral, we may use other types of intervals, say $(a,b]$, as well because functions supported by finite sets
belong to $S_\uparrow(\R) \cap S_\downarrow(\R)$ with their integrals equal to zero. 

The definite integral is additive on supporting intervals: 
If $a \leq c \leq b$, $[a,b]f \in S_\uparrow(\R) \cap S_\downarrow(\R)$ if and only if
$[a,c]f$ and $[c,b]f$ belong to $S_\uparrow(\R) \cap S_\downarrow(\R)$.
Moreover, if this is the case, we have 
\[
  \int_a^b f(t)\, dt = \int_a^c f(t)\, dt + \int_c^b f(t)\, dt. 
\]
In accordance with this additivity, it is then customary to write 
\[
  \int_b^a f(t)\, dt = - \int_a^b f(t)\, dt.
\]

\begin{Example}~
  \begin{enumerate}
\item
  Any function $f$ which is continuous on $[a,b]$ admits the definite integral $\int_a^b f(t)\, dt$ by Corollary~\ref{regulated}.
\item
  For $r \in \R$, the translated function $g(t) = f(t-r)$ satisfies $[a+r,b+r]g \in S_\uparrow(\R) \cap S_\downarrow(\R)$ if and only if
  $[a,b]f \in S_\uparrow(\R) \cap S_\downarrow(\R)$, and in this case
  \[
    \int_a^b f(t)\, dt = \int_{a+r}^{b+r} f(t-r)\, dt.
  \]
\end{enumerate}
\end{Example}

\begin{Example}
  Let $f(t) = \sin(1/t)$ for $t \not= 0$ and assign any value at $t=0$. Then $[a,b]f \in S_\uparrow(\R) \cap S_\downarrow(\R)$ for
  every bounded $[a,b] \subset \R$ by Corollary~\ref{regulated} and the definite integral $\int_a^b f(t)\, dt$ is well-defined. 
\end{Example}

\begin{Exercise}\label{scaled-integral}
  For $r>0$, $[ra,rb]f \in S_\uparrow(\R) \cap S_\downarrow(\R)$ if and only if
  the scaled function $g(t) = f(rt)$ satisfies $[a,b]g \in S_\uparrow(\R) \cap S_\downarrow(\R)$. Moreover, if this is the case,
  \[
    \int_a^b f(rt)\, dt = \frac{1}{r} \int_{ra}^{rb} f(t)\, dt.
  \]
 \end{Exercise}

Next an \textbf{indefinite integral}\index{indefinite integral} of $f$ is a function of $x$ defined by
\[
  \int_a^x f(t)\,dt
\]
with $a$ a preassigned point and $x \in \R$ satisfying $[a,x]f \in S_\uparrow(\R) \cap S_\downarrow(\R)$.
The difference of indefinite integrals is therefore a constant function
and indefinite integrals of $f$ are determined up to additive constants.

\begin{Example}
  For a function $f \in S_\uparrow(\R) \cap S_\downarrow(\R)$, the indefinite integral
  $\displaystyle \int_a^x f(t)\, dt$ is everywhere defined for any $a \in \R$ and is locally constant outside the support $[f]$ of $f$.
  In particular, the indefinite integral is constant for a sufficiently large $|x|$. 
\end{Example}

The following, known as the fundamental theorem\index{fundamental theorem in calculus} in calculus, is literally of fundamental importance.

\begin{Theorem}
An indefinite integral is a continuous function and, if $f(t)$ is continuous at $t=c$, it is differentiable at $x=c$ in such a way that 
\[
  \frac{d}{dx} \int_a^x f(t)\, dt = f(c). 
\]
\end{Theorem}

\begin{proof} Continuity of an indefinite integral of $f$ follows from the integral inequality
  \[
    \left| \int_x^y f(t)\, dt \right| \leq |x-y| \| f\|_{[x,y]}
    \]
  in view of local boundedness of $f$.

  For $\delta>0$, if $x$ satisfies $|x-c| \leq \delta$, 
  \begin{align*}
    \left| \frac{1}{x-c} \left( \int_a^x f(t)\, dt - \int_a^c f(t)\, dt\right) - f(c) \right|
    &= \left| \frac{1}{x-c} \int_c^x (f(t) - f(c))\, dt \right|\\
  &\leq \| f - f(c)\|_{[c-\delta,c+\delta]},  
  \end{align*}
which converges to $0$ as $\delta \to +0$ by continuity of $f(x)$ at $x=c$.   
\end{proof}

\begin{Corollary}
  If $f$ is continuous on an open interval $(a,b)$, it admits a primitive function $F$ in such a way that
  \[
    \int_x^y f(t)\, dt = F(y) - F(x) \equiv \bigl[F(t)\bigr]_x^y 
  \]
  for any $[x,y] \subset (a,b)$.
  
  Recall that a \textbf{primitive function}\index{primitive function} of a function $f$ defined on an open interval $(a,b)$
  is a differentiable function $F$ on $(a,b)$ satisfying $F' = f$.
  Also recall that primitive functions of $f$ are unique up to additive constants. 
\end{Corollary}

\begin{proof}
As functions of $y$ ($x$ being fixed), both sides are primitive functions of $f$ and coincide at $y=x$. 
\end{proof}

\begin{Example}
 Given $\alpha \geq 0$, consider a function $f_\alpha(t)$ of $t \in \R$ defined by 
  \[
    f_\alpha(t) =
    \begin{cases}
      t^{\alpha} &(t>0),\\
      0 &(t \leq 0), 
    \end{cases}
  \]
  which is continuous for $\alpha>0$ but has discontinuity at $t=0$ for $\alpha=0$. In either case, indefinite integrals
  are defined everywhere and given by continuous functions 
  \[
    \int_0^x f_\alpha(t)\, dt =
    \begin{cases}
      x^{\alpha + 1}/(\alpha+1) &(x>0),\\
      0 &(x \leq 0), 
    \end{cases}
  \]
  which are differentiable and give primitive functions of $f_\alpha$ for $\alpha>0$ but not for $\alpha=0$
  (no primitive function of $f_0$ exists). 
\end{Example}


Let a function $f:(a,b) \to \R$ satisfy $[x,y]f \in S_\uparrow(\R) \cap S_\downarrow(\R)$ for $a < x \leq y < b$. 
An \textbf{improper integral}\index{improper integral} of $f$ is defined to be 
\[
  \int_a^b f(t)\, dt = \lim_{(x,y) \to (a,b)} \int_x^y f(t)\, dt = \lim_{x \to a+0} \int_x^c f(t)\, dt + \lim_{y \to b-0} \int_c^y f(t)\, dt
\]
if limits exist. When $f \in S_\uparrow(\R) \cap S_\downarrow(\R)$ and $(a,b)$ is bounded,
it is reduced to the definite integral $\int_a^b f(t)\, dt$.
Note also that, when $f$ is positive, it is improperly integrable if there exists an improperly integrable function
$g$ satisfying $f \leq g$ simply because $\int_x^y f(t)\, dt \geq$ is increasing as $x \downarrow a$, $y \uparrow b$
and bounded by $\int_a^b g(t)\, dt$. 

Improperly integrable functions constitute a linear space with the improper integral giving a positive functional
but improperly integrable functions do not form a lattice. 

Related to this fact, we say that a function $f: (a,b) \to \R$ is absolutely integrable 
 if $[x,y]f \in S_\uparrow(\R) \cap S_\downarrow(\R)$ for $[x,y] \subset (a,b)$ and
 $|f|$ is improperly integrable.
 In that case, $f_\pm = 0\vee(\pm f)$ as well as $f = f_+ - f_-$ are improperly integrable and satisfy the integral inequality 
\[
  \left| \int_a^b f(x)\, dx \right| \leq \int_a^b |f(x)|\, dx. 
\]
The improper integral of $f$ is said to be \textbf{absolutely convergent}\index{absolutely convergent}. 
An improper integral is by definition \textbf{conditionally convergent}\index{conditionally convergent} if it is not absolutely convergent.

Later we shall see that absolutely convergent integrals are \textit{properly} extended to multiple integrals. 

\begin{Proposition}[Frullani integral]\label{frullani}\index{Frullani integral}
  Let a function $f:(0,\infty) \to \R$ satisfy $[x,y]f \in S_\uparrow(\R) \cap S_\downarrow(\R)$ for $0 < x \leq y < \infty$ and
  assume that $f(0) = \lim_{t \to +0} f(t)$ and $f(\infty) = \lim_{t \to \infty} f(t)$ exist.
  Then, for $0 < a < b$, the function $\frac{f(bt) - f(at)}{t}$ is improperly integrable on $(0,\infty)$ and
  \[
    \int_0^\infty \frac{f(bt) - f(at)}{t}\, dt = \Bigl(f(\infty) - f(0)\Bigr)\log \frac{b}{a}.
  \]
  Note here that $f(at)/t$ ($x \leq t \leq y$) belongs to $S_\uparrow(\R) \cap S_\downarrow(\R)$. 
\end{Proposition} 

\begin{proof}
  Take $x>0$ small and $y \geq x$ large. Then from the scaling invariance of $dt/t$ (Exercise~\ref{scaled-integral}), we have
  \begin{align*}
    \int_x^y \frac{f(bt) - f(at)}{t}\, dt &= \int_{bx}^{by} \frac{f(t)}{t}\, dt - \int_{ax}^{ay} \frac{f(t)}{t}\, dt\\
                                          &= \int_{ay}^{by} \frac{f(t)}{t}\,dt - \int_{ax}^{bx} \frac{f(t)}{t}\, dt\\
    &= \int_{a}^{b} \frac{f(ty)}{t}\,dt - \int_{a}^{b} \frac{f(tx)}{t}\, dt. 
  \end{align*}
  Since $\displaystyle \lim_{y \to \infty}f(ty) = f(\infty)$ and $\displaystyle \lim_{x \to +0} f(tx) = f(0)$ uniformly in $t \in [a,b]$, 
  \begin{align*}
    \lim_{(x,y) \to (0,\infty)} \int_x^y \frac{f(bt) - f(at)}{t}\, dt
    &= \int_a^b \frac{f(\infty)}{t}\, dt - \int_a^b \frac{f(0)}{t}\, dt\\
    &= \Bigl(f(\infty) - f(0)\Bigr)\log \frac{b}{a}.
  \end{align*}
\end{proof}

Here is a practical formula to compute improper integrals (including proper ones):
\begin{Theorem}
Let $f$ be coninuous on $(a,b)$ with $F$ its primitive function.
Then $f$ is improperly integrable if and only if $F(a+0) = \lim_{t \to a+0} F(t)$ and $F(b-0) = \lim_{t \to b-0} F(t)$ exist.
Moreover, if this is the case, we have 
\[
  \int_a^b f(t)\, dt = F(b-0) - F(a+0).
\]
\end{Theorem}

\begin{Example} For $r>0$, 
  \[
    \int_1^\infty \frac{1}{t^r}\,dt =
    \begin{cases}
      1/(r-1) &\text{if $r > 1$, }\\
      \infty &\text{otherwise.}
    \end{cases}
  \]
  \[
    \int_0^1 \frac{1}{t^r}\, dt =
    \begin{cases}
      1/(1 - r) &\text{if $r< 1$, }\\
      \infty &\text{otherwise.}
    \end{cases}
  \]
  \[
    \int_0^\infty t^ne^{-rt}\, dt = \frac{n!}{r^{n+1}}\quad (r>0, n = 0, 1, 2, \cdots).
  \]  
\end{Example}

For the existence of absolutely convergent improper integrals, the following gives a useful criterion.

\begin{Proposition}\label{DC} 
  If continuous functions $f$ and $\varphi$ defined on an open interval $(a,b)$ satisfy $|f| \leq \varphi$ with 
  the integral $\int_a^b \varphi(t)\, dt$ convergent ($a$ and $b$ can be $\pm\infty$),
  then $\int_a^b f(t)\, dt$ is absolutely convergent and satisfies
  \[
    \left| \int_a^b f(t)\, dt \right| \leq \int_a^b \varphi(t)\, dt.
  \]
\end{Proposition}

\begin{proof}
  By Lemma~\ref{doubly} (ii), $[x,y]f \in S_\uparrow(\R) \cap S_\downarrow(\R)$ and the definite integral $\int_x^y |f(t)|\, dt$
  has a meaning so that it increases when $x\downarrow a$ and $y \uparrow b$, which is combined with
  \[
    \int_x^y |f(t)|\, dt \leq \int_x^y \varphi(t)\,dt \leq \int_a^b \varphi(t)\, dt < \infty
  \]
  to see that the improper integral of $|f|$ as well as $f$ exists and satisfies
  \[
    \left| \int_a^b f(t)\, dt \right| \leq \int_a^b |f(t)|\, dt \leq \int_a^b \varphi(t)\, dt.
  \]
\end{proof}

\begin{figure}[h]
  \centering
 \includegraphics[width=0.6\textwidth]{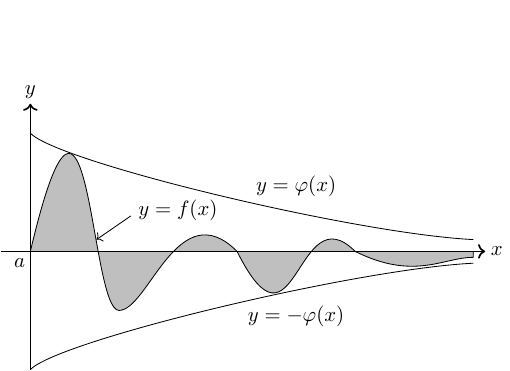}
 \caption{Dominated Integral}
\end{figure}

\begin{Example}
  Primitive functions of $\displaystyle \sin(1/x)$ on $\pm(0,\infty)$ are continuous at $x = \pm 0$.
  In fact, on $\pm(0,r)$, $|\sin(1/t)| \leq 1$ and we have
  $|\int_0^x \sin(1/t)\, dt| \leq |x| \to 0$ as $x \to 0$. 
\end{Example}

\begin{Example} The improper integral (called \textbf{gamma function})\index{gamma function}  
  \[
    \Gamma(s) = \int_0^\infty t^{s-1} e^{-t}\, dt
  \]
  exists for $s>0$.
  Use $t^{s-1} e^{-s} \leq (0,1]t^{s-1} + (1,\infty) M_s e^{-t/2}$ with
  $M_s = \sup\{ t^{s-1} e^{-t/2}; t \geq 1\} < \infty$. 
%
\end{Example}

\begin{Exercise}
  Relate the Gaussian integral
  \[
    \int_0^\infty t^n e^{-t^2}\, dt \quad(n=0,1,2,\cdots)
  \]
  to the gamma function. 
\end{Exercise}

\begin{Example}\label{sinc}
  As a typical example of conditionally convergent integrals, we pick up
  $\displaystyle \int_0^\infty \frac{\sin x}{x}\, dx$. Here the integrand is continuous (even analytic) at $x=0$ and
  the integral is improper only at $x=\infty$.
  The integral value turns out to be $\pi/2$ as seen later with the help of repeated integrals, complex analysis or Fourier analysis. 
  
  To see the convergence, we use integration by parts to have
  \[
    \int_0^{2a} \frac{\sin x}{x}\, dx = \int_0^a \frac{\sin(2x)}{x}\, dx
    = \int_0^a \left( \frac{\sin x}{x} \right)^2\, dx 
    + \left[ \frac{(\sin x)^2}{x} \right]_0^a, 
  \]
  where the last expression approaches the absolutely convergent integral 
  $\int_0^\infty (\sin x/x)^2\, dx$ (Proposition~\ref{DC}) as $a \to \infty$. 
  

  It is, however, not absolutely convergent because
  \begin{align*}
    \int_0^\infty \frac{|\sin x|}{x}\, dx &= \sum_{n=1}^\infty \int_{\pi (n-1)}^{\pi n} \frac{|\sin x|}{x}\, dx\\
                                          &\geq \sum_{n=1}^\infty \frac{1}{\pi n} \int_{\pi (n-1)}^{\pi n} |\sin x|\, dx
                                            = \sum_{n=1}^\infty \frac{2}{\pi n} = \infty. 
  \end{align*}  
\end{Example}

\begin{figure}[h]
  \centering
 \includegraphics[width=0.5\textwidth]{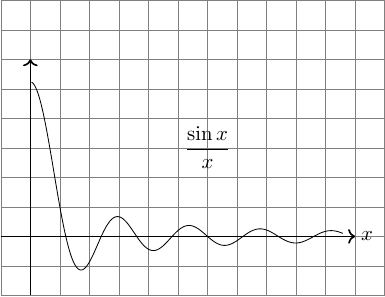}
 \caption{Sinc Function}
\end{figure}
\index{sinc function}

\begin{Exercise}
  Compute $\displaystyle\int_0^\infty \frac{(\sin x)^{2n}}{x}\, dx$ for $n = 1,2,\dots$.
%
\end{Exercise}

\begin{Exercise}\label{fi} 
  Show that the Fresnel integrals\index{Fresnel integral} 
  \[
    \int_0^\infty \cos t^2\, dt, \quad
    \int_0^\infty \sin t^2\, dt
  \]
  have meanings as improper integrals. Hint: Change the integral variable to $t = \sqrt{x}$ and then try the same trick as in the above example.
  We shall show afterwards that their values\footnote{A quick way to compute the value is to change the variable to $t = s e^{i\pi/4}$ and apply
  Cauchy's integral theorem.}
  are $\sqrt{\pi/8}$ by computing a double integral relative to polar coordinates. 
\end{Exercise}

\begin{figure}[h]
  \centering
 \includegraphics[width=0.5\textwidth]{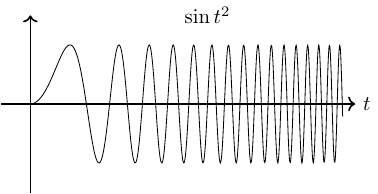}
 \caption{Fresnel Integral}
\end{figure}

Here are more amusing examples of improper integrals.

\begin{Example}
  $\displaystyle \int_0^\infty \frac{\sin^3x}{x^2}\, dx = \frac{3}{4} \log 3$.

  This improper integral is absolutely convergent and the expression
  \[
    \frac{1}{4} \int_0^\infty \frac{3\sin x - \sin(3x)}{x^2}\, dx
  \]
  allows us to apply the Frullani integral for $f(t) = (\sin t)/t$ with $a=1$ and $b=3$ to get the value. 
\end{Example}

\begin{Example}[Euler]\index{Euler} 
  $I = \displaystyle \int_0^{\pi/2} \log(\sin x)\, dx = -\frac{\pi}{2} \log 2$.

  First observe that the integral is improper at the boundary $x=0$ and is absolutely convergent. 
  
  From the translational invariance 
  $\displaystyle I = \int_{\pi/2}^\pi\log(\sin x)\, dx$ and the reflection invariance
  $I = \displaystyle \int_0^{\pi/2} \log(\cos x)\,dx$,  
  \begin{align*}
    I &= \frac{1}{2}\int_0^\pi \log(\sin x)\,dx = \int_0^{\pi/2} \log(\sin(2x))\, dx\ (\text{scale is modified by $2$})\\
      &= \int_0^{\pi/2} \log 2\, dx + \int_0^{\pi/2} \log(\sin x)\, dx + \int_0^{\pi/2} \log(\cos x)\, dx\\
      &= \frac{\pi}{2}\log 2 + 2I. 
  \end{align*} 
\end{Example}

\begin{Exercise}
 With the help of $\sin x \geq 2x/\pi$ ($0 \leq x \leq \pi/2$) show the absolute convergence of $\int_0^{\pi/2} \log(\sin x)\, dx$.
\end{Exercise}

\begin{Theorem}[Continuity in Laplace Transform] \label{laplace} 
  Let $f:(0,\infty) \to \R$ be
  a continuous\footnote{Continuity of $f$ is relaxed to the property that $[x,y]f \in S_\uparrow \cap S_\downarrow$ for $[x,y] \subset (0,\infty)$
    (see Proposition~\ref{by-parts}).} function for which $\int_0^\infty f(t)\, dt$ exists as an improper integral.
  Then $e^{-rt} f(t)$ ($r>0$) is an improperly integrable function of $t>0$ and the \textbf{Laplace transform}\index{Laplace transform} 
  $\displaystyle \int_0^\infty e^{-rt} f(t)\, dt$ 
  of $f$ is a continuous function of $r>0$ satisfying
  \[
    \lim_{r \to +0} \int_0^\infty e^{-rt}f(t)\, dt = \int_0^\infty f(t)\, dt.
  \]
\end{Theorem}

\begin{proof}
  We first consider the case that $f$ is continous on $(0,\infty)$. 
  Let $\displaystyle F(x) = -\int_x^\infty f(t)\, dt$ ($x > 0$) be a primitive function of $f$
  satisfying $\displaystyle \lim_{x \to \infty} F(x) = 0$ and $\displaystyle \lim_{x \to +0} F(x) = - \int_0^\infty f(t)\, dt$
  at boundaries.
  Then $(F(t) e^{-rt})' = f(t) e^{-rt} - r F(t) e^{-rt}$, which is integrated to get 
  \[
    \int_x^y e^{-rt} f(t)\,dt = e^{-rx} \int_x^\infty f(t)\, dt - e^{-ry} \int_y^\infty f(t)\, dt
    + r \int_x^y e^{-rt}F(t)\, dt.
  \]
  Since the last integrand is absolutely integrable on $(0,\infty)$ by Proposition~\ref{DC},
  we can take the limit $x \to +0$, $y \to \infty$ to have
  \[
    \int_0^\infty e^{-rt} f(t)\,dt - \int_0^\infty f(t)\, dt = r \int_0^\infty e^{-rt}F(t)\, dt.
  \]
  
  To see the parametric behavior of the right hand side as $r \to +0$,
  we split the integral domain at some $R>0$ and estimate partial terms by 
  \begin{align*}
    r \int_0^R e^{-rt}|F(t)|\, dt &\leq rR \sup_{t > 0} |F(t)|,\\ 
    r \int_R^\infty e^{-rt}|F(t)|\, dt &\leq e^{-rR}\sup_{t \geq R} |F(t)| \leq \sup_{t \geq R} |F(t)|. 
  \end{align*}
  We first take $R$ large enough so that the second term is small and then choose $r>0$ small enough so that $rR$ is small. 
  In total, the right hand side turns out to converge to $0$ as $r \to +0$.

  For the parametric continuity, we show that $\displaystyle \int_0^\infty e^{-rt} F(t)\, dt$ is continuous in $r>0$.
  To see this, let $r,s \geq \delta > 0$ and estimate an absolutely convergent integral
  $\displaystyle  \int_0^\infty (e^{-rt} - e^{-st}) F(t)\, dt$ by 
  \begin{align*}
    \int_0^\infty |e^{-rt} - e^{-st}| |F(t)|\, dt
    &\leq \sup_{x>0} |F(x)| \int_0^\infty |e^{-rt} - e^{-st}|\, dt\\
    &= \sup_{x>0} |F(x)| \int_0^\infty dt\, \left|\int_r^st e^{-ut}\, du\right|\\
    &\leq \sup_{x>0} |F(x)| \int_0^\infty dt\, te^{-\delta t} |r-s|\\
    &= \sup_{x>0} |F(x)| \frac{|r-s|}{\delta^2}. 
  \end{align*}
\end{proof}



\noindent\textbf{Darboux Approximation}\index{Darboux approximation} 
Originally integral was invented as a limit-sum of infinitesimals, which was necessary and useful in mathematical modelling of differential
objects. We shall here describe our definite integrals according to historical developments due to Cauchy, Riemann and Darboux.
Instead of somewhat mysterious notion of infinitesimals,
we work with partitioning of an interval and make the size of interval parts smaller and smaller.

Consider a continuous function $f$ defined on a bounded closed interval $[a,b]$ and regard it as a function on $\R$
by zero-extension.

For a partition\index{partition} $\Delta = \{ a=x_0 < x_1 < \dots < x_m = b\}$ of $[a,b]$, 
and a choice $\xi = (\xi_i)$ of sample points $\xi_i$ from open subintervals $(x_{i-1},x_i)$,
let
\begin{align*}
  f^\Delta &= \sum_{i=1}^m (x_{i-1},x_i) \sup f((x_{i-1},x_i)) + \sum_{i=0}^m [x_i,x_i] f(x_i),\\
  f_\Delta &= \sum_{i=1}^m (x_{i-1},x_i) \inf f((x_{i-1},x_i)) + \sum_{i=0}^m [x_i,x_i] f(x_i)
\end{align*}
and
\[
  f_{\Delta,\xi} = \sum_{i=1}^m (x_{i-1},x_i) f(\xi_i) + \sum_{i=0}^m [x_i,x_i] f(x_i)
\]
so that $f_\Delta \leq f \leq f^\Delta$ and $f_\Delta \leq f_{\Delta,\xi} \leq f^\Delta$.

Note that $f_{\Delta,\xi}$ is linear in $f$, whereas not for $f_\Delta$ and $f^\Delta$,
but these behave simply under a finer partition $\Delta' \supset \Delta$; $f_\Delta \leq f_{\Delta'} \leq f^{\Delta'} \leq f^\Delta$.

\begin{figure}[h]
  \centering
 \includegraphics[width=0.5\textwidth]{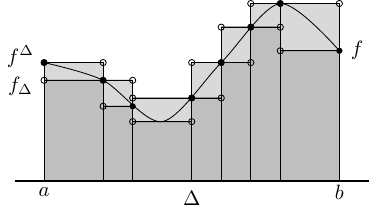}
 \caption{Darboux Approximation}
\end{figure}

Let $|\Delta| = \max\{ x_1-x_0,\dots,x_m-x_{m-1}\}$ be the mesh size of $\Delta$.
By uniform continuity of $f$ on $[a,b]$, 
\[
  C_f(\delta) = \sup\{ |f(s) - f(t)|; |s-t| \leq \delta\}
\]
descreases to $0$ as $\delta\downarrow 0$ (Theorem~\ref{UC}). 
On the other hand,
\[
  f^\Delta(x) - f_\Delta(x) = \sup\{ |f(s) - f(t)|; s,t \in (x_{i-1},x_i)\}
\]
for $x \in (x_{i-1},x_i)$ and $f^\Delta(x_i) = f(x_i) = f_\Delta(x_i)$ ($0 \leq i \leq m$)
shows that $0 \leq f^\Delta - f_\Delta \leq [a,b] C_f(|\Delta|)$, which is combined with 
\[
  |f^\Delta - f_{\Delta,\xi}| + | f_{\Delta,\xi} - f_\Delta| = f^\Delta - f_\Delta
  =  |f^\Delta - f| + |f - f_\Delta| 
\]
to obtain
\[
  |f_{\Delta,\xi} - f| \leq | f^\Delta - f_{\Delta,\xi}| + |f^\Delta - f|
  \leq 2(f^\Delta - f_\Delta) \leq [a,b] (2C_f(|\Delta|)). 
\]


Consequently, for an increasing sequence $\Delta_1 \subset \Delta_2 \subset \cdots$ of partitions
satisfying $|\Delta_n| \to 0$ ($n \to \infty$),
we see that $f_{\Delta_n} \uparrow f$ and $f^{\Delta_n} \downarrow f$,
whence $f \in S_\uparrow(\R) \cap S_\downarrow(\R)$. 

Moreover, from $0 \leq I(f^{\Delta}) - I(f_{\Delta}) \leq (b-a) C_f(|\Delta|) \to 0$ ($|\Delta| \to 0$), 
\[
  I_\uparrow(f) = \lim_{|\Delta| \to 0} I(f_\Delta) = \lim_{|\Delta| \to 0} I(f_{\Delta,\xi}) = \lim_{|\Delta| \to 0} I(f^\Delta)
= I_\downarrow(f). 
\]

The discussion so far is now summarized as follows:  

\begin{Theorem} Let $[a,b]$ be a bounded closed interval. 
  Then $C([a,b]) \subset S_\uparrow(\R) \cap S_\downarrow(\R)$ and, for $f \in C([a,b])$, 
  both $f^\Delta$ and $f_\Delta$ converge uniformly to $[a,b]f$ when $|\Delta| \to 0$ and 
  \[
    \int_a^b f(x)\, dx
    = \lim_{|\Delta| \to 0} I(f_{\Delta,\xi}), 
  \]
  i.e., given $\epsilon>0$, we can find $\delta>0$ so that $|\Delta| \leq \delta$ implies $\| f^\Delta - f_\Delta\|_{[a,b]} \leq \epsilon$
  and 
  \[
    \left| \int_a^b f(x)\, dx - I(f_{\Delta,\epsilon})\right| \leq \epsilon
  \]
 for any choice $\xi$ of sample points in $\Delta$. 
\end{Theorem}





\section{Multiple and Repeated Integrals}\label{MRI}
We now develop multi-dimensional integrals as an analogue of the single-variable case:
A \textbf{rectangle}\index{rectangle} is a product set in $\R^d$ of bounded intervals such as
$[a,b] = [a_1,b_1]\times\dots\times [a_d,b_d]$, $(a,b] = (a_1,b_1]\times\dots\times (a_d,b_d]$ and so on.
Note that there are $4^d$ choices of end points. 

A \textbf{step function}\index{step function} on $\R^d$ is defined to be a linear combination of rectangles and 
the set $S(\R^d)$ of step functions on $\R^d$ is an algebra-lattice. 
\index{+step-p@$S(\R^n)$ multiple step function}

As in the one-dimensional case, we write $L_\updownarrow = S_\updownarrow(\R^d)$ for $L = S(\R^d)$,
which are semilinear lattices and fulfill the following properties.
\index{+stepupper@$S_\uparrow$ upper extension of $S$}
\index{+steplower@$S_\downarrow$ lower extension of $S$}

\begin{Proposition}\label{SUD}~ 
  \begin{enumerate}
  \item
    For $f \in S_\updownarrow(\R^d)$, $0\diamond f$ is doubly bounded (i.e., bounded and of bounded support).
    Consequently functions in $S_\uparrow(\R^d) \cap S_\downarrow(\R^d)$ are doubly bounded as well.
   \item
    $S_\uparrow(\R^d) \cap S_\downarrow(\R^d)$ is an algebra-lattice (Proposition~\ref{immediate} (iii)). 
   \end{enumerate}
%
%
\end{Proposition}

\begin{Exercise}
Check these. 
\end{Exercise}

Given a rectangle $R$, its \textbf{volume}\index{volume} $|R|$ \index{+volume@$\lvert R \rvert$ volume of $R$} is the product of relevant widths;
$|(a,b]| = (b_1-a_1) \cdots (b_d-a_d)$ for example.
The volume function is then linearly extended to a positive functional $I$ of $S(\R^d)$
(called the \textbf{volume integral}\index{volume integral}),
which is also denoted by $I(f) = \int f(x)\, dx$ or simply $\int f$ to suppress integral variables.

The value $I(f)$ is also referred to as the \textbf{multiple integral}\index{multiple integral} of $f$ based on the fact that 
the following \textbf{repeated integral}\index{repeated integral} formula holds. 
\[
  \int f(x)\, dx = \int\cdots \int f(x_1,\dots,x_d)dx_1\cdots dx_d. 
\]
Here the order of repetitions of single-variable integrals is irrelevant, i.e., the multiple integral is invariant under
permutations of variables because the volume function is invariant under permutations.

More explicitly we have the following.

\begin{Proposition}
  Let $f$ be a step function on $\R^d$.
  \begin{enumerate}
  \item
    Chosen $1 \leq j \leq d$, $f(x_1,\dots,x_j,\dots,x_d)$ belongs to $S(\R)$ as a function of $x_j$
    (regarding $(x_1,\dots, x_{j-1},x_{j+1},\dots,x_d)$ as parameters) and its width integral
    \[
      \int f(x_1,\dots,x_j,\dots,x_d)\, dx_j
    \]
    belongs to $S(\R^{d-1})$ as a function of $(x_1,\dots, x_{j-1},x_{j+1},\dots,x_d)$.
  \item
    When two indices $1 \leq i < j \leq d$ are chosen, repeated integrals relative ro $x_i$ and $x_j$ coincide:
    \begin{multline*}
      \int \left(\int f(x_1,\dots,x_i,\dots,x_j,\dots,x_d)\,dx_i\right) dx_j\\ = 
      \int \left( \int f(x_1,\dots,x_i,\dots,x_j,\dots,x_d)\,dx_j \right)dx_i. 
    \end{multline*}
  \item
    The volume integral $I_d$ on $S(\R^d)$ is realized by a $d$-times repetition of width integrals. 
  \end{enumerate}
\end{Proposition}

\begin{proof}
  By linearity of integrals, it suffices to check assertions when $f$ is (the indicator) of a rectangle $R$,
  say $(a,b] = (a_1,b_1]\times \dots \times (a_d,b_d]$.

  (i) From definition, we have 
  \[
    \int f(x_1,\dots,x_j,\dots,x_d)\, dx_j = \begin{cases} b_j - a_j &\text{if $x_i \in (a_i,b_i]$ for $i \not=j$,}\\
                                               0 &\text{otherwise,}
                                             \end{cases} 
  \]
  which is (the indicator) of the rectangle
  $(a_1,b_1]\times \dots \times (a_{j-1},b_{b-1}]\times (a_{j+1},b_{j+1}] \times \dots \times (a_d,b_d]$ multiplied by $(b_j-a_j)$. 

  (ii) By iterating the computation in (i), we see that both of repeated integrals result in
  a rectangle 
  \begin{multline*}
    (a_1,b_1]\times \dots \times (a_{i-1},b_{i-1}]\times (a_{i+1},b_{i+1}] \times\\
    \dots \times (a_{j-1},b_{j-1}]\times (a_{j+1},b_{j+1}] \times \dots
    \times (a_d,b_d]
  \end{multline*}
  multiplied by $ (b_i-a_i)(b_j-a_j)$.

  (iii) If the process in (ii) is iterated by $d$-times, we arrive at
  \[
    \int\dots\int f(x_1,\dots,x_d)\, dx_1\dots dx_d = (b_1-a_1)\cdots (b_d-a_d),
  \]
  which is nothing but the volume of $R$. 
\end{proof}

\begin{Corollary}
The volume integral is well-defined and the multiple integral doe not depend on the repetition process. 
\end{Corollary}

Algebraically\footnote{See Appendix~\ref{tensor} for more on tensor products.} $S(\R^d)$ is identified with $S(\R)\otimes \dots \otimes S(\R)$
so that
\[
  (f_1\otimes \dots \otimes f_d)(x_1,\dots,x_d) = f_1(x_1) \dots f_d(x_d)
\]
for $f_i \in S(\R)$. 
Particularly a rectangle $R = R_1\times \dots \times R_d$ with $R_i$ intervals is identified with 
$R_1\otimes \dots \otimes R_d$ as a function on $\R^d$ 
and the volume integral $I_d$ on $S(\R^d)$ is nothing but the tensor product $I_1^{\otimes d} = I_1\otimes \dots \otimes I_1$
of the width integral $I_1$ on $S(\R)$.
Thus, if we denote by $I^{(j)}: S(\R^d) \to S(\R^{d-1})$
the partial width integral $1\otimes \dots \otimes I_1\otimes 1\otimes \dots \otimes 1$ on the $j$-th variable,
then $I_d$ is realized as a repetition of $I^{(j)}$ by $d$-times.
Note that, in the integration notation, $I^{(j)}(f)$ takes the form
\[
  I^{(j)}(f)(x_1,\dots,x_{j-1},x_{j+1},\dots,x_d)
  = \int f(x_1,\dots,x_j,\dots,x_d)\, dx_j. 
\]


Multiple integrals are then continuous relative to monotone convergence because each partial integral $I^{(j)}$ is continuous
(or one can repeat the one-dimensional argument based on a reassembling lemma for countably many rectangles). 

The volume integral $I$ on $S(\R^d)$
is thus a preintegral and we can talk about its extension $I_\updownarrow$ to $S_\updownarrow(\R^d)$,
which are permutation-invariant \index{permutation-invariance} on variables and
also denoted by $I_\updownarrow(f) = \int f(x)\, dx \in \pm(-\infty,\infty]$ for $f \in S_\updownarrow(\R^d)$
as in the case $d=1$.

We can even apply the same argument in the monotone extension of $I$ to see that 
each partial integral $I^{(j)}$ is monotone-continuously extended to
$S_\updownarrow(\R^d) \to S_\updownarrow(\R^{d-1})$ and obtain the following. 

\begin{Proposition}\label{RULI} 
  The repeated integral formula is valid even for $f \in S_\updownarrow(\R^d)$, 
  where each single-variable integral in repetition process is realized by $I_\updownarrow$ on $S_\updownarrow(\R)$. 
\end{Proposition}



For (bounded) continuous functions of bounded support, we can describe their integrals also by the Darboux approximation.
\index{Darboux approximation}
Consider a multiple partition $\Delta = \Delta_1\times\dots\times \Delta_d$ 
\index{+partition@{$\Delta$} multiple partition}
of a closed rectangle $[a,b]$ with
$\sharp(\Delta_j) = m_j+1$ so that $\Delta_j$ has $m_j+1$ points and $m_j$ open intervals as parts.
The multiple partition $\Delta$ therefore gives a decomposition of $[a,b]$ into mutually disjoint
$\prod_{j=1}^d (2m_j+1)$ parts. Note that there are $\prod_{j=1}^d m_j$ open rectangles among them.
Let $(R_i)$ be the totality of these rectangular parts in $\Delta$ so that $[a,b] = \bigsqcup R_i$ and $(R_i)$ is linearly independent in $S(\R^d)$.

\begin{figure}[h]
  \centering
 \includegraphics[width=0.8\textwidth]{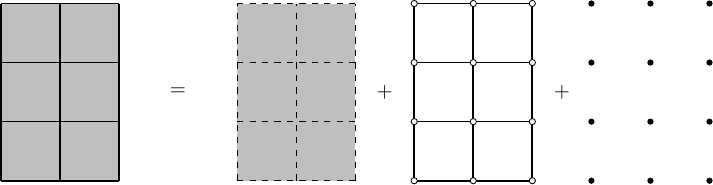}
 \caption{Full Decomposition}
\end{figure}


Associated with a bounded function $f:[a,b] \to \R$, introduce step functions on $\R^d$ by 
\[
  f^\Delta = \sum_i R_i\,(\sup f(R_i)),
  \quad
  f_\Delta = \sum_i R_i\, (\inf f(R_i))
\]
and, given a family $\xi = (\xi_i \in R_i)$ of sample points, 
another step function by 
\[
  f_{\Delta,\xi} = \sum_i R_i f(\xi_i) \in S(\R^d)
\]
so that $f_\Delta \leq f \leq f^\Delta$, $f_\Delta \leq f_{\Delta,\xi} \leq f^\Delta$ and
\[
  I(f_{\Delta,\xi}) = \sum_i |R_i| f(\xi_i). 
\]
Observe here that $|R_i| = 0$ if $R_i$ is not an open interval. 

If $\Delta' = \Delta_1'\times \dots \times \Delta_d'$ is a refinement of $\Delta$ in the sense that
$\Delta_j \subset \Delta_j'$ ($1 \leq j \leq d$), then $f_\Delta \leq f_{\Delta'} \leq f^{\Delta'} \leq f^\Delta$. 

If the mesh size of $\Delta$ 
is defined by $|\Delta| = \sqrt{|\Delta_1|^2 + \dots + |\Delta_d|^2}$ (the maximal diameter of parts),
we see that $0 \leq f^\Delta - f_\Delta \leq [a,b] C_f(|\Delta|)$. 

\begin{Example}
  Let $\Delta^{(l)}$ ($l \geq 1$) be the $l$-th dyadic partition\index{dyadic partition} of $[a,b]$.
  Then $\Delta^{(l)}$ is increasing in $l$ and $|\Delta^{(l)}| = 2^{-l} \max\{ b_j - a_j; 1 \leq j \leq d\}$. 
\end{Example}

\begin{Theorem}\label{MCI} 
  For a closed rectangle $[a,b]$ in $\R^d$, the following holds.
  \begin{enumerate}
  \item
  $C([a,b]) \subset S_\uparrow(\R^d) \cap S_\downarrow(\R^d)$ by zero-extension. 
\item
  If $f \in C([a,b])$, 
  $\lim_{|\Delta| \to 0} \| f^\Delta - f_\Delta\|_{[a,b]} = 0$ and the Cauchy-Riemann formula 
  \[
    \int_{[a,b]} f(x)\, dx \equiv I_\updownarrow(f) 
    = \lim_{|\Delta| \to 0} I(f_{\Delta,\xi}) = \lim_{|\Delta| \to 0} \sum_i |R_i| f(\xi_i) 
  \]
  holds, i.e., given $\epsilon > 0$, we can find $\delta>0$ so that $|\Delta| \leq \delta$ implies
  $\| f^{\Delta} - f_\Delta\|_{[a,b]} \leq \epsilon$ and $|I_\updownarrow(f) - I(f_{\Delta,\xi})| \leq \epsilon$ for
  any choice $\xi = (\xi_i)$ of sample points. 
\item
  Each partial integral $I^{(j)}$ ($1 \leq j \leq d$) gives rise to a linear map $C([a,b]) \to C([a,b]_j)$,
  where
  \[
    [a,b]_j = [a_1,b_1]\times \dots [a_{j-1},b_{j-1}]\times [a_{j+1},b_{j+1}] \times \dots \times [a_d,b_d], 
  \]
  so that each single-variable integral in the repeated integral formula of $I_\updownarrow$ for the volume integral in Proposition~\ref{RULI} 
  is described as a width integral on $C([a_j,b_j]) \subset S_\uparrow(\R) \cap S_\downarrow(\R)$.
\end{enumerate}
\end{Theorem}

{\small
\begin{Remark}
  With a bit more labor we can even show that the monotone extension of the volume integral on $S_\uparrow(\R^d) \cap S_\downarrow(\R^d)$
  coincides with the so-called Riemann integral. See Appendix~B for details. 
\end{Remark}}

\begin{Corollary}~ \label{rcut}
  \begin{enumerate}
    \item
  Rectangular cuts of $C_c(\R^d)$ are included in $S_\uparrow \cap S_\downarrow$.
  Here $C_c(\R^d)$ \index{+Cc@$C_c(U)$ continuous and compact} denotes the set of continuous functions on $\R^d$ having bounded supports.
\item
  A function $f: \R^d \to [0,\infty)$ supported by an open product subset
  $(a,b) = (a_1,b_1)\times \dots\times (a_d,b_d)$ ($-\infty \leq a_i < b_i \leq \infty$) of $\R^d$ belongs to $S_\uparrow(\R^d)$ if
  $f$ is continuous on $(a,b)$.
\end{enumerate}
\end{Corollary}

\begin{proof}
(i) For $f \in C_c(\R^d)$, choose a closed rectangle $R$ so that $[f] \subset R$.
  Then $0 \leq R \| f\| \pm f \in S_\uparrow$ because it is supported by $R$ and continuous on $R$.
  Thus $f \pm R\| f\| \in S_\updownarrow$ and hence $f \in S_\uparrow \cap S_\downarrow$ in view of $R \| f\| \in S$.

  Since $S_\uparrow \cap S_\downarrow$ is an algebra and contains rectangles, rectangular cuts of $f$ belong to $S_\uparrow \cap S_\downarrow$.

  (ii) follows by adapting the proof of Lemma~\ref{doubly} (ii) to the present situation. There we rely on a diagonal argument,
  which can be replaced with a sophisticated version in Lemma~\ref{two-monotone}
\end{proof}

\begin{Exercise}
  Check the theorem with the help of uniform continuity in Theorem~\ref{UC}. 
\end{Exercise}

\begin{Example}
  Let $r>0$ and consider $f(x,y) = (x+y)^{-r}$ supported by $[x>a,y>b]$ ($a\geq 0, b \geq 0$),
  which belongs to $S_\uparrow(\R^2)$ and $I_\uparrow(f)$ is calculated
  by the repeated integral formula in the following manner: 
  \begin{align*}
    \int_a^\infty dx\, \int_b^\infty dy\, (x+y)^{-r} &= \int_a^\infty \frac{1}{r-1} (x+b)^{1-r}\, dx\\
    &= \begin{cases} \frac{1}{(r-1)(r-2)} (a+b)^{2-r} &(r>2),\\
        \infty &(r \leq 2). 
      \end{cases}
  \end{align*}
\end{Example}

\begin{Example}
  Given $r_i$ ($1 \leq i \leq d$), $f(x,\dots,x_d) = e^{-(r_1x_1^2 + \dots + r_d x_d^2)}$ belongs to $S_\uparrow(\R^d)$
  and its volume integral $I_\uparrow(f)$ is calculated by Proposition~\ref{RULI} to be 
  \[
    \frac{1}{\sqrt{r_1\dots r_d}} \left( \int_{-\infty}^\infty e^{-t^2}\, dt \right)^d. 
  \]
\end{Example}

\begin{Exercise}
  Given a vector-valued function $F = (F_1,\dots,F_l):[a,b] \to \R^l$ with $F_j \in C([a,b])$, show that 
  \[
    \int_{[a,b]} F(x)\, dx \equiv \Bigl(\int_{[a,b]} F_j(x)\, dx \Bigr)_{1 \leq j \leq l} \in \R^l.
  \]
  satisfies
  \[
    \left| \int_{[a,b]} F(x)\, dx \right|
    \leq \int_{[a,b]} |F(x)|\, dx, 
  \]
  where $|v| = \sqrt{(v_1)^2+ \dots + (v_l)^2}$ for $v = (v_1,\dots,v_l) \in \R^l$.
\end{Exercise}

For later use in \S8, we add here a variation of Darboux approximation. Given a partition $\Delta$ of a closed rectangle $[a,b]$,
there is a decomposition of $(a,b]$ into open-closed subrectangles, which is described by $(a,b] = \bigsqcup_{R \in (\Delta]} R$,
and lowe/upper approximations of a bounded function $f$ on $(a,b]$ in terms of open-closed parts are defined by
\[
  f_{(\Delta]} = \sum_{R \in (\Delta]} R (\inf\{ f(x); x \in R\}),
  \quad
  f^{(\Delta]} = \sum_{R \in (\Delta]} R (\sup\{ f(x); x \in R\})
\]
so that $f_{(\Delta]} \leq f_\Delta \leq f \leq f^\Delta \leq f^{(\Delta]}$ on $(a,b]$. Moreover, given a family $\xi = \{ \xi(R)\}$
of sample points $\xi(R) \in R$ for each $R \in (\Delta]$, define 
\[
  f_{(\Delta],\xi} = \sum_{R \in (\Delta]} R \xi(R)
\]
so that $f_{(\Delta]} \leq f_{(\Delta],\xi} \leq f^{(\Delta]}$.


\begin{Theorem}\label{CRF} 
When $f$ is continuous on $(a,b]$,  we have 
  \[
    \lim_{|\Delta| \to 0} \| f^{(\Delta]} - f_{(\Delta]}\|_{(a,b]} = 0, 
    \quad
      \lim_{|\Delta| \to 0} \| f - f_{(\Delta],\xi}\|_{(a,b]} = 0
  \]
  and 
  \[
    \int_{[a,b]} f(x)\, dx = \int_{(a,b]} f(x)\, dx = \lim_{|\Delta| \to 0} I(f_{(\Delta],\xi}).
  \]
\end{Theorem}

\bigskip
As a supplement to these results, notice that, for functions in $S_\uparrow(\R^d) \cap S_\downarrow(\R^d)$ (which contains $C_c(\R^d)$),
each single-variable integral in the repeated integral formula is realized as the width integral on $S_\uparrow(\R) \cap S_\downarrow(\R)$.

In the multi-dimensional case, however,
it still entails a rectangular character and does not allow all reasonable domains as elements in $S_\uparrow(\R^d) \cap S_\downarrow(\R^d)$: 
If $D$ is a bounded open set with its boundary $\partial D$ having a lower-dimensional but non-rectangular shape,
then $D \in S_\uparrow$ but $D \not\in S_\downarrow$.
For example, an open disk $D = \{ (x,y); x^2 + y^2 < 1\}$ in $\R^2$ belongs to $S_\uparrow(\R^2)$, whereas $D \not\in S_\downarrow(\R^2)$.

This kind of defects come from the fact that $S_\updownarrow(\R^d)$ excludes lower-dimensional subsets
other than rectangular ones (as indicators).

\begin{Exercise}
  Show that any open disk $D \not= \emptyset$ does not belong to $S_\downarrow(\R^2)$ as an indicator.
\end{Exercise}




We therefore relax exact-limit description of functions in $L_\updownarrow$ to allow exceptional sets such as lower-dimensional boundaries.
This is achieved in the following sophisticated form of the method of exhaustion due to P.J.~Daniell.

\section{Daniell Extention and Convergence Theorems}
Let $(L,I)$ be an integral system on a set $X$. 

\begin{Definition}
Given a function $f: X \to [-\infty,\infty]$, 
its \textbf{upper}\index{upper integral} and  
\textbf{lower integrals}\index{lower integral} are defined by 
\[
{\overline I}(f) = \inf \{ I_\uparrow(g); g \in L_\uparrow, 
f \leq g \}, 
\quad 
{\underline I}(f) = 
\sup \{ I_\downarrow(g); g \in L_\downarrow, 
g \leq f \}, 
\]
which are elements in the extended real line 
$\overline{\R} = [-\infty,\infty]$. 
Recall that $\inf(\emptyset) = \infty$ and 
$\sup(\emptyset) = -\infty$.
\index{+I-loweri@$\underline{I}$ lower integral}
\index{+I-upperi@$\overline{I}$ upper integral}
\end{Definition}


\begin{Proposition}\label{upper}~
  Let $f,g:X \to [-\infty,\infty]$. 
  \begin{enumerate}
  \item 
${\underline I}(f) = - {\overline I}(-f)$. 
\item
  ${\overline I}(r f) = r {\overline I}(f)$ for $0 < r < \infty$. 
\item
  $\underline{I}(f) \leq \overline{I}(f)$. 
\item
If $f \leq g$, 
${\underline I}(f) \leq {\underline I}(g)$ and ${\overline I}(f) \leq {\overline I}(g)$. 
\item
When $f+g$ is well-defined, i.e., 
there is no $x \in X$ satisfying 
$f(x) = \pm\infty$ and $g(x) = \mp \infty$, we have 
${\overline I}(f+g) \leq {\overline I}(f) + {\overline I}(g)$. 
\item
For $f \in L_\uparrow \cup L_\downarrow$, 
we have ${\underline I}(f) = {\overline I}(f)$. 
Moreover this value is equal to 
$I_\uparrow(f)$ or $I_\downarrow(f)$ according 
to $f \in L_\uparrow$ or $f \in L_\downarrow$. 
  \end{enumerate}
\end{Proposition}

\begin{proof}
The assertions (i)--(v) are immediate from the definition. 


To see (vi), first notice that 
${\overline I}(f) = I_\uparrow(f)$ 
($f \in L_\uparrow$) and 
${\underline I}(f) = I_\downarrow(f)$ 
($f \in L_\downarrow$). 
Especially, 
${\underline I}(f) = \overline{I}(f) = I(f)$ 
for $f \in L$. 

Now let $f \in L_\uparrow$ and choose $f_n \in L$ 
so that $f_n \uparrow f$.  
Then 
\[
I_\uparrow(f) = \lim_n I(f_n) 
= \lim_n {\underline I}(f_n) \leq 
{\underline I}(f).
\]
On the other hand, for $f \in L_\uparrow$, we have 
${\overline I}(f) = I_\uparrow(f)$ as already checked. 
Thus ${\overline I}(f) = {\underline I}(f) = I_\uparrow(f)$. 
\end{proof}

\begin{Exercise}
Supply the details for (i)--(v). 
\end{Exercise}

Since any integral of $f$ should be between $\underline{I}(f)$ and $\overline{I}(f)$, we arrive at the following.

\begin{Definition}\label{D-integrable}
  We say that a function $f: X \to \R$ is \textbf{$I$-integrable} or simply \textbf{integrable}\footnote{The notion is due to Daniell
 and originally called `summable'.}\index{integrable} if 
${\underline I}(f) = {\overline I}(f) \in \R$ 
(the upper and the lower integrals are finite and coincide). 
The totality of integrable functions is denoted by $L^1(I)$ or simply $L^1$.
\index{+L-v@$L^1$ Daniell extension of $L$} 
For $f \in L^1$, the value 
${\underline I}(f) = {\overline I}(f) \in \R$ is denoted by $I^1(f)$.
\index{+I-v@$I^1$ Daniell extension of $I$} 

A subset $A$ of $X$ is said to be \textbf{integrable} if it is integrable as an indicator function with its integral $I^1(A)$ 
called the \textbf{$I$-measure}\index{measure} of $A$ and denoted by $|A|_I$.
  Clearly rectangles are Lebesgue integrable. 
\index{+IImeasure@$\lvert A \rvert_I$ $I$-measure}
\index{+IILmeasure@$\lvert A \rvert$ Lebesgure measure}

In the case $L = S(\R^d)$ with the volume integral $I$, $L^1$ is denoted by $L^1(\R^d)$ and
$I$-integrability is also referred to as being \textbf{Lebesgue integrable})\index{Lebesgue integrable}
by a historical reason. In accordance with this, the volume-measure of a Lebesgue integrable set $A$
is called the \textbf{Lebesgue measure}\index{Lebesgue measure} and denoted by $|A|$. 
\end{Definition}

\begin{Exercise}
  For $f:X \to [-\infty,\infty]$ and $g \in L^1$, we have $\overline{I}(f+g) = \overline{I}(f) + I^1(g)$
  and $\underline{I}(f+g) = \underline{I}(f) + I^1(g)$. 
\end{Exercise}


It is not clear at this point but all reasonable bounded sets turn out to be integrable based on convergence theorems
(see Corollary~\ref{g-delta}). 

\begin{Lemma}\label{DA} 
A function $f: X \to \R$ is integrable if and only if 
\[
\forall \epsilon > 0,\ \exists f_+\in L_\uparrow,\  
\exists f_- \in L_\downarrow, 
\quad
f_- \leq f \leq f_+, 
\  
I_\uparrow(f_+) - I_\downarrow(f_-) \leq \epsilon.
\]
Moreover, if $f_-$ increases ($f_+$ decreases) 
in such a way that $f_- \leq f \leq f_+$ and 
$I_\uparrow(f_+) - I_\downarrow(f_-) = 
I_\uparrow(f_+ - f_-) \geq 0$ goes to $0$, 
then 
\[ 
I_\downarrow(f_-) \uparrow I^1(f), 
\quad 
I_\uparrow(f_+) \downarrow I^1(f).
\]
\end{Lemma}

\begin{proof}
Use the inequality 
$I_\downarrow(f_-) \leq {\underline I}(f) 
\leq {\overline I}(f) \leq I_\uparrow(f_+)$. 
\end{proof}

\begin{Theorem}~ \label{DE} 
\begin{enumerate}
\item
The set $L^1$ is a vector lattice on $X$ and 
includes $L_\uparrow \cap L_\downarrow$. 
\item
$I^1: L^1 \to \R$ is a positive linear functional satisfying 
$I^1(f) = I_\uparrow(f) = I_\downarrow(f)$ 
for $f \in L_\uparrow \cap L_\downarrow$. 
In particular, $I^1$ is an extension (called the \textbf{Daniell extension}\index{Daniell extension}) of the preintegral $I: L \to \R$. 
\end{enumerate}
\end{Theorem}

\begin{proof}
Let $f, g \in L^1$. Assume that 
$f_+, g_+ \in L_\uparrow$ 
and $f_-, g_- \in L_\downarrow$ satisfy 
$f_- \leq f \leq f_+$, $g_- \leq g \leq g_+$.  
Then $f_- + g_- \leq f + g \leq f_+ + g_+$ and we see
that
\[
I_\uparrow(f_+ + g_+) - I_\downarrow(f_- + g_-) 
= (I_\uparrow(f_+) - I_\downarrow(f_-)) 
+ (I_\uparrow(g_+) - I_\downarrow(g_-))
\]
can be chosen arbitrarily small, i.e., 
$f + g \in L^1$ and $I^1(f+g) = I^1(f) + I^1(g)$. 

Next, let $r > 0$. Since
$r f_- \leq r f \leq r f_+$, 
we see that 
\[
I_\uparrow(r f_+) - I_\downarrow(r f_-)
= r (I_\uparrow(f_+) - I_\downarrow(f_-))
\]
can be arbitrarily small, i.e., 
$r f \in L^1$ and $I^1(r f) = r I^1(f)$.

If we notice $-f_+ \leq -f \leq -f_-$ 
($-f_+ \in L_\downarrow$, $-f_- \in L_\uparrow$), 
\[
I_\uparrow(-f_-) - I_\downarrow(-f_+) 
= I_\uparrow(f_+) - I_\downarrow(f_-)
\]
can be chosen small as well, i.e.,  
$-f \in L^1$ and $I^1(-f) = -I^1(f)$. 

So far we have checked that $L^1$ is a vector space and 
$I^1$ is a linear functional on $L^1$. 

To show that $L^1$ is closed under the lattice operation, 
it suffices to check 
$f \in L^1 \Longrightarrow f \vee 0 \in L^1$ in view of $|f| = (f\vee 0) + (-f)\vee 0$ and Definition~\ref{ll}, which 
can be seen as follows. 
From $f_- \vee 0 \leq f \vee 0 \leq f_+ \vee 0$, 
we have the inequality 
\[
0 \leq f_+ \vee 0 - f_- \vee 0 \leq f_+ - f_-, 
\]
which is used to see that 
\[
0 \leq I_\uparrow(f_+ \vee 0) - I_\downarrow(f_- \vee 0) 
= I_\uparrow(f_+ \vee 0 - f_- \vee 0) 
\leq I_\uparrow(f_+ - f_-)
\]
can be chosen arbitrarily small. 
In particular, for $f \geq 0$, 
$I^1(f) = I^1(f \vee 0) \geq 0$ as a limit of 
$I_\uparrow(f_+ \vee 0) \geq 0$. 

Finally, if $f \in L_\uparrow \cap L_\downarrow$ , 
we can find $f_\pm \in L$ such that 
$f_- \leq f \leq f_+$, which, together with 
Proposition\,\ref{upper}(v), shows that 
${\underline I}(f) = {\overline I}(f) 
\in [I(f_-), I(f_+)]$ is finite. 
\end{proof}

\begin{Definition}
  The Daniell extension of the volume integral on $S(\R^d)$ is called the
  \textbf{Lebesgue integral}.\index{Lebesgue integral}
\end{Definition}

\begin{Example}
  Target functions of definite integral are Lebesgue integrable with definite integrals equal to Lebesgue integrals.
  For improper integrals, conditionally convergent ones are not Lebesgue integrable
  because $L^1(\R)$ is closed under taking absolute value functions.
  
  We shall see in the next section that absolutely convergent ones are Lebesgue integrable. 
\end{Example}

\begin{Exercise}
  Show that integrable sets are closed under taking finite unions and differences. 
\end{Exercise}

For a later use, we record here the following.




\begin{Proposition}\label{sublattice}~
  \begin{enumerate}
  \item
     A function $f$ in $L_\updownarrow$ is integrable if and only if it is real-valued and $\pm I_\updownarrow(f) < \infty$.
  \item
    $L^1 \cap L_\uparrow - L^1 \cap L_\uparrow$ is a linear lattice and
  $I^1(f_\uparrow + f_\downarrow) = I_\uparrow(f_\uparrow) + I_\downarrow(f_\downarrow)$ for $f_\updownarrow \in L_\updownarrow \cap L^1$.
\item
  $L^1 \cap L_\uparrow = L^1\cap L_\uparrow^+ + L$ and
  $L^1 \cap L_\uparrow - L^1 \cap L_\uparrow = L^1 \cap L_\uparrow^+ - L^1 \cap L_\uparrow^+$. 
\end{enumerate}
\end{Proposition}

\begin{proof}
  (i) If $f \in L_\uparrow$ is integrable, there is $h \in L_\uparrow$ such that $f \leq h$ and $I_\uparrow(h) < \infty$,
  whence $I_\uparrow(f) < \infty$.

  Conversely, if $f \in L_\uparrow$ is real-valued, there exists an increasing sequence $(f_n)$ in $L$ satisfying
  $f_n \uparrow f$ and $I(f_n) \leq \underline{I}(f) \leq \overline{I}(f) \leq I_\uparrow(f)$ for $n \geq 1$ shows 
  that $\underline{I}(f) = \overline{I}(f) = I_\uparrow(f)$. Thus, if the condition $I_\uparrow(f) < \infty$ is further satisfied,
  $f$ is integrable and $I^1(f) = I_\uparrow(f)$. 
  
  (ii) 
  Since $L_\uparrow$ is semilinear and $L^1$ is a linear space, $L^1 \cap L_\uparrow - L^1 \cap L_\uparrow$ is a linear space.
  Let $f = f_1 - f_2$ with $f_j \in L^1 \cap L_\uparrow$. Since both $L^1$ and $L_\uparrow$ are lattices,
  $f_1 \diamond f_2 \in L^1 \cap L_\uparrow$ and then $|f| = f_1\vee f_2 - f_1 \wedge f_2 \in L^1 \cap L_\uparrow - L^1 \cap L_\uparrow$.
  Thus $L^1 \cap L_\uparrow - L^1 \cap L_\uparrow$ is closed under taking absolute values.

  (iii) Let $f \in L_\uparrow$ be expressed as $f_n \uparrow f$ with $f_n \in L$. If $f \in L^1$, $f-f_1 \in L^1 \cap L_\uparrow^+$ and
  $f = (f-f_1) + f_1 \in L^1\cap L_\uparrow^+ + L$. By a similar expression for another $g \in L^1 \cap L_\uparrow$, we see that
  \begin{align*}
    f-g &= (f-f_1) - (g-g_1) + f_1 - g_1\\
        &= (f-f_1) - (g-g_1) + 0\vee(f_1 - g_1) - 0\vee(g_1-f_1)\\
    &= \bigl(f-f_1 + 0\vee(f_1-g_1)\bigr) - \bigl(g-g_1+0\vee(g_1-f_1)\bigr)
  \end{align*}
with $f-f_1 + 0\vee(f_1-g_1)$ and $g-g_1+0\vee(g_1-f_1)$ in $L^1\cap L_\uparrow^+$.   
\end{proof}

\begin{Exercise}
  A function $f$ on $\R^d$ is said to be Riemann integrable\index{Riemann integrable},
  if we can find functions $g, h$ in $S(\R^d)$ so that $g \leq f \leq h$ and
  $\int (h(x) - g(x))\, dx$ can be arbitrarily small.
  
  Show that Riemann integrable functions are Lebesgue integrable and functions in $S_\uparrow(\R^d) \cap S_\downarrow(\R^d)$ are Riemann integrable. 
\end{Exercise}


We now establish a series of convergence theorems on integrable functions,
which exhibit some completeness (or maximality) of Daniell extension. 
To this end, we need to look into $I_\updownarrow$ more closely.






\begin{Lemma}~ \label{two-monotone}
  \begin{enumerate}
  \item 
$f_n \uparrow f$ with $f_n \in L_\uparrow$  
implies
$f \in L_\uparrow$ and 
$I_\uparrow(f_n) \uparrow I_\uparrow(f)$. 
  \item 
$f_n \downarrow f$ with $f_n \in L_\downarrow$ implies 
$f \in L_\downarrow$ and 
$I_\downarrow(f_n) \downarrow I_\downarrow(f)$. 
  \end{enumerate}
\end{Lemma}

\begin{proof}
By symmetry it suffices to prove (i). 
For each $f_n \in L_\uparrow$, choose a sequence 
$(f_{n,m})_{m \geq 1}$ so that $f_{n,m} \uparrow f_n$. 
To get the monotonicity for $(f_{n,m})_{n \geq 1}$, 
we introduce their push-ups by 
\[
g_{n,m} = f_{1,m} \vee f_{2,m} \vee \dots \vee f_{n,m}.
\]
Here $g_{1,m} = f_{1,m}$ by definition. 
Clearly $g_{n,m}$ is increasing in $n$. 
Since $f_{n,m}$ is increasing in $m$, so is $g_{n,m}$ in $m$. 
Moreover 
\[
f_{n,m} \leq g_{n,m} \leq f_1 \vee f_2 \vee \dots \vee f_n = f_n
\]
shows that $g_{n,m} \uparrow f_n$ for each $n$. 

With this preparation in hand, we pick up the diagonal 
$(g_{n,n})_{n \geq 1}$, which is an increasing sequence in $L$. 
Taking the limit $m \to \infty$ in the obvious inequality 
\[
f_{n,m} \leq g_{n,m} \leq g_{m,m} \leq f_m, 
\quad m \geq n,
\]
we obtain 
\[
f_n \leq \lim_{m \to \infty} g_{m,m} \leq f
\]
and then, letting $n \to \infty$, 
\[
f = \lim_{m \to \infty} g_{m,m} \in L_\uparrow. 
\]

Now $I_\uparrow$ is applied in the above inequalities to obtain 
\[
I(f_{n,m}) \leq I(g_{m,m}) \leq I_\uparrow(f_m) 
\quad (m \geq n)
\]
and, after taking the limit  $m \to \infty$, 
\[
I_\uparrow(f_n) \leq I_\uparrow(f) \leq 
\lim_{m \to \infty} I_\uparrow(f_m).  
\]
Thus, letting $n \to \infty$, we finally have
\[
\lim_{n \to \infty} I_\uparrow(f_n) = I_\uparrow(f).  
\]
\end{proof}



\begin{Corollary}\label{convergence}~ 
For a sequence $f_n \in L_\uparrow^+$, 
$\sum_n f_n \in L_\uparrow^+$ and 
\[
I_\uparrow\left(
\sum_n f_n 
\right) 
= \sum_n I_\uparrow(f_n).
\]
\end{Corollary}

\begin{proof}
  Though it is immediate from (i) in the lemma, this is a core of convergence theorems discussed below, whence we shall provide
  a direct proof as a record (the double sum identity being the essence of convergence theorems).

  We first remark that a function $h: X \to [0,\infty]$ belongs to $L_\uparrow^+$ if and only if $h = \sum h_n$
  for a sequence $(h_n)$ in $L^+$. Moreover, if this is the case, we have $I_\uparrow(h) = \sum I(h_n)$. 

  Returning to the proof, this remark enables us to choose sequences $(h_{m,n})_{m \geq 1}$ in $L^+$ so that
$f_n = \sum_m h_{m,n}$ and $I_\uparrow(f_n) = \sum_m I(h_{m,n})$. 
Then $\sum_n f_n = \sum_{m,n} h_{m,n} \in L_\uparrow^+$ and 
\[
I_\uparrow\left(
\sum_n f_n 
\right) 
= \sum_{m,n} I(h_{m,n}) 
= \sum_n \left(
\sum_m I(h_{m,n}) 
\right)
= \sum_n I_\uparrow(f_n). 
\]
\end{proof}



\begin{Lemma}[subadditivity of upper integrals]\label{subadditivity}\index{subadditivity}
If a function $f: X \to [0,\infty]$ has an expression 
$f = \sum_{n=1}^\infty f_n$ with $f_n \geq 0$, then 
\[
{\overline I}(f) \leq \sum_{n=1}^\infty 
{\overline I}(f_n).
\]
\end{Lemma}

\begin{proof}
We may assume that ${\overline I}(f_n) < \infty$ ($n \geq 1$). 
Given any $\epsilon>0$, choose $g_n \in L_\uparrow^+$ so that 
\[
f_n \leq g_n, \quad 
I_\uparrow(g_n) \leq {\overline I}(f_n) + \frac{\epsilon}{2^n}.  
\]
In view of $\sum_n g_n \in L_\uparrow^+$ and 
$I_\uparrow(\sum_n g_n) = \sum_n I_\uparrow(g_n)$ (Corollary~\ref{convergence}), we then have 
\[
{\overline I}(f) 
\leq I_\uparrow\left( \sum_n g_n \right) 
= \sum_n I_\uparrow(g_n) 
\leq \sum_n {\overline I}(f_n) 
+ \sum_n \frac{\epsilon}{2^n} 
= \sum_n {\overline I}(f_n) + \epsilon. 
\]
\end{proof}

\begin{Theorem}[Monotone Convergence Theorem]\index{monotone convergence theorem}
 For a real-valued function $f$ satisfying $f_n \uparrow f$ with $f_n \in L^1$,
$f$ is integrable if and only if $\displaystyle \lim_{n \to \infty} I^1(f_n) < \infty$.  
Moreover, if this is the case, $\displaystyle I^1(f) = \lim_{n \to \infty} I^1(f_n)$. 
\end{Theorem}

\begin{proof}
From $I^1(f_n) = {\overline I}(f_n) \leq {\overline I}(f)$, 
$\displaystyle \lim_{n \to \infty} I^1(f_n) = \infty$ implies 
${\overline I}(f) = \infty$ and hence $f \not\in L^1$.
Let $\displaystyle \lim_{n\to\infty} I^1(f_n) < \infty$. 
We apply the above lemma to 
$f - f_1 = \sum_{n=1}^\infty (f_{n+1} - f_n)$ and obtain 
\begin{multline*}
{\overline I}(f-f_1) \leq 
\sum_{n=1}^\infty {\overline I}(f_{n+1} - f_{n}) 
= \sum_{n=1}^\infty I^1(f_{n+1} - f_{n})\\ 
= \sum_{n=1}^\infty \Bigl(I^1(f_{n+1}) - I^1(f_{n})\Bigr) 
= \lim_{n \to \infty} I^1(f_{n+1}) - I^1(f_1), 
\end{multline*}
whence 
\[
{\overline I}(f) \leq {\overline I}(f_1) 
+ {\overline I}(f - f_1) 
= I^1(f_1) + {\overline I}(f-f_1) 
\leq \lim_n I^1(f_n).
\]
On the other hand, if we take a limit in 
$I^1(f_n) = {\underline I}(f_n) \leq {\underline I}(f)$, 
\[
\lim_n I^1(f_n) \leq {\underline I}(f) 
\leq {\overline I}(f) \leq 
\lim_n I^1(f_n), 
\]
showing that $f$ is integrable and $I^1(f) = \lim_n I^1(f_n)$. 
\end{proof}

\begin{Corollary}
The positive linear functional $I^1$ is continuous, 
i.e., $I^1$ is a preintegral on $L^1$. 
\end{Corollary}

\begin{Exercise}
  Show that integrable sets are closed under taking countable intersections.
\end{Exercise}

As an illustration of usefulness of the monotone convergence theorem,
we shall derive the \textbf{de Moivre-Stirling formula}\index{de Moivre-Stirling} 
(also known as Stirling's formula) \index{Stirling's formula} of the gamma function:
\[
  \lim_{x \to \infty} \frac{\Gamma(x+1)}{\sqrt{2\pi x} x^x e^{-x}} = 1.
  \]
To see this, in the expression
\[
  \Gamma(x+1) = \int_0^\infty s^x e^{-s}\, ds,
\]
observe that the logarithmic integrand $h(s,x) = \log(s^x e^{-s})$ ($s>0$ with $x>0$ a parameter) is maximized at $s=x$ 
with its Taylor expansion around $s=x$ given by
\[
  h(s,x) = x\log x - x - \frac{1}{2} \frac{(s-x)^2}{x} + \cdots, 
\]
which suggests us to introduce the new variable $t = (s-x)/\sqrt{x}$ to have
\[
  \Gamma(x+1) = e^{-x} x^x \sqrt{x} \int_{-\sqrt{x}}^\infty\left( 1 + \frac{t}{\sqrt{x}} \right)^x e^{-t\sqrt{x}}\, dt
\]
and the problem is reduced to showing
\[
\lim_{x \to \infty} \int_{-\sqrt{x}}^\infty\left( 1 + \frac{t}{\sqrt{x}} \right)^x e^{-t\sqrt{x}}\, dt = \sqrt{2\pi}. 
\]

To see the asymptotic behavior of this integrand, we again consider its logarithm $g(t,x)$ ($x>0$, $t> - \sqrt{x}$)
and rewrite it as
\begin{align*}
  g(t,x) &= x \log\left( 1 + \frac{t}{\sqrt{x}} \right) - t\sqrt{x}\\
         &= x \int_0^{t/\sqrt{x}} \frac{1}{1+u}\, du - x \int_0^{t/\sqrt{x}} du\\
         &= -x \int_0^{t/\sqrt{x}} \frac{u}{1+u}\, du
           = - \int_0^t \frac{v}{1+v/\sqrt{x}}\, dv. 
\end{align*}
From the last expression, a continuous function $f$ of $t \in \R$ and $x>0$ defined by 
\[
  f(t,x) = \begin{cases} e^{g(t,x)} &(t > -\sqrt{x}),\\
             0 &(t \leq -\sqrt{x})
           \end{cases}
\]
satisfies $f(t,x) \downarrow e^{-t^2/2}$ ($t \geq 0$) and $f(t,x) \uparrow e^{-t^2/2}$ ($t \leq 0$)
for the limit $x \uparrow \infty$. Notice $f(0,x) = 1$ ($x>0$). 

Since $f(t,x) \leq f(t,1) = (1+t)e^{-t}$ ($x \geq 1$, $t \geq 0$) and
$f(t,x) \leq e^{-t^2/2}$ ($x>0$, $t \leq 0$) are integrable functions of $t \in \R$,
we can apply the monotone convergence theorem to see that 
\[
  \int_{-\sqrt{x}}^\infty \left( 1 + \frac{t}{\sqrt{x}} \right)^x e^{-t\sqrt{x}}\, dt
  = \int_{-\infty}^\infty f(t,x)\, dt
  = \int_0^\infty f(t,x)\, dt + \int_{-\infty}^0 f(t,x)\, dt 
\]
converges to
\[
  \int_{-\infty}^\infty e^{-t^2/2}\, dt = \sqrt{2\pi}
\]
as $x \to \infty$ (see Example~\ref{gaussian1} for the last equality) and we are done.

\begin{Exercise}
Check the properties of $f(t,x)$ in the above explanation. 
\end{Exercise}

\begin{Theorem}[Dominated Convergence Theorem]\label{DCT}\index{dominated convergence theorem}
If a sequence $(f_n)$ in $L^1$ and a function $g \in L^1$ 
satisfy $|f_n| \leq g$ ($n \geq 1$), then 
$\inf_{n \geq 1} f_n$, $\sup_{n \geq 1} f_n$, 
$\liminf_{n \to \infty} f_n$ and 
$\limsup_{n \to \infty} f_n$ are all integrable and 
\[
I^1(\liminf f_n) \leq \liminf I^1(f_n) 
\leq \limsup I^1(f_n) 
\leq I^1(\limsup f_n). 
\]
In particular, if the limit function 
$\displaystyle f = \lim_{n \to \infty} f_n$ exists, 
$f \in L^1$ and 
\[
I^1(f) = \lim_{n \to \infty} I^1(f_n). 
\]
\end{Theorem} 

\begin{proof} For a natural number $m$, we see 
\[
-g \leq \inf_{n \geq m} f_n
\leq f_m \wedge \dots \wedge f_n 
\leq f_m \vee \dots \vee f_n 
\leq \sup_{n \geq m} f_n 
\leq g
\]
and 
\[
f_m \wedge \dots \wedge f_n \downarrow 
\inf_{n \geq m} f_n,
\quad  
f_m \vee \dots \vee f_n \uparrow \sup_{n \geq m} f_n. 
\]
whence, by the monotone convergence theorem 
and the positivity of $I^1$, we have 
$\inf_{n\geq m} f_n, \sup_{n \geq m} f_n \in L^1$ 
and 
\begin{gather*}
I^1(\inf_{n \geq m} f_n) = \lim_n 
I^1(f_m \wedge \dots \wedge f_n) 
\leq \lim_n I^1(f_m) \wedge \dots \wedge I^1(f_n) 
= \inf_{n \geq m} I^1(f_n),\\
I^1(\sup_{n \geq m} f_n) = 
\lim_n I^1( f_m \vee \dots \vee f_n) 
\geq \lim_n I^1(f_m) \vee \dots \vee I^1(f_n)
= \sup_{n \geq m} I^1(f_n).
\end{gather*}
In other words, we have 
\[
-I^1(g) \leq I^1(\inf_{n \geq m} f_n) 
\leq \inf_{n \geq m} I^1(f_n) 
\leq \sup_{n \geq m} I^1(f_n)
\leq I^1(\sup_{n \geq m} f_n)
\leq I^1(g)
\]
and then, again by the monotone convergence theorem, we see that 
$\liminf_n f_n$ and $\limsup_n f_n \in L^1$ are integrable and satisfy 
\begin{multline*}
-I^1(g) \leq I^1(\liminf_n f_n) 
\leq \liminf_n I^1(f_n)\\ 
\leq \limsup_n I^1(f_n) 
\leq I^1(\limsup_n f_n) 
\leq I^1(g).
\end{multline*}
\end{proof}

Since $(L^1,I^1)$ is agian an integral system, we can apply the Daniell extension but it does not give a strict extension.
Let $L^1_\updownarrow = (L^1)_\updownarrow$ and $I^1_\updownarrow: L^1_\updownarrow \to \pm (-\infty,\infty]$ be the monotone extensions of $(L^1,I^1)$
with the associated upper and lower integrals denoted by $\overline{I^1}$ and $\underline{I^1}$ respectively.
 \index{+L-vu@$L^1_\uparrow$ upper extension of $L^1$}
 \index{+I-vu@$I^1_\uparrow$ upper extension of $I^1$}

\begin{Theorem}[Maximality of Daniell Extension]\label{maximality}\index{maximality of Daniell extension}
  We have $\overline{I^1} = \overline{I}$ and $\underline{I^1} = \underline{I}$. The Daniell extension of $(L^1,I^1)$ is therefore
  $(L^1,I^1)$ itself.
\end{Theorem}

\begin{proof}
  By symmetry it suffices to show that $\overline{I^1} = \overline{I}$. Since $I^1$ is an extension of $I$, $I^1_\uparrow$ extends $I_\uparrow$,
  whence $\overline{I^1}(f) \leq \overline{I}(f)$ and the equality holds trivially when $\overline{I^1}(f) = \infty$.
  So we assume that $\overline{I^1}(f) < \infty$.

  Given $\epsilon > 0$, we can find a sequence $(f_n)$ in $L^1$ such that $f_n \geq 0$ for $n \geq 2$, $f \leq \sum_{n \geq 1} f_n$ and
  $\sum_{n \geq 1} I^1(f_n) \leq \overline{I^1}(f) + \epsilon$. Since $\overline{I}(f_n) = I^1(f_n) < \infty$ for any $n \geq 1$,
  we can choose a sequence $(f_{n,j})_{j \geq 1}$ in $L$ so that $f_{n,j} \geq 0$ for $(n,j) \not= (1,1)$,
  $f_n \leq \sum_{j \geq 1} f_{n,j}$ and 
  $\sum_{j \geq 1} I(f_{n,j}) \leq \overline{I}(f_n) + \epsilon/2^n$.

  Thus $f \leq \sum_{n,j \geq 1} f_{n,j}$ and 
  \[
    \sum_{n,j \geq 1} I(f_{n,j}) \leq \sum_{n \geq 1} \overline{I}(f_n) + \epsilon \leq \overline{I^1}(f) + 2\epsilon
  \]
 imply $\overline{I}(f) \leq \overline{I^1}(f) + 2\epsilon$, proving the reverse inequality.  
\end{proof}

\begin{Definition}\label{sigma-integrable}
  A subset $A \subset X$ is said to be \textbf{$\bm \sigma$-integrable}\index{+sigma-integrable@$\sigma$-integrable}
  if it is a union of countably many $I$-integrable sets.

  When $I$ is the volume integral on $S(\R^d)$, $\sigma$-integrable sets are said to be
  \textbf{Lebesgue measurable}\index{Lebesgue measurable}.
\end{Definition}

\begin{Proposition}\label{m-sets}~
  \begin{enumerate}
    \item
      A subset $A \subset X$ is $\sigma$-integrable if and only if it belongs to $L^1_\uparrow$ as an indicator function.
      \item
        $\sigma$-integrable sets are closed under taking countable unions, countable intersections and differences.
      \item
        The intersection of a $\sigma$-integrable set and an integrable set is integrable.
      \item
        Lebesgue measurable sets are closed under taking complements furthermore.
      \end{enumerate}
    \end{Proposition}

    \begin{proof}
      Non-trivial is the if part in (i). To see this, we argue as in Corollary~\ref{ipushup}:
      Let $A \in L^1_\uparrow$ and write $f_n \uparrow A$ with $0 \leq f_n \in L^1$.
      Then, for $f \in L^1$ satisfying $0 \leq f \leq A$ and $r>0$,
      the monotone convergence theorem is applied to $(rf_n)\wedge f \uparrow r\wedge f \leq f$ and we know
      $r\wedge f \in L^1$.
      
      Then $1 \wedge n(f- r\wedge f) = n(\frac{1}{n}\wedge(f-r\wedge f))$ as well as $f - r\wedge f$ is integrable.
      Moreover, $f(x) > r$ ($f(x) - r\wedge f(x) > 0$) implies $1 \wedge n(f- r\wedge f) \leq f/r$.
      
      Now the push-up formula $1\wedge n(f - r\wedge f) \uparrow [f>r]$ (Lemma~\ref{pushup}) is combined with the monotone convergence theorem
      to see that $[f>r]$ is integrable. 
      In particular, $[f_n > r]$ is integrable for each $n \geq 1$ and $[f_n > r] \uparrow A$ ($n \to \infty$) for $0 < r < 1$ shows that
      $A$ is $\sigma$-integrable. 
    \end{proof}

  \begin{Exercise}
 Check other parts in Proposition~\ref{m-sets}. 
\end{Exercise}

  


  The $I$-measure\index{measure} on $I$-integrable sets is extended to $\sigma$-integrable sets by $I^1_\uparrow$:
  In terms of an expression $A_n \uparrow A$ with $A_n$ $I$-integrable,
  \index{+IImeasure@$\lvert A \rvert_I$ $I$-measure}
\[
  |A|_I = I^1_\uparrow(A) = \lim_{n \to \infty} |A_n|_I \in [0,\infty].
\]

  \begin{Proposition}\label{m-cut}
    Let $f \in L^1(I)$ and $A$ be $\sigma$-integrable with respect to $I$. Then $Af \in L^1(I)$.
    When $f$ is relaxed to be in $(L^1)_\uparrow$, $Af \in (L^1)_\uparrow$. 
\end{Proposition}

  \begin{proof}
    We may assume that $f \geq 0$ and first consider an $I$-integrable $A$.
    Since simple functions $Ar$ ($r>0$) are $I$-integrable, so are
    $(Ar) \wedge f \in L^1(I)$ and the monotone convergence theorem is used to see that
    \[
      Af = \lim_{n \to \infty} (An) \wedge f \in L^1(I).
    \]

    Now let $A$ be $\sigma$-integrable and write $A_n \uparrow A$ with $A_n$ $I$-integrable.
    Then $A_n f \uparrow Af$ with $A_nf \in L^1(I)$ and $A_n f \leq f$.
    Again the monotone convergence theorem works here to see that $Af \in L^1(I)$.

    When $f \in (L^1)_\uparrow$, we can express $f_n \uparrow f$ with $f_n \in L^1$ and 
    $Af_n \uparrow Af$ with $Af_n \in L^1$ shows that $Af \in (L^1)_\uparrow$. 
  \end{proof}

  For a Lebesgue measurable set $A \subset \R^d$ and a function $f \in (L^1(\R^d))_\uparrow$,
  $I^1_\uparrow(Af)$ is also denoted by 
  \[
    \int_A f(x)\, dx
  \]
  so that $\int_A f(x)\, dx = I^1(Af)$ for $f \in L^1(\R^d)$.

  {\small
  \begin{Remark}
  See Appendix D for an overall account on measurable sets and measurable functions. 
  \end{Remark}}

\section{Properties and Applications of Lebesgue Integrals}
We now specialize to the volume integral on the space $S(\R^d)$ of step functions and realize how big $L^1(\R^d)$ is.

Recall Proposition~\ref{sublattice} that,
if we denote by $S_\updownarrow^1(\R^d)$ the totality of \textit{real-valued} functions in $S_\updownarrow(\R^d)$, say $f_\updownarrow$,
fulfilling $\pm I_\updownarrow(f_\updownarrow) < \infty$, then $S_\updownarrow^1(\R^d) = S_\updownarrow(\R^d) \cap L^1(\R^d)$, 
$I^1(f_\uparrow + f_\downarrow) = I_\uparrow(f_\uparrow) + I_\downarrow(f_\downarrow)$ for $f_\updownarrow \in S_\updownarrow^1(\R^d)$
and $S_\uparrow^1(\R^d) + S_\downarrow^1(\R^d)$ is a linear sublattice of $L^1(\R^d)$.
\index{+stepupper-v@$S^1_\uparrow$ integrable $S_\uparrow$}
\index{+steplower-v@$S^1_\downarrow$ integrable $S_\downarrow$}

\begin{Example}
  If an improperly integrable function $f$ supported by an open interval $(a,b) \subset \R$ is absolutely convergent, then it belongs to
  $S_\uparrow^1(\R) + S_\downarrow^1(\R)$ with the improper integral of $f$ equal to $I^1(f)$.

  To see this, let $[a_n,b_n] \subset (a,b)$ increase to $(a,b)$. Since $[a_n,b_n]f \in S_\uparrow \cap S_\downarrow$ and
  $0\vee ([a_n,b_n](\pm f)) = [a_n,b_n](0\vee (\pm f)) \in S_\uparrow \cap S_\downarrow$ increases to $(a,b) (0 \vee(\pm f))$,
  the absolute convergence implies $(a,b) (0\diamond f) \in S_\updownarrow^1(\R)$ and hence 
  $f = (a,b)f = (a,b)(0\vee f) + (a,b) (0\wedge f) \in S_\uparrow^1(\R) + S_\downarrow^1(\R)$.  
\end{Example}

\color{teal}
As a supplement to this example, we here generalize the notion of improper integral\index{improper integral}:
A function $f$ on an open interval $(a,b)$ is
said to be locally integrable\index{locally integrable} if $[x,y]f$ is Lebesgue integrable for any subinterval $[x,y] \subset (a,b)$.
Given a locally integrable function $f$, a function $F$ on $(a,b)$ is called an indefinite integral\index{indefinite integral} of $f$ if
$\int_{[x,y]} f(t)\, dt = F(y) - F(x)$ for any $[x,y] \subset (a,b)$.

Indefinite integrals of a locally integrable function $f$ always exist, which are continuous and unique up to additive constants:
Uniqueness is immediate. To see other facts, choose $c \in (a,b)$ and set
\[
  F(x) =
  \begin{cases}
    \int_{[c,x]} f(t)\, dt &(x \geq c)\\
    - \int_{[x,c]} f(t)\, dt &(x \leq c)
  \end{cases}. 
\]
Then $F$ is an indefinite integral of $f$ and continuous in $x \in (a,b)$ by the dominated convergence theorem.
It is customary to write $F(x)$ by $\int f(x)\, dx$.
The improper integral\index{improper integral} of $f$ is then defined to be
\[
  \int_a^b f(t)\, dt = \lim_{\substack{x \to a+0\\y \to b-0}} \int_{[x,y]} f(t)\, dt 
\]
which exists if and only if both $F(a) \equiv \lim_{x \to a+0} F(x)$ and $F(b) \equiv \lim_{x \to b-0} F(x)$ exist.
With these limits, the improper integral is calculated by $F(b) - F(A)$. 
\color{black}

\medskip
Returning to the multi-dimensional case,
given an open subset $U$ of $\R^d$,
let $S(U)$ be the set of linear combinations of rectangles contained in $U$ 
\index{+stepU@$S(U)$ step functions in $U$}
and $C(U)$ be the set of continuous functions on $U$, which are algebra-lattices on $U$ and identified with functions on $\R^d$ by zero extension.

The following strengthens Corollary~\ref{rcut}. 

\begin{Proposition}\label{CU}~ 
  \begin{enumerate}
  \item
    $U$ is a disjoint union of countably many open-closed rectangles. 
  \item
   The positive part $S_\uparrow^+(U)$ of the sequentially increasing extension of $S(U)$ satisfies 
\[
  S_\uparrow^+(U) = \{ f \in S_\uparrow^+(\R^d); Uf = f\} = U S_\uparrow^+(\R^d).  
\]
Here $U S_\uparrow^+(\R^d) = \{ Uf; f \in S_\uparrow^+(\R^d)\}$ by definition.
\item
 We have $C^+(U) \subset S_\uparrow^+(U) \subset S_\uparrow^+(\R^d)$. 
    \item
      For a real-valued function $f$ on $U$, let $f_\pm = 0\vee(\pm f)$ and assume that $f_\pm \in S_\uparrow^+(U)$. 
      
     Then $f$ is integrable if and only if $I_\uparrow(f_\pm) < \infty$.
     Moreover, if this is the case, we have $I^1(f) = I_\uparrow(f_+) - I_\uparrow(f_-)$. 
\end{enumerate}
\end{Proposition}

\begin{proof}
  (i) For $n \geq 1$, let $\cI_n$ be $d$-products of
  intervals of the form $((k-1)/2^n,k/2^n]$ ($k \in \Z$) and let $U_n$ be the union of $R \in \cI_n$ satisfying $\overline{R} \subset U$.
  Since each $R \in \cI_n$ is a finitely many disjoint union of elements in $\cI_{n+1}$, $U_n$ is increasing in $n$ and
  $U_n \setminus U_{n-1}$ is expressed by a disjoint union of countably many elements in $\cI_n$
  (see Figure~7\footnote{A similar figure can be found in \cite[Bild 1.3]{F} for example.}), whence
  $\bigcup_{n \geq 1} U_n = \bigsqcup_{k \geq 1} R_k$ with $R_k$ an open-closed rectangle satisfying $\overline{R_k} \subset U$.
  
  We claim that $U = \bigcup_{n \geq 1} U_n$. If not, there is $a \in U$ with the property that, 
  for $n \geq 1$ and $R \in \cI_n$, $a \in R$ implies $\overline{R} \setminus U \not= \emptyset$.
  By choosing $R_n \in \cI_n$ and  $x_n \in \overline{R_n} \setminus U$ so that $\overline{R_n} \downarrow \{ a\}$, 
  $a = \lim_{n \to \infty} x_n$ belongs to the closed set $\R^d \setminus U$, which contradicts with $a \in U$.
  
  (ii)
  The inclusions $S_\uparrow^+(U) \subset \{ f \in S_\uparrow^+(\R^d); Uf = f\} \subset U S_\uparrow^+(\R^d)$ are obvious.
  Let $U_l = \sum_{k=1}^l R_k \in S^+(U)$ in the notation of (i) so that $U_l \uparrow U$ and,
  given $h \in S_\uparrow^+(\R^d)$, choose $h_l \in S^+(\R^d)$ so that $h_l \uparrow h$.
  Then an increasing sequence $(U_l h_l)$ is in $S^+(U)$ and converges to $Uh$, proving $ Uh \in S_\uparrow^+(U)$.

  (iii) 
For $f \in C(U)$, $U_l f\in S_\uparrow \cap S_\downarrow$ (Corollary~\ref{rcut}) and, if $f \geq 0$,
  $U_l f \uparrow f$ and hence $f \in U S_\uparrow^+(\R^d)$ by Lemma~\ref{two-monotone}. 

(iv) follows from Proposition~\ref{sublattice} (i) applied for $L^1 = L^1(\R^d)$. 
\end{proof}

\medskip
\begin{Corollary}~ \label{g-delta}
  \begin{enumerate}
  \item
    Open sets as well as closed sets and their differences in $\R^d$ are Lebesgue measurable.
  \item
    Countable intersections of bounded open subsets of $\R^d$ are Lebesgue-integrable.
    \item      
      A function $f \in C(U)$ is integrable if and only if so is $|f| \in C^+(U)$.
      In that case, $f_\pm = 0 \vee (\pm f) \in S_\uparrow(U)$ are integrable and we have 
      $I^1(f) = I^1(f_+) - I^1(f_-)$. 
\end{enumerate}
\end{Corollary}

\begin{figure}[h]\index{dyadic tiling}
  \centering
 \includegraphics[width=0.5\textwidth]{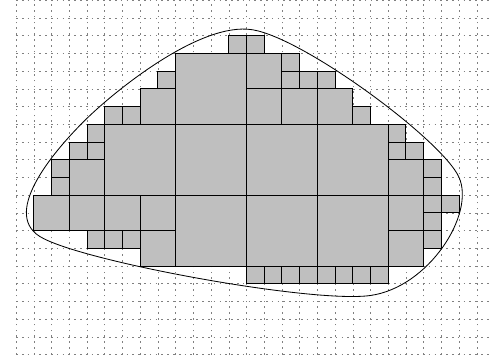}
 \caption{Dyadic Tiling}
\end{figure}

Notice here that an open subset $U$ of $\R^d$ is Lebesgue measurable and $C^+(U) \subset S^+_\uparrow(\R^d) \subset (L^1(\R^d))_\uparrow$,
whence $\int_U f(x)\, dx \in [0,\infty]$ has a meaning for $f \in C^+(U)$ as the value $I^1_\uparrow(f)$. 

\begin{Example} 
  For an open set $U \subset \R^d$, consider a continuous function $f$ supported by a compact subset of $U$ and 
  let $C_c(U)$ \index{+Cc@$C_c(U)$ continuous and compact} be the totality of such functions,
  which is identified with a subspace of $\{ f \in C_c(\R^d); [f] \subset U\}$ by the obvious inclusion $C_c(U) \subset C_c(\R^d)$.
  Recall also that $S(U) = \{ f \in S(\R^d); Uf = f\}$.
  
  These are linear sublattices of $L^1(\R^d)$ and the volume integral (or the Lebesgue integral) is restricted to provide
  integral systems.
  By the way of definition,
  their Daniell extensions are then realized as restrictions of $I^1$ to $C_c(U)^1 \subset L^1(\R^d)$ and $S(U)^1 \subset L^1(\R^d)$ respectively
  (see Theorem~\ref{transfer} for a more systematic account).

  Moreover, in view of $S(U) \subset C_c(U)^1$ and $C_c(U) \subset S(U)^1$, the maximality of Daniell extension reveals that
  $S^1(U) \subset C_c(U)^1$ and $C_c(U)^1 \subset S(U)^1$, i.e., $C_c(U)^1 = S(U)^1$, which is denoted by $L^1(U)$.
  \index{+L-vU@$L^1(U)$ Daniell extension of $S(U)$} 
\end{Example}  

Based on this fact, we henceforth regard $L^1(U)$ as a Daniell extension of $S(U)$ or $C_c(U)$ relative to the volume integral. 

\begin{Exercise}
  Show that $S(U) \subset C_c(U)^1$ and $C_c(U) \subset S(U)^1$.
  Hint: $S(U)$ and $C_c(U)$ are mutually approximated by doubly bounded seqeuncial limits in a dominated fashion.  
\end{Exercise}

\begin{Proposition}
  For an open subset $U$ of $\R^d$, $L^1(U) = UL^1(\R^d)$. 
\end{Proposition}

\begin{proof}
  Not obvious is the inclusion $UL^1(\R^d) \subset L^1(U)$.
  
  Since $L^1(\R^d)$ is a linear lattice, it suffices to check $Uf \in L^1(U)$ for any $0 \leq f \in L^1(\R^d)$. 
To see this, express $U = \bigsqcup_{k \geq 1} R_k$ with $R_k$ rectangles (Proposition~\ref{CU} (i)) 
and approximate $f$ by $f_\downarrow \leq f \leq f_\uparrow$ ($0 \leq f_\updownarrow \in S_\updownarrow(\R^d)$)
so that $I_\uparrow(f_\uparrow - f_\downarrow)$ is small.

Then, in view of $R_kf_\updownarrow \in S_\updownarrow(U)$ (we can approximate $f_\updownarrow \geq 0$ by monotone sequences in $S^+(\R^d)$) and
$I_\uparrow(R_k(f_\uparrow - f_\downarrow)) \leq I_\uparrow(f_\uparrow - f_\downarrow)$,
$R_kf$ is approximated by $R_kf_\updownarrow \in S_\updownarrow(U)$, whence 
$R_kf \in L^1(U)$.

Since the sequence $(R_1+\cdots + R_m)f$ in $L^1(U)$ converges to $Uf$ as $m \to \infty$ in such a way that
$(R_1+\cdots+R_m)f$ is dominated by $0 \leq f \in L^1(\R^d)$, the monotone convergence theorem is applied to have $Uf \in L^1(U)$.
\end{proof}

{\small
\begin{Remark}
  By a $\sigma$-induction (i.e., a monotone class argument) with measure-theoretical completion accompanied,
  one can generalize the last cutting property to arbitrary Lebesgue measurable sets.
\end{Remark}}

\begin{Exercise}
Let $U$ be a bounded open set of $\R^d$ and $f$ be a bounded continuous function on $U$. Then $f \in L^1(U)$. 
\end{Exercise}

  Let $C_b(U)$ be the totality of bounded continuous functions on an open subset $U$ of $\R^d$ and regard it as defined on $\R^d$ by zero extension. 

\begin{Proposition}\label{ccl}
We have $C_b(U) L^1(\R^d) = L^1(U) \subset L^1(\R^d)$. 
\end{Proposition}

\begin{proof}
  Since $C_b(U)$ is a linear lattice, it suffices to show $hf \in L^1(\R^d)$ for each $h \in C_b^+(U)$ and $f \in L^1(\R^d)$.  
  By the level approximation in Appendix D, express $h$ as $h_n \uparrow h$ with $h_n$ positive linear combinations of open sets.
  We know $h_n f \in L^1(\R^d)$ by Corollary~\ref{g-delta} (i) and Proposition~\ref{m-cut}.
  The dominated convergence theorem is then applied for $|h_n f| \leq \| h\| |f|$ to see $hf \in L^1(\R^d)$, whence
  $hf = Uhf \in UL^1(\R^d) = L^1(U)$. 
\end{proof}

\begin{Proposition}
  For a compact\footnote{In $\R^d$, this is equivalent to requiring that $K$ is bounded and closed.} subset $K$ of $\R^d$,
  let $C(K)$ be the set of continuous functions on $K$. 
  Then, by zero-extension to $\R^d \setminus K$, $C^+(K) \subset S_\downarrow^1(\R^d)$ and hence 
      $C(K) \subset S_\uparrow^1(\R^d) + S_\downarrow^1(\R^d)$. 
\end{Proposition}
    
\begin{proof}
  Choose an open rectangle $R$ so that $K \subset R$. Then $R \setminus K \in S_\uparrow^1(\R^d)$ as a bounded open subset
  and $K = R - (R\setminus K) \in S(\R^d) - S_\uparrow^1(\R^d) = S(\R^d) + S_\downarrow^1(\R^d) = S_\downarrow^1(\R^d)$.
  
  Now each $f \in C^+(K)$ is extended to $h \in C^+(\R^d)$ thanks to Tieze extension (Theorem~\ref{tietze}),
  which is assumed to have a comact support by replacing it with $\theta h$.
  Here $\theta \in C_c^+(\R^d)$ is chosen\footnote{For example, $\theta = \theta_1\otimes \dots \otimes \theta_d$
  with $\theta_i \in C_c^+(\R)$ satisfying $(a_i,b_i)\theta_i = (a_i,b_i)$.} to satisfy $R\theta = R$.
  
  Then $h \in C_c(\R^d) \subset S_\uparrow(\R^d) \cap S_\downarrow(\R^d)$ is combined with $K\| f\|_\infty \in S_\downarrow(\R^d)$ to see that
  $f = h \wedge (K \| f\|_\infty) \in S_\downarrow(\R^d)$, which is integrable in view of $I_\downarrow(f) \geq 0$.
\end{proof}

\begin{Example}\label{OCI} 
  Let $\phi: U \to V$ be a bicontinuous \textbf{change-of-variables}\index{change-of-variables}, i.e., $U$ and $V$ are open subsets of $\R^d$,
  $\phi: U \to V$ is a bijection with $\phi$ and $\phi^{-1}$ continuous. Let $[a,b]$ be a closed rectangle inside $U$.
  
  Then, for any rectangle $R$ such that $\overline{R} = [a,b]$, say an open-closed one $(a,b]$, $\phi(R)$ is Lebesgue-integrable.

  In fact, if $R = (a,b]$ for example, we can choose a sequence $(b_n)$ of points in $U$ so that $(a,b_n) \downarrow (a,b]$ inside $U$.
  Since $\phi([a,b_n])$ is bounded as a continuous image of $[a,b_n]$,
  $\phi((a,b]) = \bigcap \phi((a,b_n))$ is a countable intersection of bounded open sets $\phi((a,b_n))$, whence
  $\phi((a,b])$ is Lebesgue integrable by Corollary~\ref{g-delta}. 
\end{Example}

{\small
\begin{Remark}
There is a bicontinuous change-of-variables which does not preserve Lebesgue measurable (more precisely negligible) sets. 
\end{Remark}}


The following are simple applications of the dominated convergence theorem. Notice that the continuity in $t$ is equivalent to
the sequential one: If $t_n \to t$ in $(a,b)$, then  $\displaystyle \lim_{n \to \infty}\int_{\R^d} f(x,t_n)\, dx = \int_{\R^d} f(x,t)\, dx$.

\begin{Proposition}[Parametric continuity]\label{PC}\index{parametric continuity}
  Let $f(x,t)$ be a real-valued function on $\R^d\times (a,b)$ and assume the following conditions. 
  \begin{enumerate}
  \item
    For each $t \in (a,b)$, $f(x,t)$ is an integrable function
    of $x \in \R^d$. 
    \item
    For each $x \in \R^d$, $f(x,t)$ is continuous in $t \in (a,b)$. 
  \item
    There exists $g \in L^1(\R^d)$ satisfying $|f(x,t)| \leq g(x)$.
  \end{enumerate}  
Then $\displaystyle \int_{\R^d} f(x,t)\, dx$ is a continuous function of $t \in (a,b)$.   
\end{Proposition}

\begin{Proposition}[Parametric differentiability]\label{PD}\index{parametric differentiability}
  Let $f(x,t)$ be a function on $\R^d\times (a,b)$ satisfying the following conditions.
  \begin{enumerate}
    \item
For each $t \in (a,b)$, $f(x,t)$ is integrable as a function of $x \in \R^d$,
\item
For each $x \in \R^d$, $f(x,t)$ is differentiable in $t \in (a,b)$.  
\item
There exists $g \in L^1(\R^d)$ satisfying 
$\displaystyle \left|\frac{\partial f}{\partial t}(x,t)\right| \leq g(x)$. 
\end{enumerate}
Then $\displaystyle \frac{\partial f}{\partial t}(x,t)$ is 
integrable as a function of $x$ and we have 
\[
\frac{d}{dt} \int_{\R^d} f(x,t)\,dx = \int_{\R^d} 
\frac{\partial f}{\partial t}(x,t)\,dx
\]
for $a < t < b$.
\end{Proposition}

\begin{proof}
  Thanks to an integral inequality 
\[
\left|
\frac{f(x,t+h) - f(x,t)}{h} 
\right|
= \frac{1}{|h|} \left| \int_t^{t+h} \frac{\partial}{\partial s} f(x,s)\, ds \right|
\leq g(x)
\]
for $t, t+h \in (a,b)$, we can apply the dominated convergence theorem in the limit
\[
  \lim_{h \to 0} \int_{\R^d} \frac{f(x,t+h) - f(x,t)}{h}\, dx 
\]
to get the assertion. 
\end{proof}

\begin{Corollary}
  Let $U \subset \R^d$ be an open subset and $f$ be a function on $U\times (a,b)$ satisfying the condition that 
  \begin{enumerate}
  \item
    $f(x,t_0)$ is an integrable function of $x \in U$ for some $t_0 \in (a,b)$, 
  \item
  $f(x,t)$ is partially differentiable with respect to $t$ for any $x \in U$,
    \item
  $\displaystyle \frac{\partial f}{\partial t}(x,t) = f'(x,t)$ is a continuous function of $(x,t) \in U\times (a,b)$ and
  there exists $g \in L^1(U)$ satisfying
  \[
    \left|\frac{\partial f}{\partial t}(x,t)\right| \leq g(x)\quad (x \in U, a < t < b).
  \]
    \end{enumerate}
    Then, for each $t \in (a,b)$,
    both $f(x,t)$ and $\displaystyle \frac{\partial f}{\partial t}(x,t)$ are integrable functions of $x \in U$ and 
  $\displaystyle \int_U f(x,t)\, dx$ is continuously differentiable in $t \in (a,b)$ so that
  \[
    \frac{d}{dt} \int_U f(x,t)\, dx = \int_U \frac{\partial f}{\partial t}(x,t)\, dx.
  \]
\end{Corollary}

\begin{proof}
  By (iii), $f'(x,t)$ is an integrable function of $x \in U$ and satisfies
  \[
    \left| \int_{t_0}^t \frac{\partial f}{\partial s}(x,s)\, ds \right| \leq |t-t_0| g(x)
  \]
  for $t \in (a,b)$. Since $\int_{t_0}^t f'(x,s)\, ds \in C(U)$ by Theorem~\ref{MCI} (iii) as a partial integration of $f'(x,s)$ on $s$,
  the integrability of $g$ shows that
  \[
    f(x,t) - f(x,t_0) = \int_{t_0}^t \frac{\partial f}{\partial s}(x,s)\, ds
  \]
  is integrable as a function of $x \in U$ in view of Corollary~\ref{g-delta} (iii) and so is $f(x,t)$ due to (i).

  Thus all the hypotheses in parametric differentiability are satisfied and we have
  \[
    \frac{d}{dt} \int_U f(x,t)\, dx = \int_U \frac{\partial f}{\partial t}(x,t)\, dt,
  \]
  which is in turn continuous in $t \in (a,b)$ by parametric continuity. 
\end{proof}

\begin{Example}
  The gamma function\index{gamma function} 
  \[
    \Gamma(t) = \int_0^\infty x^{t-1} e^{-x}\, dx
  \]
  is infinitely differentiable in $t>0$.

  For $t=1$, $e^{-x}$ is integrable in $x>0$ and, for $0 < a < 1 < b$, 
  \[
    \left(\frac{\partial}{\partial t}\right)^n x^{t-1}e^{-x} = x^{t-1} e^{-x} (\log x)^n
  \]
  is continuous in $(x,t) \in (0,\infty)\times (a,b)$ and dominated by a continuous function 
  \[
  g(x) = (0,1] x^{a-1}e^{-x} |\log x|^n + (1,\infty) x^{b-1} e^{-x} (\log x)^n
  \]
  of $x>0$. Since 
  \[
    \int_0^1 x^{a-1} e^{-x} |\log x|^n\, dx < \infty,
    \quad
  \int_1^\infty x^{b-1} e^{-x} (\log x)^n\, dx < \infty,   
  \]
  $g$ is integrable and the hyptheses in the corollary are fulfilled. 
\end{Example}

\begin{Exercise}
Check the integrability of $g$. 
\end{Exercise}

\begin{Example}
  Differentiation of $\displaystyle \int_0^\infty \frac{dx}{t + x^2} = \frac{\pi}{2\sqrt{t}}$ ($t>0$) gives
  \[
    \int_0^\infty \frac{dx}{(t + x^2)^{n+1}} = \frac{\pi}{2t^n\sqrt{t}} \frac{(2n-1)!!}{(2n)!!}\quad (n=1,2,\cdots).
  \]
\end{Example}

\begin{Exercise}
Find a dominating function of each integrand. 
\end{Exercise}

For a later use in \S9, we describe partitions of unity\index{partition of unity} in the present context.
Let $A \subset \R^d$ be a Lebesgue measurable set and $\rho \in C_c^+(\R^d)$ be a probability density function on $\R^d$, i.e.,
$\int \rho(x)\, dx = 1$. Then, for the translation $\rho_x(y) = \rho(y-x)$ of $\rho$ by $x \in \R^d$,
$A\rho_x$ is integrable (Proposition~\ref{m-cut}) and 
\[
  A^\rho(x) \equiv \int_A \rho(y-x)\, dy = \int_{(A-x) \cap [\rho >0]} \rho(y)\, dy \in [0,1]
\]
(a moving average\index{moving average} of $A$ by $\rho$) is continuous as a function of $x \in \R^d$ by parametric continuity,
which is in the class $C^n$ if so is $\rho$ by parametric differentiability.

From the last equality, one sees that $A^\rho$ vanishes outside an open set $A - [\rho>0] = \bigcup_{a \in A} (a- [\rho>0])$
and $A^\rho(x) = 1$ if $x+ [\rho>0] \subset A$.
Thus, for $\rho$ satisfying $[\rho>0] \subset B_r(0)$, we have 
\[
  \{ x \in \R^d; B_r(x) \subset A\} \leq A^\rho \leq \bigcup_{a \in A} B_r(a).
\]
Here $B_r(a) = \{ x \in \R^d; |x-a| < r\}$\index{+openball@$B_r(a)$ open ball at $a$} denotes
an open ball\index{open ball} of radius $r>0$ centered at $a \in \R^d$.  
This estimate will be utilized when $r>0$ is small. In that case, $\rho$ approximately represents the so-called delta function.

\begin{Proposition}[partition of unity]\label{pou}\index{partition of unity}
  Given a finite open covering $(U_i)_{1 \leq i \leq l}$ of a compact set $K$ in $\R^d$,
  we can find functions $h_i \in C_c^+(\R^d)$ satisfying $[h_i] \subset U_i$, $\sum_i h_i \leq 1$ and $\sum_i h_i = 1$ on $K$.
\end{Proposition}

\begin{proof}
  For each $x \in K$, choose $i$ and then $r>0$ so that $B_{4r}(x) \subset U_i$. 
  By compactness of $K$, we can find a finite subcovering $(B_{r_j}(x_j))$ of $(B_r(x))_{x \in K}$ with the property that, 
  for each $j$, there is an $i$ satisfying $B_{4r_j}(x_j) \subset U_i$.
 Let $B_j \subset B_{2r_j}(x_j)$ be defined inductively by
  \[
    B_{2r_1}(x_1) \cup \dots \cup B_{2r_j}(x_j) = B_1 \sqcup B_2 \sqcup \dots \sqcup B_j
    \quad
    (j=1,2,\dots)  
  \]
  and smooth them out by an approximate delta function $\rho \in C_c^+(B_\delta(0))$ for $\delta = \wedge_j r_j$. 
  We then have $[B_j^\rho] \subset \overline{B_{2r_j+\delta}(x_j)} \subset B_{4r_j}(x_j) \subset U_i$ and 
  \[
    0 \leq \sum_j B_j^\rho = \left(\bigcup_j B_{2r_j}(x_j)\right)^\rho \leq 1. 
  \]
  Moreover, for $x \in K \subset \bigcup_j B_{r_j}(x_j)$, 
  \[
    x + [\rho>0] \subset B_\delta(x) \subset B_{r_j + \delta}(x_j) 
    \subset \bigcup_j B_{2r_j}(x_j)
  \]
  shows that $\sum_j B_j^\rho = \left(\bigcup_j B_{2r_j}(x_j)\right)^\rho = 1$ on $K$.

  Finally, grouping $\{j\}$ into $\bigsqcup_i J_i$ (possibly $J_i = \emptyset$) so that $[B_j^\rho] \subset U_i$ ($j \in J_i$)
  for $i=1,2,\dots,l$, 
  partial sums $h_i = \sum_{j \in J_i} B_j^\rho$ meet the conditions.
\end{proof}

\section{Null Functions and Null Sets}
Exceptional sets mentioned in \S\ref{MRI} are now clearly and firmly described as null sets. 
A function $f: X \to [-\infty,\infty]$ is said to be \textbf{null}\index{null function (set)} or \textbf{negligible}\index{negligible}
if $\overline{I}(|f|) = 0$.
In view of $0 \leq \underline{I}(|f|) \leq \overline{I}(|f|)$, a real-valued function $f$ is null if and only if $f \in L^1$ and $I^1(|f|) = 0$.
A subset $A \subset X$ is null or negligible if so is the indicator function of $A$, i.e., $\overline{I}(A) = I^1(A) = 0$.

Here are simple properties of negligibleness.

\begin{Proposition}~ \label{negligible}
  \begin{enumerate}
  \item
    If $|f| \leq g$ with $g$ a null function, then $f$ is a null function. In particular, a subset of a null set is null.
  \item
    If $(f_n)$ is a sequence of positive null functions, $\sum f_n$ is a null function.
    Likewise, if $(A_n)$ is a sequence of null sets, the union $\bigcup A_n$ is a null set.
  \item
    $f$ is a null function if and only if $[f \not= 0]$ is a null set. 
  \end{enumerate}
\end{Proposition}

\begin{proof}
  (i) follows from the monotonicity of $\overline{I}$.
  
  (ii) follows from the subadditivity of $\overline{I}$ and $\bigcup A_n \leq \sum A_n$.

  (iii) If $f$ is a null function, $\infty |f| = |f| + |f| + \cdots$ is null as well and $[f \not= 0] \leq \infty|f|$ shows that
  $[f \not= 0]$ is a null set. Conversely, if $[f \not= 0]$ is a null set,
  $\infty |f| = [f \not= 0] + [f \not= 0] + \cdots$ is a null function and hence so is $|f| \leq \infty |f|$. 
\end{proof}

As a consequence of (iii), we observe that, for an integrable function $f$, its integral $I^1(f)$ as well as integrability
remains unchanged when $f$ is modified on a null set.

\begin{Example}
    For a function $f$ on $\R^d$, we define the domain of continuity to be the maximal open subset $U$ on which $f$ is continuous. 
    The discontinuity support $[[f]]$ of $f$ is then defined to be $\R^d \setminus U$.
    
  If $[[f]]$ is a null set with respect to the volume integral, $f$ is Lebesgue integrable if and only if $f|U \in C(U)$ is
  Lebesgue integrable.

  When $f$ is the indicator of a subset $A$ of $\R^d$, the discontinuity support is the boundary $\partial A$ of $A$ and,
  under the condition that $|\partial A| = 0$, $A$ is Lebesgue integrable if and only if $|U| < \infty$. 
\end{Example}

For functions $f,g:X \to [-\infty,\infty]$, we write $f\; \mathring{\leq}\; g$ if $[f > g]$ is a null set, which is a semi-order relation among
$\overline{\R}$-valued functions with the associated equivalence relation denoted by $f \;\mathring{=}\; g$.
\index{+almost@$f\;\mathring{=}\; g$ almost equality}
Note that $f \;\mathring{=}\; g$, i.e., $f \;\mathring{\leq}\; g$ and $g \;\mathring{\leq}\; f$, means that $[f \not= g]$ is a null set.

More generally, given a condition $P$ on elements in the base set $X$ of an integral system $(L,I)$,
we say that $P$ is \index{almost all (almost every)}\textbf{almost}\footnote{This tasteful usage of `almost'
  originates from H.~Lebesgue's `presque partout'.\index{presque partout}}
satisfied if $X \setminus [P]$ is a null set.

It is then customary and very useful to talk integrability about functions which are well-defined on $X$ modulo null sets: 
A function $f$ is integrable in this (extended) sense and write $f \mathring{\in} L^1$
if there exists $g \in L^1$ such that $f\;\mathring{=}\; g$,
with its integral $I^1(f)$ well-defined to be $I^1(g)$. 

\begin{Example}
  $\log|x|$ is locally integrable as a function of $x \in \R$ and its indefinite integral (not a primitive function) is given by
  a continuous function $x\log|x| - x +C$. 
\end{Example}

\begin{Exercise}
Check this fact. 
\end{Exercise}

The monotone convergence theorem is now strengthened as follows.

\begin{Theorem}\label{MI} 
  Let $(f_n)$ be an increasing sequence in $L^1$ with $f = \lim f_n$ and assume that $\displaystyle \lim_{n \to \infty} I^1(f_n) < \infty$.
  Then $[f=\infty]$ is a null set and $[f < \infty] f$ is integrable so that 
  $\displaystyle I^1([f<\infty]f) = \lim_{n \to \infty} I^1(f_n)$.
\end{Theorem}

\begin{proof}
  Let $g_n = f_{n+1} - f_n \in L^1$ so that $f- f_1 = \sum_{n \geq 1} g_n$.
  Then, thanks to the subadditivity of upper integrals, 
  \[
    \overline{I}(f-f_1) \leq \sum_{n \geq 1} \overline{I}(g_n) = \sum_{n \geq 1} I^1(g_n)
    = \lim_{n \to \infty} I^1(\sum_{j=1}^n g_j)
    = \lim_{n \to \infty} I^1(f_n-f_1), 
  \]
  which is combined with the monotonicity
  $\lim I^1(f_n-f_1) = \lim \overline{I}(f_n - f_1) \leq \overline{I}(f-f_1)$ to get the equality
  $\overline{I}(f-f_1) = \lim I^1(f_n-f_1)$.
  

  Here $[f-f_1 = \infty] \leq r(f-f_1)$ ($r>0$) is used to have 
  \[
    \overline{I}([f-f_1 = \infty])
    \leq r \overline{I}(f-f_1) = r \lim_{n \to \infty} I^1(f_n-f_1) 
    = r(\lim_{n \to \infty} I^1(f_n) - I^1(f_1)). 
  \]
  Since $I_\uparrow(f) < \infty$ and $r>0$ is arbitrary, this implies that $[f=\infty] = [f-f_1=\infty]$ is a null set
  and then $[f<\infty]f_n \in L^1$ satisfies $I^1([f<\infty]f_n) = I^1(f_n)$ as a modification by a null function.
  
  Now the original monotone convergence theorem is applied to $[f<\infty]f_n \uparrow [f<\infty]f$
  with $\displaystyle \lim_{n \to \infty} I^1([f<\infty]f_n) = \lim_{n \to \infty} I^1(f_n) < \infty$ to see that $[f<\infty]f$ is integrable and
  \[
    I^1([f<\infty]f) = \lim I^1([f<\infty] f_n) = \lim I^1(f_n).
  \]  
\end{proof}


\begin{Corollary}\label{practical}
  Let $f_j \in L^1_\uparrow$ satisfy $I^1_\uparrow(f_j) < \infty$ ($j=1,2$).
  Then $[f_j = \infty]$ ($j=1,2$) are null sets, $[f_1\wedge f_2 < \infty]f_j$ is integrable and 
  $I^1\bigl([f_1\wedge f_2 < \infty]f_1 - [f_1\wedge f_2 < \infty] f_2\bigr) = I^1_\uparrow(f_1) - I^1_\uparrow(f_2)$.
\end{Corollary}

\begin{Exercise}
  Let $f:X \to (-\infty,\infty]$ satisfy $f_n \uparrow f$ with $f_n \in L^1$. 
  Then $\displaystyle\overline{I}(f) = \lim_{n \to \infty} I^1(f_n)$. 
\end{Exercise}

As an almost version of convergence theorem, we record here the following.

\begin{Theorem}[Dominated Series Convergence]\label{DSC}
  Let $(f_n)_{n \geq 1}$ be a sequence of integrable functions with $f_n$ dominated by $g_n \in L^1$ so that
  $|f_n| \leq g_n$ and $\sum_{n=1}^\infty I^1(g_n) < \infty$.
 Then $\sum_{n \geq 1} f_n(x)$ converges absolutely to an integrable function $f(x)$ for almost all $x$ and
  $I^1(f) = \sum_{n=1}^\infty I^1(f_n)$.   
\end{Theorem}

\begin{proof}
  Since the increasing sequence $(\sum_{k=1}^n g_k)$ in $L^1$ satisfies
  \[
    \lim_{n \to \infty} I^1\Bigl(\sum_{k=1}^n g_k\Bigr) = \sum_{n \geq 1} I^1(g_n) < \infty,
  \]
  letting $g = \sum_{n \geq 1} g_n \in L^1_\uparrow$, $[g = \infty]$ is a null set
  with $[g<\infty]g$ and $[g<\infty]f_n \mathring{=} f_n$ integrable.
  Then $f \equiv [g<\infty] \sum_{n \geq 1} f_n$ is integrable and
  \begin{align*}
    I^1(f) &= \lim_{n \to \infty} I^1\Bigl([g<\infty] \sum_{k=1}^n f_k\Bigr)
    = \lim_{n \to \infty}  \sum_{k=1}^n I^1([g<\infty] f_k)\\
           &= \lim_{n \to \infty} \sum_{k=1}^n I^1(f_k) = \sum_{n \geq 1} I^1(f_n)
  \end{align*}
  by the dominated convergence theorem. 
\end{proof}

\begin{Example}
  For $s > 1$, $t^{s-1}/(e^t-1)$ is an integrable function of $t>0$ and we have 
  \[
    \int_0^\infty \frac{t^{s-1}}{e^t - 1}\, dt = \zeta(s) \Gamma(s),
  \]
  where the zeta function\index{zeta function} $\zeta(s)$ is defined by
  \[
    \zeta(s) = \sum_{n=1}^\infty \frac{1}{n^s}.
  \]
  In fact, the dominated series convergece (of ordinary version)
  is applied to the expression $t^{s-1}/(e^t-1) = \sum_{n=0}^\infty t^{s-1} e^{-t} e^{-nt}$ to have 
  \begin{align*}
    \int_0^\infty \frac{t^{s-1}}{e^t - 1}\, dt
    = \sum_{n=0}^\infty \int_0^\infty t^{s-1} e^{-t} e^{-nt}\, dt
    &= \sum_{n=0}^\infty \frac{1}{(n+1)^s} \int_0^\infty t^{s-1} e^{-t}\, dt\\
    &=\zeta(s) \Gamma(s). 
  \end{align*}
\end{Example}

At this point, we have various monotone extensions\index{monotone extension} of $L$ between $L^1 \cap L_\uparrow$ and $L^1_\uparrow = (L^1)_\uparrow$:
\begin{gather*}
  \{ f \in L_\uparrow; I_\uparrow(f) < \infty\},\\
  \{ f:X \to \R; f \,\mathring{\in}\, L_\uparrow\},\\
  \{ f:X \to \R; f \,\mathring{\in}\, L_\uparrow, I^1_\uparrow(f) < \infty\}
\end{gather*}
and so on.
Among these, the last one is interesting because every integrable function is a difference of functions belonging to this class
(see Appendix~C), whereas the first one is practically useful because concrete integrable functions are differences of
functions in this class, 
which shall be utilized in repeated integrals discussed in the next section.

\section{Repeated Integrals Revisited}
Historically a reasonable formulation of the subject had not been apparent for a while and
it was crucial to allow exceptional points which constitute a null set.
We here present a practical form of the so-called Fubini theorem without getting involved much in measurability. 


Let $d = d'+d''$ and express $x \in \R^d$ by $x = (x',x'')$ with $x' \in \R^{d'}$ and $x'' \in \R^{d''}$.
For $A \subset \R^d$, let $A' \subset \R^{d'}$ ($A'' \subset \R^{d''}$) be
the projection \index{projection} of $A$ to the $d'$-component ($d''$-component) respectively and
the slice \index{slice} of $A$ by $x' \in \R^{d'}$ ($x'' \in \R^{d''}$) is defined to be 
$A_{x'} = \{ a'' \in \R^{d''}; (x',a'') \in A\}$ ($A_{x''} = \{ a' \in \R^{d'}; (a',x'') \in A\}$).
Thus $A' = \{ x' \in \R^{d'}; A_{x'} \not= \emptyset\}$.

Note that, for an open set $U$, slices $U_{x'}$, $U_{x''}$ as well as projections $U'$, $U''$ are open sets.

\begin{figure}[h]
  \centering
 \includegraphics[width=0.4\textwidth]{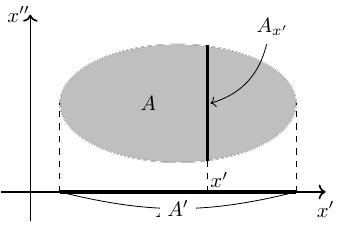}
 \caption{Projection and Slice}
\end{figure}





Proposition~\ref{RULI} is here paraphrased as follows.

\begin{Lemma}\label{repeated-upintegral}
  Let $f \in S_\uparrow(\R^d)$. Then, for each $x' \in \R^{d'}$, $f(x',\cdot) \in S_\uparrow(\R^{d''})$ and 
  $\int f(x',x'')\, dx''$ belongs to $S_\uparrow(\R^{d'})$ as a function of $x'$ in such a way that
  \[
    \int f(x)\, dx = \int dx'\, \int f(x',x'')\, dx''.
  \]
\end{Lemma}

\begin{Proposition}\label{repeated-integral}\index{repeated integral}
  Let $f$ be a real-valued function $f$ on an open subset $U \subset \R^d$ and assume that $f_\pm = 0\vee(\pm f) \in S_\uparrow^+(U)$.
  Recall that the condition $f_\pm \in S_\uparrow^+(U)$ holds if $f \in C(U)$ (Proposition~\ref{CU}).
  The zero-extension of $f$ to $\R^d$ is here denoted by $Uf$. 
  
  Then $Uf$ is Lebesgue integrable if and only if $|f| \in S_\uparrow^+(U)$ fulfills
  $I_\uparrow(U|f|) = I_\uparrow(Uf_+) + I_\uparrow(Uf_-) < \infty$, i.e., 
  \[
    \int_{U'} dx' \int_{U_{x'}} |f(x',x'')|\, dx'' < \infty.
  \]
  Moreover, if this is satisfied, we have
  \[
    \int_U f(x)\, dx = \int_{U'} dx' \int_{U_{x'}} f(x',x'')\, dx''. 
  \]
  Here $f(x',\cdot)$ belongs to $L^1(\R^{d''})$ for almost all $x' \in U'$ and
  $\displaystyle \int_{U_{x'}} f(x',x'')\, dx''$ is integrable as a function of $x' \in U'$. 
\end{Proposition}

\begin{proof}
  By Proposition~\ref{CU} (iv) and Lemma~\ref{repeated-upintegral}, the integrability of $Uf$ is equivalent to the condition that
  $|f| \in S_\uparrow^+(U)$ satisfies
  \[
    \int_{U'} dx' \int_{U_{x'}} |f(x',x'')|\, dx'' =  I_\uparrow(U|f|) = I_\uparrow(Uf_+) + I_\uparrow(Uf_-) < \infty. 
  \]
  Moreover, under this condition,
  \begin{align*}
    \int_U f(x)\, dx &= \int_U f_+(x)\, dx - \int_U f_-(x)\, dx\\
                     &= \int_{U'} dx' \int_{U_{x'}} f_+(x',x'')\, dx'' - \int_{U'} dx' \int_{U_{x'}} f_-(x',x'')\, dx''.
  \end{align*}
  In this expression, Corollary~\ref{practical} is used to see that $\int_U f_\pm(dx)\, dx < \infty$ implies
  $f_\pm(x',\cdot) \in L^1(\R^{d''})$, i.e.,
  $f(x',\cdot) = f_+(x',\cdot) - f_-(x',\cdot) \in L^1(\R^{d''})$ for almost all $x' \in U'$ (as functions on $\R^{d''}$ supported by $U_{x'}$)
  and 
  \begin{gather*}
    \int_{U'} dx' \int_{U_{x'}} f_+(x',x'')\, dx'' - \int_{U'} dx' \int_{U_{x'}} f_-(x',x'')\, dx''\\
    = \int_{U'} dx' \int_{U_{x'}} f(x',x'')\, dx''. 
  \end{gather*}
  %
  %
%
\end{proof}


\begin{Corollary}\label{repeated-integral2}
Let $\varphi < \psi$ be $\overline{\R}$-valued continuous functions on an open subset $U \subset \R^d$ and
$D = \{ (x,y); x \in U, \varphi(x) < y < \psi(x)\} \in S_\uparrow(\R^{d+1})$ be an open 
set bordered by $\varphi$ and $\psi$ (a graph region\index{graph region}).

Then, for a real-valued function $f$ on $D$ satisfying $f_\pm \in S_\uparrow^+(D)$,
$D|f| \in S_\uparrow(\R^{d+1})$ and its integral $I_\uparrow(D|f|)$ is calculated by 
\[
\int_D |f| =  \int_Udx\, \int_{\varphi(x)}^{\psi(x)} |f(x,y)|\, dy
\]
so that this is finite if and only if $Df \in L^1(\R^{d+1})$. Moreover, if this is the case, 
\[
  I^1(Df) = \int_D f = \int_U dx\, \int_{\varphi(x)}^{\psi(x)} f(x,y)\, dy,
\]

In particular, for the choice $f \equiv 1$ and $U = (a,b)$,
the Lebesgue measure $|D|$ of $D$ is expressed by a one-variable integral
\[
  |D| = \int_a^b (\psi(x) - \varphi(x))\, dx, 
\]
which is exactly the area formula for $D$ in the elementary calculus.
\end{Corollary}

\begin{figure}[h]
  \centering
 \includegraphics[width=0.4\textwidth]{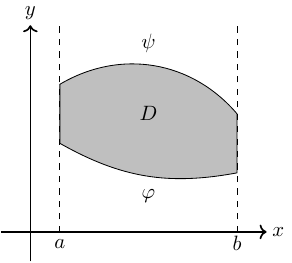}
 \caption{Graph Region}
\end{figure}

{\small
\begin{Remark}
  Repeated integral furmulas discussed so far are further extended to integrable functions, known as Fubini's theorem.
  Our formulation is just restricted to $L_\updownarrow$ in essence, which enables us to avoid almost-technicalities.
  
  One of most efficient way to circumvent these is to rephrase the integrability according to Appendix\ref{moreon} and write down
  $M$-integrability in repeated integrals. 
\end{Remark}}

\begin{Example}
  When $U$ is bounded and $\phi:U \to \R$ is a continuous function, 
the choice $\varphi = \phi - \delta$ and $\psi = \phi + \delta$ with $\delta>0$ gives 
\[
  |D| = \int_U 2\delta\, dx = 2\delta |U| \downarrow 0\quad (\delta \downarrow 0).
\]
Thus the graph $\{(x,\phi(x)); x \in U\} \subset U\times \R$ of $\phi$, which is included in $D$ for any $\delta>0$, is a null set.
This property remains valid even for an unbounded $U$ as a countable union of bounded graphs. 
\end{Example}


\begin{Example}\label{gaussian1}
  Consider the repeated integral of $e^{-(1+x^2)y}$ supported by the first quadrant $(0,\infty)^2 \subset \R^2$.
  In terms of the half Gaussian integral\index{Gaussian integral} $C = \int_0^\infty e^{-x^2}\, dx$, this is
  \[
    \int_{x>0,y>0} e^{-(1+x^2)y}\, dxdy = \int_0^\infty dy\, e^{-y} \int_0^\infty e^{-yx^2}\, dx = C \int_0^\infty e^{-y} \frac{1}{\sqrt{y}}\, dy = 2C^2, 
  \]
  which is equal to
  \[
    \int_0^\infty dx\, \int_0^\infty e^{-(1+x^2)y}\, dy = \int_0^\infty \frac{1}{x^2+1}\, dx = \frac{\pi}{2}. 
  \]
  Thus
  \[
    \int_0^\infty e^{-x^2}\, dx = \frac{\sqrt{\pi}}{2}.
    \]
\end{Example}

\begin{Example}[Dirichlet integral]\index{Dirichlet integral}
  From repeated integrals of the double integral $\displaystyle \int_{x>0, y > r} e^{-xy} \sin x\, dxdy$ ($r>0$), we have 
  \[
    \int_0^\infty e^{-rx} \frac{\sin x}{x}\, dx = \frac{\pi}{2} - \arctan r, 
  \]
  which is the Laplace transform of $\sinc(x) = (\sin x)/x$ ($x > 0$) and its continuity at $r=+0$ (Theorem~\ref{laplace})
  results in the formula 
  \[
    \displaystyle \int_0^\infty \frac{\sin x}{x}\, dx = \frac{\pi}{2}.  
  \]
\end{Example}

\begin{Exercise}
  Compute $\displaystyle \int_0^\infty \frac{(\sin x)^{2n+1}}{x}\, dx$ for $n=1,2,\dots$.
%
\end{Exercise}


In Proposition~\ref{repeated-integral}, `almost all' qualification appeared instead of `all', which is really needed.

\begin{Example}\label{punctured}
  Let $U = \{ (x,y) \in \R^2; 0 < x^2 + y^2 < 1\}$ and consider a continuous function $f$ on $U$ defined by
  $f(x,y) = 1/(x^2+y^2)^{\alpha/2}$ ($\alpha>0$). Then,
  \[
    \int_{-\sqrt{1-x^2}}^{\sqrt{1-x^2}} f(x,y)\, dy = 2\int_0^{\sqrt{1-x^2}} \frac{1}{(x^2+y^2)^{\alpha/2}}\, dy
    \quad(x \not= 0)
  \]
  and
  \[
    \int_{-1}^1 f(0,y)\, dy = 2 \int_0^1 \frac{1}{y^\alpha}\, dy
    =
    \begin{cases}
      2/(1-\alpha) &(\alpha < 1)\\
      \infty &(\alpha \geq 1)
    \end{cases}.
  \]
  Thus the partial integral $\int f(x,y)\, dy$ is not integrable at $x = 0$ if $\alpha \geq 1$.
  
  Even in that case, $\int f(x,y)\, dy$ is continuous in $-1 \leq x \leq 1$ as a $[0,\infty]$-valued function and
  one sees that
  \[
    \int_{-1}^1 dx\, \int_{-\sqrt{1-x^2}}^{\sqrt{1-x^2}} f(x,y)\, dy = \int_U f(x,y)\, dxdy, 
  \]
  which turns out to be finite if and only if $\alpha < 2$ (see Example~\ref{polardisk}). 
\end{Example}

\begin{Exercise}
Show that $\int_{-\sqrt{1-x^2}}^{\sqrt{1-x^2}} f(x,y)\, dy$ is continuous as a $[0,\infty]$-valued function of $x \in [-1,1]$. 
\end{Exercise}

\color{teal}
As a theoretical application, we take up the integration-by-parts formula for improper integrals.
Let $f$ and $g$ be real-valued functions on an open interval $(a,b)$ satisfying the condition that
$[x,y]f, [x,y]g \in S_\uparrow(\R) \cap S_\downarrow(\R)$ for $[x,y] \subset (a,b)$.
Let $F$ and $G$ be indefinite integrals of $f$ and $g$ respectively, i.e., $F$ and $G$ are continuous functions on $(a,b)$ satisfying
\[
  \int_x^y f(t)\, dt = F(y) - F(x),
  \quad
  \int_x^y g(t)\, dt = G(y) - G(x)
\]
for any $[x,y] \subset (a,b)$.

We assume that the improper integrals of $f$ and $g$ exist, i.e., $F$ and $G$ are continuously extended to $[a,b]$ so that
\[
  F(a) = \lim_{x \to a+0} F(x), \quad
  F(b) = \lim_{x \to b-0} F(x)
\]
and similarly for $G$. 

For a subinterval $[x,y] \subset (a,b)$, $[x,y]G \in S_\uparrow(\R) \cap S_\downarrow(\R)$ by Corollary~\ref{regulated} and then
$[x,y] fG \in S_\uparrow(\R) \cap S_\downarrow(\R)$ by Proposition~\ref{immediate} (iii).
Thus we can talk about improper integrability of product functions $fG$ and $Fg$. 

\begin{Proposition}\label{by-parts}
  The improper integrals of $fG$ and $Fg$ 
  exist if one of them exists and, if this is the case, we have
\[
  \int_a^b F(t)g(t)\, dt + \int_a^b f(t)G(t)\, dt = F(b) G(b) - F(a) G(a).
\]
\end{Proposition}

\begin{proof}
Let $h = f\otimes g$, i.e., $h(s,t) = f(s)g(t)$ ($s,t \in (a,b)$) with
\[
  h_+ = f_+\otimes g_+ + f_-\otimes g_-,
  \quad
  h_- = f_+\otimes g_- + f_-\otimes g_+. 
\]

In view of $[x,y]f_\pm = [x,y](|f| \pm f), [x,y]g_\pm = [x,y](|g|\pm g) \in S_\uparrow^+(\R)$ for $[x,y] \subset (a,b)$,
a bounded open set $D = \{ (s,t); x < s < t < y\}$ satisfies
$Dh_\pm = D ([x,y]\times [x,y]) h_\pm \in S_\uparrow^+(D)$ by Proposition~\ref{CU} (ii).
Since $h$ is bounded on $[x,y]\times [x,y]$, $Dh$ as well as $Dh_\pm$ is integrable and
$\displaystyle \int_D h(s,t)\, dsdt$ is calculated in two ways by Corollary~\ref{repeated-integral2} to have 
\begin{align*}
  \int_x^y ds f(s) \int_s^y dt g(t)
  &= \int_x^y ds f(s) (G(y) - G(s))\\
  &= (F(y) - F(x)) G(y) - \int_x^y f(s) G(s)\, ds
\end{align*}
and
\begin{align*}
  \int_x^y dt g(t) \int_x^t f(s)\, ds 
  &= \int_x^y dt g(t) (F(t) - F(x))\\
  &= \int_x^y F(t)g(t)\,dt - F(x)(G(y) - G(x)), 
\end{align*}
whence
\[
  \int_x^y f(s) G(s)\, ds + \int_x^y F(t)g(t)\, dt = F(y)G(y) - F(x)G(x).
\]
The assertion is now obtained by taking limits $x \to a$ and $y \to b$. 
\end{proof}

{\small
\begin{Remark}
 When repeated integrals are fully developped, the integration-by-parts formula can be generalized to locally integrable functions. 
\end{Remark}}

\begin{Exercise}
  A continuous function $f$ in Theorem~\ref{laplace} can be replaced with one satisfying $[x,y]f \in S_\uparrow(\R) \cap S_\downarrow(\R)$
  for $[x,y] \subset (0,\infty)$. 
\end{Exercise}

\color{black}

\section{Jacobian Formula}

To fully appreciate the power of repeated integrals, we here establish the change-of-variables formula in multiple integrals.

\begin{Theorem}[Transfer Principle]\label{transfer}\index{transfer principle}
Let $(L,I)$, $(M,J)$ be integral systems on sets $X$, $Y$ respectively, $\rho:X \to [0,\infty)$ be a function on $X$ and 
$\phi: X \to Y$ be a map satisfying $\rho(M\circ \phi) \subset L$ and $I(\rho(g\circ \phi)) = J(g)$ ($g \in M$).
Here $\rho (g\circ \phi)$ denotes the product function of $\rho$ and $g\circ\phi$. 

Then we have $\rho(M^1 \circ \phi) \subset L^1$ and 
$I^1(\rho(g\circ \phi)) = J^1(g)$ ($g \in M^1$) for their Daniell extensions. 
\end{Theorem}

\begin{proof}
We just check $\rho(M_\uparrow \circ \phi) \subset L_\uparrow$, 
$I_\uparrow(\rho(g\circ \phi)) = J_\uparrow(g)$ ($g \in M_\uparrow$) and so on, step by step.
Details are left to the reader. 
\end{proof}

\begin{Corollary}~ 
  \begin{enumerate}
\item 
If $\phi: X \to Y$ is bijective and 
$L = M \circ \phi$, we have 
$L^1 = M^1\circ \phi$ and $J^1(g) = I^1(g\circ \phi)$ 
($g \in M^1$). 
  \item 
If integral systems $(L,I)$, $(M,J)$ on a set $X$ 
satisfy $L \subset M$, $J|_L = I$, i.e., 
$(M,J)$ is an extension of $(L,I)$, then 
$L^1 \subset M^1$ and $J^1$ on $M^1$ is an extension of $I^1$ on $L^1$.
  \end{enumerate}
\end{Corollary}


\begin{Example}
For $f \in L^1(\R^d)$ and $y \in \R^d$, 
$f(x+y)$ is integrable as a function of $x \in \R^d$ and 
\[
\int_{\R^d} f(x+y)\, dx = \int_{\R^d} f(x)\,dx. 
\]
This follows from translational invariance of the volume integral. 
\end{Example}

\begin{Exercise}
For a function $f \in L^1(\R^d)$ and 
a positive real $r > 0$, check the identity
\[
\int_{\R^d} f(r x)\,dx 
= r^{-d} \int_{\R^d} f(x)\,dx.
\]
\end{Exercise}



We now state our goal (Jacobian\footnote{Recall that the Jacobian of a smooth change-of-variables $y = \phi(x)$ is 
  $\det(\phi'(x))$, which is also denoted by $\frac{\partial(y_1,\dots,y_d)}{\partial(x_1,\dots,x_d)}$.} formula\index{Jacobian formula})
in this section as follows.

\begin{Theorem}\label{jacobian}
  Let $U$, $V$ be open subsets of $\R^d$ and $\phi:U \to V$ be a \index{smooth}smooth change-of-variables, i.e., $\phi$ is bijective with 
  $\phi$ and $\phi^{-1}$ differentiable and the derivative $\phi':U \to M_d(\R)$ of $\phi$ continuous.
  Note that $\phi'(x)$ is an invertible matrix for each $x \in U$.
  
  Then, for $g \in C_c(V) \cup C^+(V)$,
  \[
  \int_V g(y)\, dy = \int_U g(\phi(x)) |\det(\phi'(x))|\, dx. 
    \] 
  
  Thanks to the transfer principle, the Jacobian formula remains valid for $g \in L^1(V)$ as well.
\end{Theorem}

\begin{Corollary}
  A function $g$ on $V$ is Lebesgue integrable (Legesgue negligible)
  if and only if so is $(g\circ\phi) |\det(\phi')|$ on $U$ ($g\circ\phi$ on $U$).
\end{Corollary}

{\small
\begin{Remark}
  Nowadays, there seems some confusion in what Jacobian\index{Jacobian} means. In view of historical flow,
  it was used (and is still used) to express $\det(\phi'(x))$
  but a recent usage is widened to refer to its absolute value as well or even the differential matrix $\phi'(x)$.  
\end{Remark}}

\begin{Proposition}\label{measurable-null}
  A smooth change-of-variables preserves Lebesgue measurable sets as well as Lebesgue null sets.
\end{Proposition}
  
\begin{proof}
  Let $\phi:U \to V$ be a smooth change-of-variables and $B \subset V$ be Lebesgue measurable.
  By Proposition~\ref{m-sets} we have an expression $B_m \uparrow B$ with $(B_m)_{m \geq 1}$ an increasing sequence of Lebesgue integrable sets. 
Since $|\det \phi'|^{-1}$ is a continuous function on $U$, we can find a sequence $h_n \in C_c^+(U)$ so that $h_n \uparrow |\det\phi'|^{-1}$.
Then
\[
 (B_m\circ\phi) |\det\phi'|h_n \in L^1(U) C_c(U) \subset L^1(U)
\]
(Proposition~\ref{ccl}) and $h_n(B_m\circ\phi) |\det\phi'| \uparrow B_m\circ\phi$ shows that $B_m\circ\phi \in L^1_\uparrow(U)$, whence 
$B_m\circ\phi \uparrow B\circ\phi \in L^1_\uparrow(U)$ by Lemma~\ref{two-monotone}. 
Thus $B\circ\phi = \phi^{-1}(B)$ is Lebesgue measurable by Proposition~\ref{m-sets} (i). 

For a null set $B$,
$\int_U (B\circ \phi) |\det \phi'| = \int_V B = |B| = 0$
and $\phi^{-1}(B) = [(B\circ \phi) |\det \phi'|>0]$ is a null set by Proposition~\ref{negligible} (iii). 
\end{proof}

\begin{Corollary}
  Let $U$ be an open subset of $\R^d$ and $\phi: U \to \R$ be a smooth function satisfying $\phi'(x) \not= 0$ ($x \in U$).
  Then each level set $[\phi = c]$ is negligible relative to the volume integral. 
\end{Corollary}

\begin{Example}\label{singularity}
  For $\alpha>0$ and $\beta > 0$, consider the integral
  \[
    \int_{\R^2} \frac{e^{-\beta(x^2+y^2)}}{|x^2+y^2-1|^\alpha}\, dxdy. 
  \]
  whose integrand has singularity on the circle $x^2 + y^2 =1$. Since this unit circle is negligible as a level set of $x^2 + y^2$, 
  the integral as well as integrability is reduced to 
  \[
    \int_{x^2+y^2 \not= 1} \frac{e^{-\beta(x^2+y^2)}}{|x^2+y^2-1|^\alpha}\, dxdy, 
  \]
  which is an integral of continuous function on an open subset. 
\end{Example}

\medskip
  \noindent
  \underline{Proof of Jacobian Formula}
  
  We first deal with the special case when $\phi: \R^d \to \R^d$ is realized by a matrix multiplication:
  Let $T$ be an invertible matrix of size $d$. Then for $f \in C_c(\R^d)$ and hence for $f \in L^1(\R^d)$ by the transfer principle,
  \[
    \int f(Tx)\, dx = |\det T|^{-1} \int f(x)\, dx, 
  \]
  whence $|T^{-1}A| = |\det T|^{-1} |A|$ for a Lebesgue integrable set $A$ of $\R^d$. 
 Remark here that under an invertible linear transformation of variables $C_c(\R^d) \subset S_\uparrow(\R^d) \cap S_\downarrow(\R^d)$ is invariant,
  whereas $S_\uparrow(\R^d) \cap S_\downarrow(\R^d)$ is not as noticed before.
 
  Since any invertible matrix is a product of elementary ones and the volume integral is permutation-invariant,
  the repeated integral formula on $S_\uparrow(\R^d) \cap S_\downarrow(\R^d)$ reduces the problem to checking it for two-dimensional matrices
  \[
    \begin{pmatrix}
      \alpha & 0\\
      0 & \beta
    \end{pmatrix},
    \quad
    \begin{pmatrix}
      1 & \gamma\\
      0 & 1
    \end{pmatrix}, 
  \]
  where $\alpha, \beta \in \R^\times$ and $\gamma \in \R$.
  
  For these,
  the scale covariance and the translational invariance of the width integral are combined with repeated integrals to conclude as follows: 
  \[
    \int_{\R^2} f(\alpha x,\beta y)\, dxdy = \frac{1}{|\alpha||\beta|} \int_{\R^2} f(x,y)\, dxdy.
  \]
  \begin{align*}
    \int_{\R^2} f(x+\gamma y,y)\, dxdy &= \int_\R dy \int_\R f(x+\gamma y,y)\, dx\\
    &(\text{by the translational invariance of $\int dx$})\\
    &= \int_\R dy \int_\R f(x,y)\, dx = \int_{\R^2} f(x,y)\, dxdy.  
  \end{align*}


  Next we go on to the non-linear case after J.~Schwartz\cite{Schw}. 
  For the Jacobian formula on $C_c(V)$, it is enough to show the validity for $g \in C_c^+(V)$,
  which in turn implies the case $C^+(V)$ because each $g \in C^+(V)$ is expressed in the form $g_n \uparrow g$ with $g_n \in C_c^+(V)$. 
  
  To establish the formula on $C_c^+(V)$, we need some notations in norm estimates.
  For a numerical vector $x = (x_1,\dots,x_d) \in \R^d$ and a real matrix $A = (a_{i,j})_{1 \leq i,j \leq d}$, set
\[
  \| x\|_\infty = \max_{1 \leq i \leq d}\{ |x_i| \},  
  \quad
  \|A\| = \max_{1 \leq i \leq d} \{ \sum_{j=1}^d |a_{i,j}|\}, 
\]
where $\| A\|$ is the operator norm relative to $\|\cdot\|_\infty$ and satisfies inequalities 
$\| A x\|_\infty \leq \| A\| \| x\|_\infty$, $\| AB\| \leq \| A \| \| B\|$.

\begin{Exercise}
Check these inequalities. 
\end{Exercise}

Let $[a,b]$ ($b_1-a_1 = \dots = b_d-a_d = 2r$) be a $d$-dimensional closed cube contained in $U$.
By the fundamental formula in calculus, we have 
\[
  \phi(x) - \phi(c) = \sum_{j=1}^d \int_0^1dt\, (x_j - c_j) \frac{\partial \phi}{\partial x_j}(tx + (1-t)c) 
\]
and then
\[
  \| \phi(x) - \phi(c)\|_\infty \leq \| x-c\|_\infty \max_{0 \leq t \leq 1}\| \phi'(tx+(1-t)c)\| 
\]
for $x,c \in [a,b]$. 

In particular, choosing $c = (a+b)/2$,
we see that $\phi([a,b])$ is included in the closed cube of center $\phi(c)$ and width $2r \| \phi'\|_{[a,b]}$, whence 
\[
  |\phi((a,b])| \leq \| \phi'\|_{[a,b]}^d (2r)^d = \| \phi'\|_{[a,b]}^d |(a,b]|, 
\]
where, for a subset $C \subset U$, 
$\| \phi'\|_C = \sup\{ \| \phi'(x)\|; x \in C\}$. 
Recall here that $\phi((a,b])$ is a Lebesgue integrable set (Example~\ref{OCI}).

Invoking the chain rule $(\phi'(c)^{-1}\phi)' = \phi'(c)^{-1} \phi'$,
the above estimate applied to a map $\phi'(c)^{-1}\phi:U \to \phi'(c)^{-1}(V)$ takes the form 
\[
  |\phi((a,b])| = |\det \phi'(c)|  |\phi'(c)^{-1} \phi((a,b])|
 \leq |\det \phi'(c)|  |\phi'(c)^{-1} \phi'\|_{[a,b]}^d |(a,b]|. 
\]

Let $f = g\circ \phi \in C_c^+(U)$and divide $(a,b]$ into a multiple partition $\Delta$
so that $(a,b]$ is a disjoint union of open-closed subcubes $(R_i)_{1 \leq i \leq m}$ with width $2r$. 
We then apply the above inequality for each $R_i$ with the center $\xi_i$ of $R_i$ as a sample point 
to have
\[
  I^1(f_{(\Delta],\xi}\circ \phi^{-1}) = \sum_i f(\xi_i)|\phi(R_i)|
  \leq
  \sum_i f(\xi_i) |\det \phi'(\xi_i)|  \|\phi'(\xi_i)^{-1} \phi'\|_{R_i}^d |R_i|. 
\]
Note here that $\phi(R_i) = R_i\circ\phi^{-1}$ as an indicator function and hence $|\phi(R_i)| = I^1(R_i\circ \phi^{-1})$.

Now, letting $m \to \infty$ so that $r \to 0$, 
$f_{\Delta,\xi}$ converges uniformly to $(a,b]f$ and the dominated convergence theorem gives
\begin{align*}
  \lim_{m \to \infty} I^1(f_{(\Delta],\xi}\circ \phi^{-1}) &= I^1(((a,b]f)\circ \phi^{-1})\\
  &= I^1(\phi((a,b]) (f\circ\phi^{-1})) = \int_{\phi((a,b])} f(\phi^{-1}(y))\, dy, 
\end{align*}
whereas, in the right hand side of the inequality,
$\phi'(\xi_i)^{-1} \phi'(x)$ converges to the identity matrix uniformly in $i$ and $x \in R_i$,
which is combined with the Cauchy-Riemann formula (Theorem~\ref{CRF}) to get 
\[
  \lim_{m \to \infty}  \sum_i f(\xi_i) |\det \phi'(\xi_i)|  \|\phi'(\xi_i)^{-1} \phi'\|_{R_i}^d |R_i|
  = \int_{(a,b]} f(x) |\det \phi'(x)|\, dx,  
\]
concluding that 
\[
  \int_{\phi((a,b])} g(y)\, dy
  \leq \int_{(a,b]} g(\phi(x)) |\det \phi'(x)|\, dx.
\]

Now express $U$ as a disjoint union of countably many dyadic cubes $(a,b]$ satisfying $[a,b] \subset U$
so that $f = \sum (a,b]f$.
Since $g \in C_c^+(V)$ is integrable, we can apply the dominated convergence theorem (or the dominated series convergence) to
the expression $\sum ((a,b]f)\circ \phi^{-1} = f\circ \phi^{-1} = g$ to obtain 
\begin{align*}
  \int_V g(y)\, dy = \sum_{(a,b]} \int_{\phi((a,b])} g(y)\, dy
  &\leq \sum_{(a,b]} \int_{(a,b]} g(\phi(x)) |\det \phi'(x)|\, dx\\
  &= \int_U g(\phi(x)) |\det \phi'(x)|\, dx. 
\end{align*}


Since the last integrand $h$ is in $C_c^+(U)$, $\phi^{-1}:V \to U$ is applied for $h$, together with
the chain rule $\phi'(\phi^{-1}(y)) (\phi^{-1})'(y) = (\phi\circ \phi^{-1})'(y) = \text{id}$, to obtain the reverse inequality 
\[
   \int_U g(\phi(x)) |\det \phi'(x)|\, dx \leq 
   \int_V g(y)\, dy, 
 \]
proving the Jacobian formula for $g \in C_c^+(V)$.  


\begin{Exercise}
  Check the integrability of $\phi((a,b]) g$ when $[a,b] \subset U$. Hint: Express $(a,b_n) \downarrow (a,b]$ and notice that $\phi(a,b_n)g$ is
  integrable. 
\end{Exercise}

\begin{Exercise} Provide the details of
\[
  \lim_{m \to \infty}  \sum_i f(\xi_i) |\det \phi'(\xi_i)|  \|\phi'(\xi_i)^{-1} \phi'\|_{R_i}^d |R_i|
  = \int_{(a,b]} f(x) |\det \phi'(x)|\, dx.  
\]
Hint: The right hand side is equal to 
$\displaystyle \lim_{m \to \infty} \sum_i f(\xi_i) |\det(\phi'(\xi)| |R_i^\circ|$.
\end{Exercise}

{\small
\begin{Remark}
  If we use the technique of partition of unity concerning open coverings,
  we can dispense with convergence theorems in Lebesgue integrals and complete the whole proof within Cauchy-Riemann integrals.  
\end{Remark}}

\begin{Example} Let $n=2$ and 
  $\phi: (0,\infty)\times (-\pi,\pi) \to \R^2 \setminus (-\infty,0]\times\{0\}$ be
  the polar coordinate transformation\index{polar coordinate transformation}
  $\phi(r,\theta) = (r\cos\theta,r\sin\theta)$. Here old variables are $(r,\theta)$ and we regard $(x,y)$ as
  new variables. Then
  \[
    \det(\phi'(r,\theta)) =
    \begin{vmatrix}
      \cos\theta & -r\sin\theta\\
      \sin\theta & r\cos\theta
    \end{vmatrix}
    = r
  \]
  and, if an open set $U \subset (0,\infty)\times (-\pi,\pi)$ is transformed into
  an open set $V \subset \R^2 \setminus (-\infty,0]\times\{0\}$ by $\phi$, the equality 
  \[
    \int_V g(x,y)\, dxdy = \int_U g(r\cos\theta,r\sin\theta)\,rdrd\theta
  \]
  holds for $g \in C^+(V)$. Thus, if $f \in C(V)$ satisfies $|f| \leq g$ with
  \[
    \int_U g(r\cos\theta,r\sin\theta)\,rdrd\theta < \infty, 
  \]
  then $f \in L^1(V)$ and
  \[
    \int_V f(x,y)\, dxdy = \int_U f(r\cos\theta,r\sin\theta)\,rdrd\theta. 
  \]
  Since $\R^2 \setminus V = (-\infty,0]\times \{ 0\}$ is a null set, $f$ is integrable as a function on $\R^2$ and we also have
  \[
    \int_{\R^2} f(x,y)\, dxdy = \int_0^\infty dr\, r \int_{-\pi}^\pi f(r\cos\theta,r\sin\theta)\,d\theta.
  \]
\end{Example}

\begin{figure}[h]
  \centering
 \includegraphics[width=0.8\textwidth]{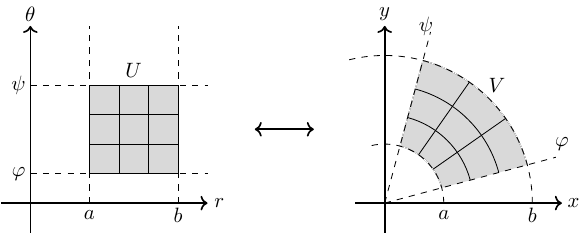}
 \caption{Polar Coordinates}
\end{figure}

\begin{Example}\label{polardisk}
  As a supplement to Example~\ref{punctured}, we have 
  \[
    \int_{x^2 + y^2 < 1} \frac{1}{(x^2+y^2)^{\alpha/2}}\, dxdy = 2\pi \int_0^1 \frac{r}{r^\alpha}\, dr =
    \begin{cases}
      1/(2-\alpha) &(\alpha < 2)\\
      \infty &(\alpha \geq 2)
    \end{cases}. 
\]
\end{Example}

\begin{Exercise}
  Express the integral
  \[
    \int_{x^2 + y^2 > 1} \frac{e^{-\beta(x^2+y^2)}}{(x^2+y^2 - 1)^\alpha}\, dxdy
  \]
  ($\alpha>0$, $\beta > 0$) in Example~\ref{singularity} by the Gamma function. 
\end{Exercise}

\begin{Example}\index{Gaussian integral}
  For $C = \int_{-\infty}^\infty e^{-t^2}\, dt$ we have 
\begin{align*}
  C^2 &= \int_{\R^2} e^{-(x^2+y^2)}\, dxdy 
  = \int_{(0,\infty)\times (-\pi,\pi)} e^{-r^2} r\, drd\theta\\
  &= \int_0^\infty e^{-r^2} r\, dr \int_{-\pi}^\pi d\theta = \pi \int_0^\infty e^{-r^2} d(r^2) = \pi, 
\end{align*}
showing $C = \sqrt{\pi}$ again. 
\end{Example}

As a popular application, we shall express the beta function in terms of the gamma function.

Recall that the gamma function\index{gamma function} is defined by
\[
  \Gamma(t) = \int_0^\infty x^{t-1} e^{-x}\, dx
  = 2\int_0^\infty x^{2t-1} e^{x^2}\, dx \quad (t>0), 
\]
which is a continuous replacement of factorial in the sense that $(t-1)! = \Gamma(t)$.
The \textbf{beta function}\index{beta function} is defined by a possibly improper integral 
\[
  B(s,t) = \int_0^1 x^{s-1} (1-x)^{t-1}\, dx
  \quad (s>0, t>0). 
\]

\begin{Exercise}
These improper integrals are well-defined. 
\end{Exercise}

\begin{Theorem}
  The beta function is expressed by
   \[
    B(s,t) = 2\int_0^{\pi/2} \cos^{2s-1}\theta \sin^{2t-1}\theta\, d\theta.
  \]
  and related to the gamma function by
  \[
    B(s,t) = \frac{\Gamma(s)\Gamma(t)}{\Gamma(s+t)}.
  \]
\end{Theorem}

\begin{proof}
  The expression of trigonometric integral is immediate from the variable change $x = \cos^2\theta$ ($0 \leq \theta \leq \pi/2$).

  We repeat the argument of Gaussian integral in polar coordinates. 
  \begin{align*}
    \Gamma(s)\Gamma(t) &= 4\int_0^\infty x^{2s-1}e^{-x^2} \int_0^\infty y^{2t-1} e^{-y^2}\, dy\\
                       &= 4 \int_{(0,\infty)\times (0,\infty)} x^{2s-1} y^{2t-1} e^{-(x^2+y^2)}\, dxdy\\
                       &= 4 \int_0^\infty dr\, r \int_0^{\pi/2} r^{2(s+t)-2} e^{-r^2} \cos^{2s-1}\theta \sin^{2t-1}\theta\, d\theta\\
    &= 2\int_0^\infty r^{2(s+t) - 1} e^{-r^2}\, dr B(s,t) = \Gamma(s+t) B(s,t). 
  \end{align*} 
\end{proof}

\begin{Exercise}[Dirichlet integral]\index{Dirichlet integral} 
  Let $D = \{ (x_1,\dots,x_n) \in \R^n; x_1 \geq 0, \dots, x_n \geq 0, x_1 + \dots + x_n \leq 1\}$ be an $n$-dimensional simplex. 
  For stricly positive reals $a_0, a_1,\dots,a_n$, show that
  \[
    \int_D x_1^{a_1-1} \dots x_n^{a_n-1}(1-x_1-\dots - x_n)^{a_0-1}\, dx_1\cdots dx_n
    = \frac{\Gamma(a_0)\Gamma(a_1) \cdots \Gamma(a_n)}{\Gamma(a_0 + a_1+ \dots + a_n)}. 
  \] 
\end{Exercise}

\section{Surface Integrals}
As another application of the Jacobian formula, we shall describe the curvilinear extent of a geometric object such as
the length of a curve or the area of a surface.

\begin{figure}[h]
  \centering
 \includegraphics[width=0.5\textwidth]{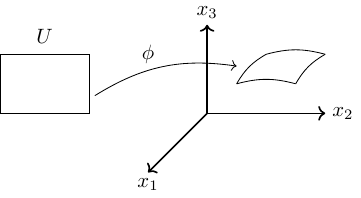}
 \caption{Parametrized Object}
\end{figure}

Let our geometric object $M \subset \R^d$ be parametrized by coordinates $u = (u_1,\dots,u_m) \in U$ in the form $x = \phi(u)$.
Here $U$ is an open subset of $\R^m$ and $\phi: U \to \R^d$ is a \index{smooth}smooth (i.e., continuously differentiable)
injective map satisfying $\text{rank}(\phi'(u)) = m$ ($u \in U$)
and $\phi(U) = M$. 

For a small rectangle $\Delta u = \Delta u_1 \times \dots \times \Delta u_m$ inside $U$, its image under $\phi$ is approximately
a parallelotope in $\R^d$ spanned by vectors 
\[
  |\Delta u_1| \partial_1\phi, \dots,
  |\Delta u_m|\partial_m \phi, 
  \quad
 \partial_i\phi = \frac{\partial\phi}{\partial u_i} \in \R^d
\]
with its $m$-dimensional volume given by $\sqrt{\det(\partial_i\phi|\partial_j\phi)}|\Delta u|$.

\begin{Exercise}
  Show that the $m$-dimensional volume of a parallelotope spanned by vectors $\xi_j \in \R^d$ ($1 \leq j \leq m \leq d$) is 
  $\sqrt{\det(\xi_i|\xi_j)}$. (See \cite[Theorem 6.2.16]{Sch} for example.)
\end{Exercise}

\begin{figure}[h]
  \centering
 \includegraphics[width=0.5\textwidth]{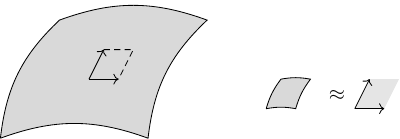}
 \caption{Linear Approximation}
\end{figure}

Thus it is reasonable to define the $m$-dimensional \textbf{extent}\index{extent} of $M$ by
\[
  \int_M |dx|_M = \int_U \sqrt{\det(\partial_i\phi|\partial_j\phi)}\, du
\]
with $\sqrt{\det(\partial_i\phi|\partial_j\phi)}$ called the \textbf{extent density}\index{extent density} of $\phi$.
\index{+surface@$\lvert dx\rvert_M$ surface measure of $M$}
The definition is fairly speculative but it certainly bears several desirable properties:
\begin{enumerate}
\item
  It correctly responds under scaling: For $r>0$, $rM$ is parametrized by $r\phi$ and 
  $\sqrt{\det(\partial_ir\phi|\partial_jr\phi)} = r^m \sqrt{\det(\partial_i\phi|\partial_j\phi)}$ shows 
  \[
    \int_U \sqrt{\det\Bigl(\frac{\partial r\phi}{\partial u_i}\Bigl| \frac{\partial r\phi}{\partial u_j}\Bigr.\Bigr)}\, du
    = r^m \int_U \sqrt{\det\Bigl(\frac{\partial\phi}{\partial u_i}\Bigl| \frac{\partial\phi}{\partial u_j}\Bigr.\Bigr)}\, du.  
  \]
  \item
It is invariant under Euclidean transformations in $\R^d$. 
Let $T: \R^d \to \R^d$ be a Euclidean transformation, then $TM$ is parametrized by $T\phi$ and
$\det(\partial_i(T\phi)|\partial_j(T\phi)) = \det(\partial_i\phi|\partial_j\phi)$ gives the invariance.
\item
It is independent of choices of parametrization. 
In fact, for another parametrization $V \ni v \mapsto \psi(v) \in M$ of $M$ with $V$ an open subset of $\R^m$, the chain rule 
\[
  \Bigl(\frac{\partial\phi}{\partial u_i}\Bigl| \frac{\partial\phi}{\partial u_j}\Bigr.\Bigr)
  = \sum_{k,l} \Bigl(\frac{\partial\phi}{\partial v_k} \Bigl| \frac{\partial\phi}{\partial v_l}\Bigr.\Bigr)
  \frac{\partial v_k}{\partial u_i} \frac{\partial v_l}{\partial u_j}
\]
gives\footnote{$\frac{dv}{du} = \det\frac{\partial v}{\partial u}$ is the Jacobian with
  $\frac{\partial v}{\partial u} = (\frac{\partial v_j}{\partial u_i})$ denoting the differential matrix.} 
\[
  \sqrt{\det\Bigl(\frac{\partial\phi}{\partial u_i}\Bigl| \frac{\partial\phi}{\partial u_j}\Bigr.\Bigr)}
  = \sqrt{\det\Bigl(\frac{\partial\phi}{\partial v_k} \Bigl| \frac{\partial\phi}{\partial v_l}\Bigr.\Bigr)}
  \, \Bigl|\frac{dv}{du}\Bigr|,  
\]
whence
\[
  \int_U \sqrt{\det\Bigl(\frac{\partial\phi}{\partial u_i}\Bigl| \frac{\partial\phi}{\partial u_j}\Bigr.\Bigr)}\, du 
  = \int_V \sqrt{\det\Bigl(\frac{\partial\phi}{\partial v_k} \Bigl| \frac{\partial\phi}{\partial v_l}\Bigr.\Bigr)}\, dv. 
\]
\end{enumerate}

\begin{Definition}
  The last property of extent density allows us to define the \textbf{surface integral}\index{surface integral}
  of a function $f$ on $M \subset \R^d$ in a coordinate-free fashion by
  \[
    \int_M f(x)\, |dx|_M
    = \int_U f(\phi(u)) \sqrt{\det\Bigl(\frac{\partial\phi}{\partial u_i}\Bigl| \frac{\partial\phi}{\partial u_j}\Bigr.\Bigr)}\, du,  
  \] 
  where the notation indicates that it is based on a measure $|\cdot|_M$ in $M$. 

Let $I_\phi$ be a preintegral on $C_c(U)$ defined by
\[
  I_\phi(g) = \int_U g(u) \sqrt{\det\Bigl(\frac{\partial\phi}{\partial u_i}\Bigl| \frac{\partial\phi}{\partial u_j}\Bigr.\Bigr)}\, du  
\]
with its Daniell extension denoted by $I_\phi^1: L^1(U,\phi) \to \R$.

Since parametrization-independence in the surface integral is based on the Jacobian formula,
the integrability of a function on $M$ (\textbf{surface-integrability}) 
has a meaning and the set $L^1(M)$ of surface-integrable functions turns out to be a linear lattice isomorphic to $L^1(U,\phi)$. 
\end{Definition}

\medskip
\begin{Example}~ 
  \begin{enumerate}
  \item For a smooth curve $C \subset \R^d$ parametrized by $x = \phi(t)$ ($a < t < b$) with $m=1$, 
    \[
      \int_C |dx|_C = \int_a^b \left| \frac{d\phi}{dt} \right|\, dt
    \]
    is the length of $C$.
  \item
    For a smooth surface $M \subset \R^d$ parametrized by $x = \phi(s,t)$ with $(s,t) \in U \subset \R^2$,
    \[
      \int_M |dx|_M = \int_U \sqrt{\Bigl(\frac{\partial \phi}{\partial s}\Bigl|\frac{\partial \phi}{\partial s}\Bigr.\Bigr)
        \Bigl(\frac{\partial \phi}{\partial t}\Bigl|\frac{\partial \phi}{\partial t}\Bigr.\Bigr)
        - \Bigl(\frac{\partial \phi}{\partial s}\Bigl|\frac{\partial \phi}{\partial t}\Bigr.\Bigr)^2}\, dsdt.
    \]
    When $d=3$ and $\phi$ is denoted by $\phi(s,t) = (x(s,t),y(s,t),z(s,t))$, the extent density takes the form 
    \begin{multline*}
      \Bigl|\frac{\partial\phi}{\partial s}\times \frac{\partial\phi}{\partial t}\Bigr| = \\
      \sqrt{\Bigl(\frac{\partial y}{\partial s}\frac{\partial z}{\partial t} - \frac{\partial z}{\partial s}\frac{\partial y}{\partial t}\Bigr)^2
        + \Bigl(\frac{\partial z}{\partial s}\frac{\partial x}{\partial t} - \frac{\partial x}{\partial s}\frac{\partial z}{\partial t}\Bigr)^2
        + \Bigl(\frac{\partial x}{\partial s}\frac{\partial y}{\partial t} - \frac{\partial y}{\partial s}\frac{\partial x}{\partial t}\Bigr)^2}.
    \end{multline*}
  \item
    Let $\varphi$ be a continuously differentiable function of $u \in U$ with $U$ an open subset of $\R^{d-1}$
    and consider a $(d-1)$-dimesional surface $M = \{ (\varphi(u),u); u \in U\}$ in $\R^{d}$ with $\phi(u) = (\varphi(u),u)$. Then
    \[
      \det (\partial_j\phi|\partial_k\phi) = \det (\phi'){}^t(\phi') = 1 + |\varphi'|^2
    \]
    by the Cauchy-Binet formula in Appendix~\ref{determinant}
    (or by a simple computation with rank-one operators) and the surface integral on $M$ is described by
    \[
      \int_M f(x)\, |dx|_M = \int_U f(\varphi(u),u)) \sqrt{1+ |\varphi'(u)|^2}\, du.
    \]
    \end{enumerate}
\end{Example}

\begin{Example}
  Consider a circle $(x-a)^2 + z^2 = b^2$ ($0 < b < a$) in the $xz$-plane and rotate it around the $z$-axis to get
  a torus\index{torus} $(\sqrt{x^2+y^2} - a)^2 + z^2 = b^2$. To compute the surface area, we parametrize its upper half by
  \[
    x = (a+b\sin\theta)\cos\varphi,\quad
    y = (a+b\sin\theta)\sin\varphi,\quad 
    z = b\cos\theta
  \]
  with
  \[
    0 \leq \varphi \leq 2\pi,\quad -\frac{\pi}{2} \leq \theta \leq \frac{\pi}{2}. 
  \]
  Then
  \[
    \frac{\partial(x,y,z)}{\partial(\varphi,\theta)}
    =
    \begin{pmatrix}
      -(a+b\sin\theta)\sin\varphi & b\cos\theta\cos\varphi\\
      (a+b\sin\theta)\cos\varphi & b\cos\theta\sin\varphi\\
      0 & -b\sin\theta\\
    \end{pmatrix}
  \]
  and
  \[
    {\rule{0pt}{5mm}}^t\left(\frac{\partial(x,y,z)}{\partial(\varphi,\theta)}\right)
    \left(\frac{\partial(x,y,z)}{\partial(\varphi,\theta)}\right)
    =
    \begin{pmatrix}
      (a+b\sin\theta)^2 & 0\\
      0 & b^2
    \end{pmatrix}
  \]
  shows that the density is $b(a+b\sin\theta)$. Thus the toral surface area is
  \[
    2b \int_{-\pi/2}^{\pi/2} (a+b\sin\theta)\, d\theta \int_0^{2\pi} d\varphi = 2\pi a 2\pi b. 
  \]   
\end{Example}

\begin{Exercise}
Compute the length of the coil $C \subset \R^3$: $\phi(t) = (a\cos t,b\sin t, bt)$ ($0 \leq t \leq \tau$). 
\end{Exercise}

\begin{Exercise}
  The $(d-1)$-dimensional extent of the simplex\index{simplex}
  $M = \{ x \in \R^d; x_1 \geq 0, \dots, x_d \geq 0, x_1 + \dots + x_d = 1\}$ in $\R^d$ is
  $\sqrt{d}/(d-1)!$. 
\end{Exercise}

\begin{Exercise}
  Assume that $\phi:U \to M \subset \R^d$ is a product of 
  an $m'$-dimensional parametrization $\varphi:U' \to M' \subset \R^{d'}$ and
  an $m''$-dimensional parametrization $\psi:U'' \to M'' \subset \R^{d''}$,
  i.e., $U = U'\times U''$, $M = M'\times M''$ and $\phi(u) = (\varphi(u'),\psi(u''))$ for
  $u = (u',u'') \in \R^{m'}\times \R^{m''}$. 

  Then
  $\int_M |dx|_M = \int_{M'} |dx'|_{M'}\, \int_{M''} |dx''|_{M''}$ with $(x',x'') \in \R^{d'}\times \R^{d''}$. 
\end{Exercise}


{\small
\begin{Remark}
Intuitively, a single coordinate parametrization is enough to almost cover $M$ by removing lower dimensional negligible parts.
\end{Remark}}

We shall now extend the construction so far for a single coordinate parametrization to the case of multiple parametrization. 

Assume that 
we are given a family of continuously differentiable one-to-one maps
$\phi_\alpha: U_\alpha \to \R^d$ ($U_\alpha$ being an open subset of $\R^m$) so that 
(i) $\text{rank}(\phi_\alpha'(u): \R^m \to \R^d) = m$ for $u \in U_\alpha$, 
(ii) $M = \bigcup_\alpha \phi_\alpha(U_\alpha)$ and (iii), if $\phi_\alpha(U_\alpha) \cap \phi_\beta(U_\beta) \not= \emptyset$,
the bijection\footnote{$\phi_\beta^{-1}\phi_\alpha$ is not a composite map but a single symbolic notation.}
$\phi_\beta^{-1}\phi_\alpha: \phi_\alpha^{-1}(\phi_\beta(U_\beta)) \to \phi_\beta^{-1}(\phi_\alpha(U_\alpha))$
defined by $\phi_\alpha(u) = \phi_\beta((\phi_\beta^{-1}\phi_\alpha)(u))$
($u \in \phi_\alpha^{-1}(\phi_\beta(U_\beta))$) is continuously differentiable.
(The geomtric object $M \subset \R^d$ is a so-called immersed submanifold.)

A one-to-one map $\varphi$ of an open subset $U$ of $\R^m$ into $M$ is then called
a \textbf{coordinate chart}\index{coordinate chart} of $M$ if
both $\varphi^{-1}(\phi_\alpha(U_\alpha)) = \{ u \in U; \varphi(u) \in \phi_\alpha(U_\alpha)\}$ and
$\phi_\alpha^{-1}(\varphi(U)) = \{ u \in U_\alpha; \phi_\alpha(u) \in \varphi(U)\}$ are open in $\R^m$ with 
the associated bijection $\phi_\alpha^{-1}\varphi: \varphi^{-1}(\phi_\alpha(U_\alpha)) \to \phi_\alpha^{-1}(\varphi(U))$
as well as its inverse map continuously differentiable for each $\phi_\alpha:U_\alpha \to \R^d$.
Here $\phi_\alpha^{-1}\varphi$ is defined by $\phi_\alpha((\phi_\alpha^{-1}\varphi)(u)) = \varphi(u)$ ($u \in \varphi^{-1}(\phi_\alpha(U_\alpha))$).  
  
Thus each map $U_\alpha \ni u \mapsto \phi_\alpha(u) \in M$ is a coordinate chart and,
if $\psi:V \to M$ is another coordinate chart, 
the \textbf{coordinate transformation}\index{coordinate transformation}
$\psi^{-1}\varphi: \varphi^{-1}(\psi(V)) \ni u \mapsto v \in \psi^{-1}(\varphi(U))$
defined by $\varphi(u) = \psi(v)$ is a continuously differentiable bijection
from an open set $\varphi^{-1}(\psi(V))$ in $\R^m$ onto another open set $\psi^{-1}(\varphi(U))$ in $\R^m$.

\begin{figure}[h]
  \centering
 \includegraphics[width=0.5\textwidth]{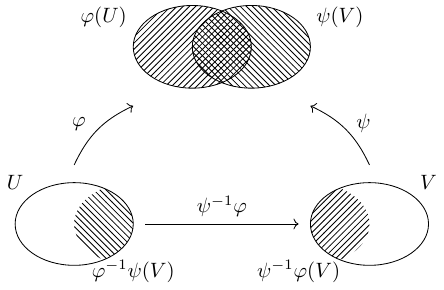}
 \caption{Coordinate Transformation}
\end{figure}



  


Since open sets are Lebesgue measurable, so are their cuts and unions.
Moreover Lebesgue measurable sets are preserved under coordinate transformations (Proposition~\ref{measurable-null}), 
which enables us to cut and union $L^1(U,\varphi)$ for
various coordinate charts $\varphi:U \to M$ to obtain a single space $L^1(M)$: The detailed construction is as follows. 


Consider a function $f$ on $M$ which admits a finitely many coordinate charts $(\phi_\alpha: U_\alpha \to M \subset \R^d)$ satisfying
$\phi_\alpha(U_\alpha)f \in L^1(\phi_\alpha(U_\alpha))$ for each $\alpha$ and $(\bigcup_\alpha \phi_\alpha(U_\alpha))f = f$.

\begin{Lemma}~ \label{cutjoin}
  \begin{enumerate}
    \item
      We can find measurable sets $A_\alpha \subset U_\alpha$ so that $\bigcup_\alpha \phi_\alpha(U_\alpha) = \bigsqcup \phi_\alpha(A_\alpha)$
      (a disjoint union).
  \item
Let $\phi:U \to \bigcup_\alpha \phi_\alpha(U_\alpha) \subset \R^d$ be a coordinate chart of $M$. 
Then $\phi(U)f \in L^1(\phi(U))$.
\end{enumerate}
\end{Lemma}

\begin{proof}
  (i) Write $\alpha =1,2,\dots,l$ and let $A_\alpha$ be defined by
  \begin{align*}
    A_\alpha
    &= \phi_\alpha^{-1}\Bigl(\bigl(\phi_1(U_1) \cup \dots \cup \phi_\alpha(U_\alpha)\bigr)
      \setminus \bigl((\phi_1(U_1) \cup \dots \cup \phi_{\alpha-1}(U_{\alpha-1})\bigr)\Bigr)\\
    &= \phi_\alpha^{-1}\Bigl( \phi_\alpha(U_\alpha)
      \setminus \bigl((\phi_1(U_1) \cup \dots \cup \phi_{\alpha-1}(U_{\alpha-1})\bigr)\Bigr)\\
    &= U_\alpha 
      \setminus \phi_\alpha^{-1}\bigl((\phi_1(U_1) \cup \dots \cup \phi_{\alpha-1}(U_{\alpha-1})\bigr), 
   \end{align*}
  which is Lebesgue measurable as a difference of open subsets.

(ii)
Since $U_\alpha \cap \phi_\alpha^{-1}\phi(U)$ is Lebesgue measurable as an open subset, so is $A_\alpha \cap \phi_\alpha^{-1}\phi(U)$,
whence $A_\alpha \cap \phi_\alpha^{-1}\phi(U) (f\circ \phi_\alpha) \sqrt{\det(\partial_i\phi_\alpha|\partial_j\phi)}$ belongs to $L^1(U_\alpha)$
as a cut of $(f\circ \phi_\alpha) \sqrt{\det(\partial_i\phi_\alpha|\partial_j\phi_\alpha)}$ by a Lebesgue measurable set and then,
by the Jacobian formula applied to
$\phi^{-1}\phi_\alpha: U_\alpha \cap \phi_\alpha^{-1}\phi(U) \to U \cap \phi^{-1}\phi_\alpha(U_\alpha)$,
$(U \cap \phi^{-1}\phi_\alpha(A_\alpha)) (f\circ\phi) $ is Lebesgue integrable for each $\alpha$.
Consequently
\[
  U (f\circ\phi) \sqrt{\det(\partial_i\phi|\partial_j\phi)} = \sum_\alpha (U \cap \phi^{-1}\phi_\alpha(A_\alpha))
  (f\circ\phi) \sqrt{\det(\partial_i\phi|\partial_j\phi)}
\]
belongs to $L^1(U)$, i.e., $\phi(U)f \in L^1(\phi(U))$. 
\end{proof}

Let $L(M)$ be the totality of functions considered so far. Clearly $L(M)$ is closed under lattice operations and in fact a linear lattice
in view of the above lemma.

\begin{Exercise}
Show that $L(M)$ is a linear space. 
\end{Exercise}

For $f \in L(M)$, choose $\varphi_\alpha:U_\alpha \to M$ as before and measurable sets $A_\alpha \subset U_\alpha$
so that $\bigcup_\alpha \varphi_\alpha(U_\alpha) = \bigsqcup_\alpha \varphi_\alpha(A_\alpha)$ (Lemma~\ref{cutjoin} (i)).
A linear functional $I(f)$ of $f \in L(M)$ is then well-defined by
\[
  I(f) = \sum_\alpha I_{\varphi_\alpha}(A_\alpha (f\circ \varphi_\alpha)).  
\]
In fact, for another choice $\psi_\beta:V_\beta \to M$ with $B_\beta \subset V_\beta$ covering $f$, 
\begin{align*}
  \sum_\alpha I_{\varphi_\alpha}(A_\alpha (f\circ \varphi_\alpha))
  &= \sum_{\alpha,\beta} I_{\varphi_\alpha}((A_\alpha \cap \varphi_\alpha^{-1}(\psi_\beta(B_\beta))) (f\circ \varphi_\alpha))\\
  &= \sum_{\alpha,\beta} I_{\psi_\beta}((\psi_\beta^{-1}(\varphi_\alpha(A_\alpha)) \cap B_\beta) (f\circ \psi_\beta))\\
  &= \sum_\beta I_{\psi_\beta}(B_\beta (f\circ \psi_\beta)).  
\end{align*}

The linear functional $I(f)$ is a preintegral because $f_n \downarrow 0$ for $f_n \in L(M)$ implies
\[
  \lim_{n \to \infty} I(f_n) =
  \lim_{n \to \infty} \sum_\alpha I_{\varphi_\alpha}(A_\alpha (f_n\circ \varphi_\alpha))
 = \sum_\alpha \lim_{n \to \infty} I_{\varphi_\alpha}(A_\alpha (f_n\circ \varphi_\alpha)) = 0. 
\]

Let $I^1: L^1(M) \to \R$ be the Daniell extension of $I$ on $L(M)$, which contains $L^1(\varphi(U))$ as a linear sublattice for
each coordinate chart $\varphi:U \to M$ in such a way that
$I^1(f) = I_\varphi^1(f\circ\varphi)$ ($f \in L^1(\varphi(U))$) with $I^1(f)$ reasonably denoted by
\[
  \int_M f(x)\, |dx|_M.
\]

\begin{Exercise}
  Let $M \subset \R^d$ be the product of $M' \subset \R^{d'}$ and $M'' \subset \R^{d''}$.
  Then, for $f \in C_c(M)$, functions
  \[
    x' \mapsto \int_{M''} f(x',x'')\, |dx''|_{M''}, \quad
    x'' \mapsto \int_{M'} f(x',x'')\, |dx'|_{M'}
  \]
  belong to $C_c(M')$ and $C_c(M'')$ respectively for which the repeated integral formula holds:
  \[
    \int_M f(x)\, |dx|_M = \int_{M'} |dx'|_{M'} \int_{M''} f(x',x'')\, |dx''|_{M''}. 
  \] 
\end{Exercise}

\bigskip
\noindent 
Density formula: The following is known as a smooth version of the \textbf{coarea formula}\index{coarea formula} in geometric measure theory.


Let $\psi: D \ni x \mapsto v \in \R^n$ ($D \subset \R^d$ being an open set) 
be a \index{submersion}submersion\footnote{Namely, $\psi$ is continuously differentiable with $\text{rank}(\psi'(x)) = n$ everywhere.}
and $M$ be a \textbf{level set}\index{level set} $[\psi = v]$ of $\psi$ at $v \in \R^n$. 
Let $f \in C_c(D)$ be localized in a neighborhood of a point $a \in M \subset \R^d$.
Thanks to the inverse mapping theorem, after a suitable permutation of coordinates of $x$,
we may assume that $x \mapsto (u,v)$ with $u = (x_1,\dots,x_m)$ and
$v = \psi(x)$ is a local diffeomorphism in a neighborhood of $a$ ($m+n = d$).
Here \textbf{diffeomorphism}\index{diffeomorphism} is synonymous with smooth change-of-variables. 

Then the inverse diffeomorphism is of the form $(u,v) \mapsto x = (u, \varphi(u,v))$ and their differentials are given by
\[
  \frac{\partial x}{\partial(u,v)} =
                                     \begin{pmatrix}
                                       1_m & 0\\
                                       \frac{\partial \varphi}{\partial u} & \frac{\partial\varphi}{\partial v}
                                     \end{pmatrix}, \quad 
    \frac{\partial (u,v)}{\partial x} =
  \begin{pmatrix}
    1 & & 0 &0 &\dots &0\\
      & \ddots & & \vdots & \ddots & \vdots\\
    0 & & 1 & 0&\dots &0\\
    \frac{\partial v}{\partial x_1} & \dots & \frac{\partial v}{\partial x_m}
            & \frac{\partial v}{\partial x_{m+1}} &\dots & \frac{\partial v}{\partial x_{m+n}} 
  \end{pmatrix}. 
\]
Since these are inverses of each other, we have
\[
   \frac{\partial \varphi}{\partial v} 
  \begin{pmatrix}
  \frac{\partial v}{\partial x_{m+1}} &\dots & \frac{\partial v}{\partial x_{m+n}} 
  \end{pmatrix} = 1_n,\quad 
  \frac{\partial\varphi}{\partial u} 
  + \frac{\partial\varphi}{\partial v}
  \begin{pmatrix}
   \frac{\partial v}{\partial x_1} & \dots & \frac{\partial v}{\partial x_m} 
  \end{pmatrix} = 0. 
  \]
Here $1_m$ and $1_n$ denote identity matrices of size $m$ and $n$ respectively.

As a local parametrization of level sets $[\psi = v] \subset \R^d$ ($v$ moving in a small open subset $V \subset \R^n$),
we can take one of the form $U \ni u \mapsto x = (u, \varphi(u,v)) \in \R^d$
(with $U$ a neighborhood of $a \in \R^m$ and $\varphi(u,v)$ a continuously differentiable function of $(u,v)$)
so that the extent density is given by
\[
  \sqrt{\det\bigl(\delta_{i,j} + (\partial_i \varphi| \partial_j\varphi)\bigr)},
  \quad
  \partial_i\varphi = \frac{\partial\varphi}{\partial u_i}(u,v) 
\]
and the surface integral of $f$ on $[\psi = v] \subset D$ by 
\[
  \int_{[\psi = v]} f(x)\, |dx|_{[\psi=v]} = 
  \int_U f(u,\varphi(u,v)) \sqrt{\det\bigl(\delta_{i,j} + (\partial_i \varphi| \partial_j\varphi)\bigr)}\, du.
\]
\begin{Lemma}
  Let $A$ be an $m\times m$ invertible matrix, $C$ be an $n\times n$ invertible matrix and $B$ be an $n\times m$ matrix. 
  We set $G = - C^{-1}BA^{-1}$ so that
    \[
    \begin{pmatrix}
      A & 0\\
      B & C
    \end{pmatrix}^{-1}
    =
    \begin{pmatrix}
      A^{-1} & 0\\
      G & C^{-1}
    \end{pmatrix}.
  \]

  Then
  \[
    \det(A)^{-2} \det({}^tAA + {}^tBB) = \det(C)^2 \det(G\,{}^tG + C^{-1\,}{}^tC^{-1}). 
  \]
\end{Lemma}

\begin{proof} Just compute as follows: 
  \begin{align*}
    \det(A)^{-2} \det({}^tAA + {}^tBB) &= \det(1_m + {}^t(BA^{-1}) (BA^{-1}))\\
                                       &= \det(1_n + (BA^{-1})\,{}^t(BA^{-1}))\\
    &= \det(1_n + (-CG)\,{}^t(-CG))\\
&=  \det(C)^2 \det(G\,{}^tG + C^{-1\,}{}^tC^{-1}).
  \end{align*}
Here Sylvester's formula (Appendix~\ref{determinant}) is used in the second line.   
\end{proof}

We apply the above lemma for $A = 1_m$, $B = \frac{\partial \varphi}{\partial u}$ and $C = \frac{\partial \varphi}{\partial v}$
with $(G\ C^{-1}) = \frac{\partial \psi}{\partial x}$ to get 
\[
  \det\Bigl(\delta_{i,j} + \bigl(\frac{\partial\varphi}{\partial u_i} | \frac{\partial\varphi}{\partial u_{j}}\bigr)\Bigr)
  =  \det\Bigl(\frac{\partial\varphi}{\partial v}\Bigr)^2 \det(\psi_\imath'|\psi_\jmath'),
  \quad
 (\psi_\imath'|\psi_\jmath') = \sum_{k=1}^d \frac{\partial \psi_\imath}{\partial x_k} \frac{\partial \psi_\jmath}{\partial x_k}, 
\]
which is used to see 
\begin{align*}
  \int_{\psi^{-1}(V)} f(x)
  &\sqrt{\det(\psi_\imath'|\psi_\jmath')}\, dx\\
  &= \int_{U\times V} f(u,\varphi(u,v)) \sqrt{\det(\psi_\imath'|\psi_\jmath')} \Bigl|\det\Bigl(\frac{\partial x}{\partial(u,v)}\Bigr)\Bigr|\, dudv\\
  &= \int_{U\times V} f(u,\varphi(u,v)) \sqrt{\det(\psi_\imath'|\psi_\jmath')} \Bigl|\det\Bigl(\frac{\partial \varphi}{\partial v}\Bigr)\Bigr|\, dudv\\
  &= \int_V dv\,
    \int_U f(u,\varphi(u,v))
    \sqrt{\det\Bigl(\delta_{i,j} + \bigl(\frac{\partial\varphi}{\partial u_i} | \frac{\partial\varphi}{\partial u_{j}}\bigr)\Bigr)}\, du\\
  &= \int_V dv\, \int_{[\psi=v]} f(x)\, |dx|_{[\psi = v]}. 
\end{align*}

Finally this localized identity is patched up globally, 
this time by a partition of unity\footnote{A geometric form of Fubini theorem can be also used.} (Proposition~\ref{pou}), to have the following.

\begin{Theorem}\label{coarea}
Given a submersion $\psi: \R^d \supset D \ni x \mapsto \psi(x) \in \R^n$ and a function $f \in C_c(D)$, we have 
  \[
    \int_D f(x) \sqrt{\det(\psi_\imath'|\psi_\jmath')}\, dx
    = \int_{\psi(D)} dv\, \int_{[\psi=v]} f(x)\, |dx|_{[\psi = v]}. 
  \]
\end{Theorem}

\begin{proof}
By the local formula, each point $a \in [f]$ has an open neighborhood $W$ such that the global forumla holds 
if $f$ belongs to $C_c(W) \subset C_c(D)$. From the finite covering property, we can find a finitely many such open sets $W_\alpha$
so that $[f] \subset \bigcup W_\alpha$. We apply the partition of unity to this covering to get $h_\alpha \in C_c(W_\alpha)$ satisfying
$\sum_\alpha h_\alpha = 1$ on $[f]$.

Then $f_\alpha = h_\alpha f \in C_c(W_\alpha)$ is summed to be $f$ and we have
\begin{align*}
  \int_D f(x) \sqrt{\det(\psi_\imath'|\psi_\jmath')}\, dx
  &= \sum_\alpha \int_D f_\alpha(x) \sqrt{\det(\psi_\imath'|\psi_\jmath')}\, dx\\    
  &= \sum_\alpha \int_{\psi(D)} dv\, \int_{[\psi=v]} f_\alpha(x)\, |dx|_{[\psi = v]}\\
  &= \int_{\psi(D)} dv\, \int_{[\psi=v]} f(x)\, |dx|_{[\psi = v]}. 
\end{align*}
\end{proof}

\begin{Corollary} Let $M$ be a level set $[\psi = v]$ of $\psi$. Then
  \[
    \lim_{V \to v}  \frac{1}{|V|} \int_{\psi^{-1}(V)} f(x) \sqrt{\det(\psi_\imath'|\psi_\jmath')}\, dx = \int_M f(x)\, |dx|_M. 
  \] 
\end{Corollary}

{\small
\begin{Remark}
    By rewriting surface integrals in terms of Hausdorff measure,
    a further generalization is known as the coarea formula\index{coarea formula}
  in geometric measure theory (see \cite{Fe}).
\end{Remark}}

\begin{Example}
  Let $\psi: D=\R^d \to \R^n$ be a linear map,
  which is identified with an $n\times d$ matrix, and assume that the cut of $\psi$ by the last $n$ columns
  is invertible as an $n\times n$ matrix. 
  Then local coordinates $(u,v) \in \R^{m+n}$ satisfying
  \[
    \begin{pmatrix}
      u\\
      v 
    \end{pmatrix}
    =
  \left(
    \begin{array}{c c}
      I_m & 0\\
\multicolumn{2}{c}{\psi}
    \end{array}
  \right)
  x
  \iff
    x =
    \begin{pmatrix}
      I_m & 0\\
      B & C
    \end{pmatrix}
    \begin{pmatrix}
      u\\
      v
    \end{pmatrix}
    \ \text{with}\ 
    \varphi(u,v) = Bu + Cv    
  \]
  provides global one and the equality of 
  \[
    \int_D f(x) \sqrt{\det((\psi_i'|\psi_j'))}\, dx = \int_{\R^d} f(x) \sqrt{\det(\psi {}^t\psi)}\, dx 
  \]
  and
  \[
    \int_{\psi(D)} dv \int_{[\psi=v]} F(x)\, |dx|_{[\psi=v]}
    = \int_{\R^n} dv \int_{\R^m} f(u,\varphi(u,v)) |\det(C)| \sqrt{\det(\psi {}^t\psi)}\, du 
  \]
  is reduced to the identity
  \[
    \int_{\R^d} f(x)\, dx = |\det(C)| \int_{\R^n} dv \int_{\R^m} f(u,\varphi(u,v))\, du,  
  \]
  which is nothing but a combination of the (linear) Jacobian formula and repeated integrals.
\end{Example}

\begin{Example}
  Let $\psi(x) = |x|$ for $0 \not= x \in \R^d$ ($D = \R^d \setminus\{0\}$, $n=1$). Then $\psi'(x) = \frac{x}{|x|}$ and
  \begin{align*}
    \int_{D} f(x)\, dx &= \int_0^\infty dr \int_{rS^{d-1}} f(x)\, |dx|_{rS^{d-1}}\\
                       &= \int_0^\infty dr\, r^{d-1} \int_{S^{d-1}} f(r\omega)\, |d\omega|_{S^{d-1}}. 
  \end{align*}
  Now, for the choice $f(x) = e^{-|x|^2}$,
  \[
    \int_{\R^d} e^{-|x|^2}\, dx = \int_D e^{-|x|^2}\, dx = |S^{d-1}| \int_0^\infty r^{d-1} e^{-r^2}\, dr. 
  \]
  
  The left hand side is equal to $\pi^{d/2}$ as a multiple Gaussian integral and the integral in the right hand side is
  expressed in terms of the gamma function by $\Gamma(d/2)/2$, resulting in the \textbf{spherical integral}\index{spherical integral}
  \[
    |S^{d-1}| = \frac{2\pi^{d/2}}{\Gamma(d/2)}. 
  \]
  \index{+spherical@$\lvert S^{d-1} \rvert$ spherical integral}

  With this formula in hand, $f(x) = e^{-|x|^\beta}/|x|^\alpha$ for $\alpha \in \R$ and $\beta>0$ is then calculated as follows: 
  \begin{align*}
    \int_{\R^d} \frac{e^{-|x|^\beta}}{|x|^\alpha}\, dx = |S^{d-1}| \int_0^\infty r^{d-1} \frac{e^{-r^\beta}}{r^\alpha}\, dr
          &= |S^{d-1}| \frac{1}{\beta} \int_0^\infty s^{-1 + (d-\alpha)/\beta} e^{-s}\, dt\\
          &=
             \begin{cases}
             \frac{2\pi^{d/2}}{\beta \Gamma(d/2)} \Gamma( \frac{d-\alpha}{\beta})
             &(\alpha < d),\\
             \infty &(\alpha \geq d). 
             \end{cases}
  \end{align*} 
\end{Example}

\begin{Exercise}
Check that $|S^0| = 2$, $|S^1| = 2\pi$ and $|S^2| = 4\pi$. 
\end{Exercise}

\begin{Exercise}
  Compute the volume $V_d$ of the unit ball $\{ x \in \R^d; |x| \leq 1\}$ in $\R^d$ and show that
  $V_d \sim (2\pi e/d)^{d/2}/\sqrt{\pi d}$ as $d \to \infty$. 
\end{Exercise}


{\small
\begin{Remark}
Square-rooted determinant densities are closely related to the Jacobian. \index{Jacobian}
To see this, consider a smooth map $\varphi: U \to \R^n$ defined on an open set $U \subset \R^m$
with $\varphi'(x)$ a matrix-valued function of size $n\times m$.
Then $\det(\partial_i\varphi|\partial_j\phi) = \det({}^t\varphi'(x) \varphi'(x))$ and 
$\det(\varphi_i'(x)|\varphi_j'(x)) =  \det(\varphi'(x)\, {}^t\varphi'(x))$,
whence these coinside by Sylvester's formula (Appendix~\ref{determinant}).

When $m=n$, they are reduced to $\det(\varphi'(x))^2$.
In accordance with this fact, their square roots are also called the Jacobian of $\varphi$.
\end{Remark}} 

We here restrict ourselves to the case $n=1$;
$\psi$ is a scalar function and the level set $[\psi=v]$ is a hypersurface\index{hypersurface} in $\R^d$.
In the local coordinate expression
\[
  \int_{[\psi=v]} f(x)\, |dx|_{\psi=v} = \int_U f(x) \left| \frac{\partial\varphi}{\partial v}\right|  |\psi'(x)|\, du
  \quad
  (x = (u,\varphi(u,v))), 
\]
notice that the density function is equal to the norm of a vector 
\[
(-\frac{\partial \varphi}{\partial u}, 1) = \frac{\partial \varphi}{\partial v} \psi' = \frac{1}{\frac{\partial \psi}{\partial x_d}} \psi', 
\]
which is normal to the hypersurface $[\psi=v]$ at the point $x = (u,\varphi(u,v)) \in [\psi = v]$.
With the aid of the normal unit vector\index{normal vector}
\[
  \bm{n}(x) = \frac{1}{|\psi'(x)|} \psi'(x) 
\]
pointing to the direction of increasing $\psi$, we introduce
a vector-valued measure (so-called \textbf{surface element})\index{surface element} by
\[
  dx_{[\psi=v]} = \bm{n}(x) |dx|_{[\psi=v]}
\]
and, for a continuous vector field $F(x) \in \R^d$ ($x \in D$) of compact support,
define the \textbf{surface integral}\footnote{Also called the \textbf{flux}\index{flux} of a vector field $F$
through the hypersurface $[\psi=v]$.}\index{surface integral} of $F$ on the hypersurface $[\psi=v]$ by
\[
  \int_{[\psi=v]} F(x)\cdot dx_{[\psi=v]} = \int_{[\psi=v]} F(x)\cdot \bm{n}(x)\, |dx|_{[\psi=v]}.
\]
In local coordinates,
\[
  \bm{n}(x) = \frac{\epsilon}{\sqrt{1 + |\partial\varphi/\partial u|^2}} \Bigl(-\frac{\partial\varphi}{\partial u}, 1\Bigr)
  \quad\text{with}\quad 
  \epsilon = \frac{\partial\varphi}{\partial v}/\left|\frac{\partial\varphi}{\partial v}\right| \in \{ \pm 1\}
\]
and 
\[
  \int_{[\psi=v]} F(x)\cdot dx_{[\psi=v]} = \int_U F(x)\cdot\psi'(x) \left|\frac{\partial \varphi}{\partial v}\right|\, du
    \quad
  (x = (u,\varphi(u,v))).
\]
When $F(x) = f(x) \bm{n}(x)$, this is reduced to the surface integral of the scalar function $f$.

\begin{figure}[h]
  \centering
 \includegraphics[width=0.4\textwidth]{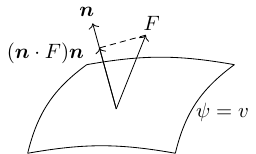}
 \caption{Flux}
\end{figure}

Now assume that $F$ is continuously differentiable on $D$. 
If the compact support of $F$ is contained in an open set $W \subset \R^d$ for which $W \simeq U\times V$ by a local coordinate description
of $\psi$ with $U \subset \R^{d-1}$ an open rectangle and $V \subset \R$ an open interval, we claim 
\[
  \frac{d}{dv} \int_{[\psi=v]} F(x)\cdot dx_{[\psi=v]} = \int_{[\psi=v]} \frac{f(x)}{|\psi'(x)|}\, |dx|_{[\psi=v]}.
\]
Here $f$ is the \textbf{divergence}\index{divergence} of $F$: 
\[
 f(x) =  \text{div} F = \sum_{i=1}^d \frac{\partial F_i}{\partial x_i}.
\]

In fact, in local coordinates, the formula takes the form
\[
  \frac{d}{dv} \int_U \epsilon F\cdot(-\frac{\partial\varphi}{\partial u},1)\, du
  = \int_U \epsilon f(x) \frac{\partial\varphi}{\partial v}\, du
  \quad
  (x = (u,\varphi(u,v))). 
\]
From the chain rule relation
\begin{align*}
  \frac{\partial}{\partial v} \left( F\cdot(-\frac{\partial\varphi}{\partial u},1) \right) 
  &= \frac{\partial}{\partial v} \left( - \sum_{i=1}^{d-1} F_i\frac{\partial\varphi}{\partial u_i} + F_d \right)\\
  &= - \sum_{i=1}^{d-1} \frac{\partial F_i}{\partial x_d} \frac{\partial\varphi}{\partial v} \frac{\partial\varphi}{\partial u_i}
  - \sum_{i=1}^{d-1} F_i \frac{\partial^2\varphi}{\partial v \partial u_i} + \frac{\partial F_d}{\partial x_d} \frac{\partial\varphi}{\partial v}\\
  &= - \sum_{i=1}^{d-1} \left( \frac{\partial F_i}{\partial x_d} \frac{\partial\varphi}{\partial v} \frac{\partial\varphi}{\partial u_i}
    + F_i \frac{\partial^2\varphi}{\partial v \partial u_i} + \frac{\partial F_i}{\partial x_i} \frac{\partial\varphi}{\partial v} \right)
  + f\frac{\partial\varphi}{\partial v}\\
  &= - \sum_{i=1}^{d-1} \frac{\partial}{\partial u_i} (F_i \frac{\partial\varphi}{\partial v}) + f\frac{\partial\varphi}{\partial v}, 
\end{align*}
the difference of the local formula therefore amounts to
\[
   \sum_{i=1}^{d-1} \int_U \frac{\partial}{\partial u_i}(F_i \frac{\partial\varphi}{\partial v})\, du, 
\]
which vanishes by repeated integral expressions in view of the fact that
$F_i \frac{\partial\varphi}{\partial v}$ vanishes at the boundary of $U$:
\begin{align*}
  \int_U \frac{\partial}{\partial u_i}(F_i \frac{\partial\varphi}{\partial v})\, du
  &=
 \int_{U_i} (du)_i \int_{a_i}^{b_i} \frac{\partial}{\partial u_i}(F_i \frac{\partial\varphi}{\partial v})\, du_i\\
  &= \int_{U_i}
    \left(\left.F_i \frac{\partial\varphi}{\partial v}\right|_{u_i=b_i}
    - \left.F_i \frac{\partial\varphi}{\partial v}\right|_{u_i=a_i} \right)\, (du)_i = 0.
\end{align*}
Here $U = (a_1,b_1)\times \dots\times (a_{d-1},b_{d-1})$,
\[
  U_i = (a_1,b_1)\times \dots \times (a_{i-1},b_{i-1})\times (a_{i+1},b_{i+1})\times \dots \times (a_{d-1},b_{d-1})
\]
and $(du)_i = du_1\cdots du_{i-1} du_{i+1}\cdots du_{d-1}$. 

The local formula is now glued together by a partition of unity (cf.~the proof of Theorem~\ref{coarea}) to obtain the global formula.  

\begin{Theorem} Let $\psi: D \to \R$ be a submersion and
  $F \in C_c^1(D,\R^d)$ be a continuously differentiable vector field of compact support on $D$.
  Then
 \[
   \frac{d}{dv} \int_{[\psi=v]} F(x)\cdot dx_{[\psi=v]} = \int_{[\psi=v]} \frac{\text{div}\, F(x)}{|\psi'(x)|}\, |dx|_{[\psi=v]}
 \]
   for $v \in \psi(D)$.  
\end{Theorem}

Combined with the coarea formula (Theorem~\ref{coarea}), we have the following.

\begin{Corollary}[Divergence Theorem\footnote{A flow version of divergence theorem can be found in \cite{Sa,St}. }]\label{div}
  Let $[a,b] \subset \psi(D)$. Then
  \[
    \int_{[a \leq \psi \leq b]} \text{div}\, F(x)\, dx = \int_{[\psi=b]} F(x)\cdot dx_{[\psi=b]} - \int_{[\psi=a]} F(x)\cdot dx_{[\psi=a]}. 
  \]
\end{Corollary}

\begin{figure}[h]
  \centering
 \includegraphics[width=0.3\textwidth]{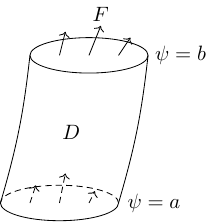}
 \caption{Flux Flow}
\end{figure}

As a limit case, consider the situation where $\psi(D) = (0,c)$ and $[\psi=v]$ approaches to a closed set
$[\psi=0]$ as $v \to +0$ in such a way that 
$[\psi=0] \subset \partial D$, 
$[0 \leq \psi < b] = [\psi=0] \sqcup [\psi < b]$ is an open set of $\R^d$ for any $0 < b < c$ and
\[
  \lim_{v \to +0} \int_{[\psi=v]} f(x)\, |dx|_{[\psi=v]}= 0
  \quad
  (f \in C_c^+([\psi \geq 0])). 
\]
The above flow version \index{flow version} is then filled with the inner boundary $[\psi=0]$ to get the boundary version \index{boundary version}
\[
  \int_{[0 \leq \psi < b]} \text{div}\, F(x)\, dx = \int_{[\psi=b]} F(x)\cdot dx_{\partial \widetilde D}
  \quad
 (F \in C_c^1([\psi \geq 0],\R^d)).  
\]

\begin{Example}[Cylinder]
  Let $x = (x',x'') \in \R^{d'}\times \R^{d''}$ with $d'+d'' = d$, $D = \{ x \in \R^d; |x'|>0\}$ and 
  $\psi(x) = |x'|$.

  Then $\psi(D) = (0,\infty)$ and the inner boundary $[\psi=0] = \{0\}\times \R^{d''}$ fills up $D$ to
  get the whole space $\R^d$. 
\end{Example}

\begin{figure}[h]
  \centering
 \includegraphics[width=0.3\textwidth]{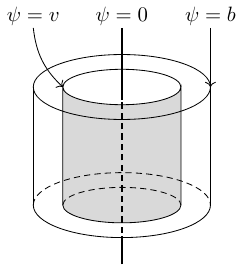}
 \caption{Cylinder}
\end{figure}

\begin{Example}[Sphere]
  Let $\psi(x) = |x|$ and $D = \{x \in \R^d;  0 < |x| \}$. Then $\bm{n}(x) = x/|x|$ ($x \in D$), $[\psi \geq 0] = \R^d$
  and, for $F \in C_c^1(\R^d,\R^d)$, 
  \[
    \int_{[|x| < r]} \text{div}\, F(x)\, dx = \int_{[|x|=r]} \frac{F(x)\cdot x}{r}\,|dx|_{[|x|=r]}
    = r^{d-1} \int_{S^{d-1}} F(r\omega)\cdot \omega\, d\omega. 
  \]
\end{Example}

\begin{Exercise}[hydrostatic balance equation]\index{hydrostatic balance}
  Let $D$ be a bounded open subset of $\R^d$ with a smooth boundary $\partial D$. Then
  \[
    \int_{\partial D} dx_{\partial D} = 0
  \]
  as a vector in $\R^d$. 
\end{Exercise}

\begin{Exercise}
  Show that, if $[0 \leq \psi < b]$ is open for some $0 < b < c$, then $[0 \leq \psi < v]$ is open for any $v \geq b$. 
\end{Exercise}

\section{Complex Functions}
So far, we have mainly dealt with real functions.
For further applications to subjects such as Fourier analysis or quantum analysis, 
we should not avoid complex functions in values as well as in variables.
Here the results established for real-valued functions are naturally extended to complex-valued ones.

Given a linear lattice $L$, we shall work with the complexified function space $L_\C = L + iL$,
which is referred to as a \textbf{complex lattice}\index{complex lattice} if the lattice condition is strengthened to
$|f| \in L_\C$ ($f \in L_\C$), i.e. $\sqrt{f^2+g^2} \in L$ ($f,g \in L$).
A linear functional $I: L \to \R$ is obviously extended to a complex-linear functional on $L_\C$, which is also denoted by $I$.

Here are some of simple facts on complex lattices.

\begin{Proposition}
  Let $L_\C$ be a complex lattice. 
  \begin{enumerate}
  \item
  Given a positive functional $I$ on $L$, the integral inequality $|I(f)| \leq I(|f|)$ ($f \in L_\C$) holds. 
  \item
    The complexification $L_\uparrow \cap L_\downarrow + iL_\uparrow \cap L_\downarrow$ of
$L_\uparrow \cap L_\downarrow$ is also a complex lattice.    
  \end{enumerate}
\end{Proposition}

\begin{proof}
  (i) Since $I$ is real-valued on $L$, 
  $\overline{I(f)} = I(\overline{f})$ for $f \in L_\C$. 
If $I(f) = 0$, the integral inequality holds trivially.
Otherwise the polar expression $I(f) = |I(f)| e^{i\theta}$ ($\theta \in \R$) is combined with
$e^{-i\theta}f + e^{i\theta} \overline{f} \leq 2|f|$ to get 
\[
  2|I(f)| = e^{-i\theta} I(f) + e^{i\theta} \overline{I(f)} = I(e^{-i\theta} f + e^{i\theta}\overline{f}) \leq 2I(|f|). 
\]

(ii)
In the expression $|f+ig| = \bigl||f| + i |g|\bigr|$ ($f, g \in L_\uparrow \cap L_\downarrow$),
we approximate $|f|, |g| \in L_\uparrow \cap L_\downarrow$ by sequences $\varphi'_n$, $\varphi''_n$, $\psi'_n$ and $\psi''_n$ in $L^+$ so that
$\varphi_n' \uparrow |f|$, $\varphi_n'' \downarrow |f|$, $\psi_n' \uparrow |g|$ and $\psi_n'' \downarrow |g|$.

Then $|\varphi_n' + i \psi_n'|, |\varphi_n'' + i \psi_n''| \in L$ satisfy
\[
  |\varphi_n' + i \psi_n'| \uparrow \bigl| |f| + i|g| \bigr|, \quad
  |\varphi_n'' + i \psi_n''| \downarrow \bigl| |f| + i|g| \bigr|, 
\]
which implies $|f+ig| \in L_\uparrow \cap L_\downarrow$. 
\end{proof}

\begin{Example}
  The complexification $S_\C(\R^d) = S(\R^d) + i S(\R^d)$ is a complex lattice and the volume integral on $S^1(\R^d)$
  is therefore complex-linearly extended to a functional $I: S_\C(\R^d) \to \C$ and then further to
  \[
    I_\updownarrow: S_\uparrow(\R^d) \cap S_\downarrow(\R^d) + i S_\uparrow(\R^d) \cap S_\downarrow(\R^d) \to \C
  \]
  so that
  $|I_\updownarrow(f)| \leq I_\updownarrow(|f|)$ for $f \in S_\uparrow(\R^d) \cap S_\downarrow(\R^d) + i S_\uparrow(\R^d) \cap S_\downarrow(\R^d)$. 
\end{Example}

When $d=1$, definite and indefinite integrals as well as improper integrals
are extended to complex-valued functions by replacing $S_\uparrow(\R) \cap S_\downarrow(\R)$
with its complexifications in such a way that the fundamental theorem of calculus and the continuity of Laplace transform remain valid.

As to the Daniell extension of a preintegral $I$ on a linear lattice $L$,
the dominated convergence theorem as well as parametric differentiation holds for complex-valued functions.
When $L_\C$ is a complex lattice, so is $L_\C^1 = L^1 + i L^1$ (Proposition~\ref{LC1})
but this is not obvious at all (see Appendix~C for details).

For the volume integral in $\R^d$ ($d \geq 2$), however,
it is practically enough to consider an open set $U \subset \R^d$ and the associated linear lattice
$C(U)\cap L^1(U)$, for which $C(U) \cap L^1(U) + i C(U) \cap L^1(U)$ is a complex lattice. 

Various convergence theorems such as parametric continuity or differentiability, the dominated convergence theorem
and the dominated series convergence theorem are obviously extended to complex-valued functions.

\begin{Example}\label{expc}~ 
  \begin{enumerate}
    \item 
    For a complex parameter $c \not= 0$, 
  \[
    \int e^{cx}\, dx = \frac{1}{c} e^{cx}. 
  \]
  Indeed, for $c = -a-ib$ with $a>0$ and $b \in \R$, $e^{-(a+ib)t} = e^{-at}(\cos bt - i\sin bt)$ is integrable on $(0,\infty)$ and
  \[
    \int_0^\infty e^{-(a+ib)t}\, dt = \frac{1}{a+ib} = \frac{a}{a^2+b^2} - i \frac{b}{a^2+b^2}.
  \]
  \item
    The identity in (i) 
    is differentiated repeatedly with respect to a complex parameter $c \in \C^\times$ to get 
 \[
   \int x^ne^{cx}\, dx = \frac{\partial^n}{(\partial c)^n} \left(\frac{1}{c} e^{cx}\right). 
  \]   
\end{enumerate}
\end{Example}

\medskip
For $f \in S_\C^1(\R^d) = S^1(\R^d) + i S^1(\R^d)$ and $\xi \in \R^d$, $e^{-ix\xi} f(x)$ ($x\xi = x_1\xi_1+ \dots + x_d\xi_d$) is integrable
and the \textbf{Fourier transform}\index{Fourier transform} $\widehat f$ of $f$ is defined by
\[
  {\widehat f}(\xi) = \int_{\R^d} f(x) e^{-ix\xi}\, dx, 
\]
which is a continuous function of $\xi \in \R^d$ by Proposition~\ref{PC}.

The Fourier transform is known to be isometric (called the Plancherel formula) in the sense that
\[
  \int_{\R^d} |\widehat f(\xi)|^2\, d\xi = (2\pi)^d \int_{\R^d} |f(x)|^2\, dx
\]
for any integrable $f$ satisfying $\int |f(x)|^2\, dx < \infty$. 

\begin{Example}
  The Fourier transform of an interval $f = [-1,1]$ is
  \[
    \widehat f(\xi) = 2 \frac{\sin\xi}{\xi}
  \]
  with the Plancherel formula of the form 
  $\int_{-\infty}^\infty (\sin\xi)^2/\xi^2\, d\xi = \pi$, 
  which turns out to be the Dirichlet integral by Example~\ref{sinc}. 
\end{Example}

\begin{Example}
  The Fourier transform of $(0,\infty) e^{-rx}$ ($r>0$) is
  \[
    \int_0^\infty e^{-rx} e^{-ix\xi}\, dx = \frac{1}{r+i\xi}. 
  \]
\end{Example}

\begin{Example}
  For $0 \not= a \in \R$, $0 \not= b \in \R$ and $r>0$, 
  \[
    \int_0^\infty e^{-rt} \frac{e^{iat} - e^{ibt}}{t}\, dt = \log \frac{r-ib}{r-ia}. 
  \]
  
  Let $s \in \R$ and consider 
  \[
    f(s) = \int_0^\infty e^{-rt} \frac{e^{iast} - e^{ibst}}{t}\, dt 
  \]
  Then
  \[
    f'(s) = \int_0^\infty e^{-rt} (iae^{iast} - ibe^{ibst})\, dt = \frac{ia}{r-ias} - \frac{ib}{r-ibs}
  \]
  is integrated to get
  \[
    f(s) = \log(s -r/(ib)) - \log(s - r/(ia)) = \log \frac{r-ibs}{r-ias}. 
\]
Here the constant of integration is specified by the condition $f(0) = 0$. 
\end{Example}

\begin{Example}
 For $a>0$, $b>0$ and $r \geq 0$, the following holds. 
\begin{align*}
    \int_0^\infty e^{-rt} \frac{\cos(at) - \cos(bt)}{t}\, dt &= \frac{1}{2} \log \frac{r^2+b^2}{r^2+a^2},\\
    \int_0^\infty e^{-rt} \frac{b\sin(at) - a\sin(bt)}{t^2}\, dt &= \frac{ab}{2} \int_0^1\log\frac{r^2+b^2u^2}{r^2+a^2u^2}\, du
\end{align*}
with
\[
  \lim_{r \to +0} \int_0^1\log\frac{r^2+b^2u^2}{r^2+a^2u^2}\, du = 2\log\frac{b}{a}.
\]

  The first equality for $r>0$ is just the real part of the formula in the previous example and the limit case $r=0$ is a consequence of
  Theorem~\ref{laplace} because $(\cos(at) - \cos(bt))/t$ is improperly integrable on $(0,\infty)$ in view of 
  \[
    \left( \frac{\sin(at)}{at} - \frac{\sin(bt)}{bt} \right)' = \frac{\cos(at) - \cos(bt)}{t} - \frac{b\sin(at) - a\sin(bt)}{abt^2}.
  \]

  To see the second equality, the function 
  \[
    g(s) = \int_0^\infty e^{-rt} \frac{b\sin(ast) - a\sin(bst)}{t^2}\, dt
  \]
 of $s \in \R$ is differentiated to have 
  \[
    g'(s) = ab \int_0^\infty e^{-rt} \frac{\cos(ast) - \cos(bst)}{t}\, dt
  = \frac{ab}{2} \log \frac{r^2+b^2s^2}{r^2+a^2s^2},   
  \]
  whence
  \[
    g(s) = \frac{ab}{2} \int_0^s \log \frac{r^2+b^2u^2}{r^2+a^2u^2}\, du.
  \]

  Finally, in view of $1 \leq (r^2+ b^2u^2)/(r^2+a^2u^2) \leq b^2/a^2$, the dominated convergence theorem is applied to get
  $\lim_{r \to +0} g(s) = sab\log(b/a)$. 
\end{Example}

\begin{Exercise}~
  \begin{enumerate}
    \item
  With the help of $(u\log(r^2+a^2u^2))'$, express the indefinite integral of $\log(r^2+a^2u^2)$ by arctangent.
\item
  By integrating $\displaystyle \frac{d}{dt} \log(r^2+ u^2t)$ from $t=a^2$ to $t=b^2$ and using repeated integrals, show that
  \[
    \lim_{r \to +0} \int_0^1 \log \frac{r^2+b^2u^2}{r^2+a^2u^2}\, du = 2\log\frac{b}{a}. 
  \]
\end{enumerate}
\end{Exercise}

\begin{Example}
If $f: (0,\infty) \to \C$ is a continuous function for which $\int_0^\infty f(t)\,dt$ exists as an improper integral,
the Laplace transform of $f$ is holomorphically extended to $z \in \C$ satisfying $\text{Re}\,z > 0$ so that 
\[
  \int_0^\infty e^{-zt} f(t) = z \int_0^\infty e^{-zt} F(t)\, dt.
\]
Here $F(t) = \int_0^t f(x)\,dx$ denotes an indefinite integral of $f$, which 
is a contunuous bounded function of $t \geq 0$ satisfying $F(0) = 0$.

In fact, from 
\[
  \Bigl( e^{-zt} F(t) \Bigr)' = - z e^{-zt} F(t) + e^{-zt} f(t)
  \quad (t>0), 
\]
$e^{-zt} f(t)$ ($t>0$) is improperly integrable and we have 
\[
  \int_0^\infty e^{-zt} f(t)\, dt = z \int_0^\infty e^{-zt} F(t)\, dt, 
\]
where the integral in the right hand side is absolutely convergent and continuous in the parameter $\text{Re}\,z>0$.
Then its holomorphy follows from 
\[
  \oint_C dz\, \int_0^\infty e^{-zt} F(t)\, dt
  = \int_0^\infty \oint_C e^{-zt} F(t)\, dz\, dt = 0
\]
for any closed path $C$ in the right half plane $\text{Re}\, z > 0$.

Let $F(t)$ be continuously extended to $t \in \R$ by $F(t) = 0$ ($t \leq 0$). 
Since
\[
G(z) \equiv \int_0^\infty e^{-zt}F(t) = \int_0^\infty e^{-ist} e^{-rt} F(t)\,
  = \int_{-\infty}^\infty e^{-ist} e^{-rt} F(t)\, dt
\]
is the Fourier transform of a continuous (square) integrable function $e^{-rt} F(t)$ of $t \in \R$,
the Laplace transform is injective on $f$:

If $\int_0^\infty e^{-rt} f(t)\, dt = 0$ for $r>0$,
then the Fourier transform of $e^{-rt} F(t)$ is identically zero and hence $e^{-rt}F(t) = 0$ ($t>0$).
Since $F$ is a primitive function of $f$, this implies $f(t) = 0$ ($t>0$).

More explicitly by Dirichlet's Fourier inversion formula, we have
\[
  e^{rt} F(t) = \frac{1}{2\pi} \lim_{s \to \infty}\int_{-s}^s e^{it\tau} G(r+i\tau)\, d\tau,  
\]
whence $f = 0 \iff G = 0$ implies $F = 0$, i.e., $f=0$. 
\end{Example}

\begin{Example}\label{gaussian} 
  Let $a \in \C$ have a strictly positive real part. Then
  \[
    \int_{-\infty}^\infty e^{-ax^2}\, dx = \sqrt{\frac{\pi}{a}}, 
  \]
  where the complex root $\sqrt{a}$ is chosen so that $\sqrt{a} > 0$ for $a > 0$.

  We shall three proofs of the formula. The first and the second ones assume the case $a>0$, whereas the third one gives a direct proof.
  Let $a = r + is$ with $r>0$ and $s \in \R$. 

  (i) Regard the integral as a function $G(s)$ of $s \in \R$.
  By parametrix differentiation, $G'(s) = \int_{-\infty}^\infty (-ix^2) e^{-ax^2}\, dx$, which
  is combined with
  \[
    \frac{d}{dx} \Bigl( x e^{-ax^2} \Bigr) = e^{-ax^2} - 2ax^2 e^{-ax^2}
  \]
  to get $ G'(s) =  - \frac{i}{2(r+is)} G(s)$, i.e., 
 $\frac{d}{ds} \left( \sqrt{r+is}\, G(s) \right) = 0$. 
Thus the function $\sqrt{r+is}\, G(s)$ is constant in $s$ with
\[
  \sqrt{r+is}\, G(s) = \sqrt{r}\, G(0) = \sqrt{\pi}.
\]

(ii) Once the integral in question is shown to be holomorphic in the complex parameter $a$,
the identity is immediate by an analytic continuation.
To see the holomorphy, we check that the parameter dependence allows indefinite integral with respect to $a$:
Given a closed path $a(t)$ ($\alpha \leq t \leq \beta$) in the half plane $\{a \in \C; \text{Re}\, a > 0\}$,
$\frac{da}{dt} e^{-a(t)x^2}$ is integrable as a function of $(t,x) \in [\alpha,\beta]\times\R$ and we have
\[
  \oint da \int_{-\infty}^\infty e^{-ax^2}\, dx = \int_{-\infty}^\infty \oint e^{-ax^2}\, da \, dx = 0. 
\]
Note that a primitive function of the integrand $e^{-ax^2}$ is given by
\[
  \int da\, \sum_{n=0}^\infty \frac{(-x^2)^n}{n!} a^n = \sum_{n=0}^\infty \frac{(-x^2)^n}{n!} \frac{1}{n+1} a^{n+1}
  = \frac{1- e^{-ax^2}}{x^2}, 
\]
which is biholomorphic in $(a,x)$. 

(iii) 
Consider the half-Gaussian integral 
\[
  H = \int_0^\infty e^{-(r+is)x^2}\, dx. 
\]
Then
\begin{align*}
 H^2 &= \int_{x>0,y>0} e^{-(r+is)(x^2+y^2)}\, dxdy\\
         &= \frac{\pi}{2} \int_0^\infty e^{-(r+is)\rho^2} \rho\, d\rho
           = \frac{\pi}{4} \int_0^\infty e^{-(r+is)u}\,du = \frac{\pi}{4} \frac{1}{r+is}
\end{align*}
and we obtain 
\[
  H = \frac{\sqrt{\pi}}{2\sqrt{r+is}}
\]
with the branch of square root specified by
\[
  \text{Re}\, H = \frac{1}{2} \int_0^\infty \frac{e^{-ru}}{\sqrt{u}} \cos(s u)\, du > 0. 
\]
Here the positivity follows from the fact that $e^{-tu}/\sqrt{u}$ is monotone-decreasing in $u>0$ and $\cos(su)$ oscillates periodically.
\end{Example}

\begin{Example} Example~\ref{gaussian} is extended to 
   \[
    \int_{-\infty}^\infty e^{-ax^2 + bx}\, dx = e^{b^2/4a} \sqrt{\frac{\pi}{a}} 
  \]
  for $b \in \C$. This is most simply reduced to Example~\ref{gaussian} by Cauchy's integral theorem but we here calculate as follows:
  \begin{align*}
    \sum_{n \geq 0} \frac{1}{n!} b^n \int_{-\infty}^\infty x^n e^{-ax^2}\, dx
    &= \sum_{m \geq 0} \frac{1}{(2m)!} b^{2m} \int_{-\infty}^\infty x^{2m} e^{-ax^2}\, dx\\
    &= \sum_{m \geq 0} \frac{1}{(2m)!} b^{2m} \left( -\frac{\partial}{\partial a}\right)^m \int_{-\infty}^\infty e^{-ax^2}\, dx\\
    &= \sum_{m \geq 0} \frac{1}{(2m)!} b^{2m} \frac{(2m)!}{4^m m!} a^{-m -1/2}\pi^{1/2}\\
    &= e^{b^2/4a} \sqrt{\frac{\pi}{a}}. 
  \end{align*}
\end{Example}

\begin{Example}
  We now evaluate the Fresnel integral\index{Fresnel integral} 
  \[
    \int_0^\infty e^{ix^2}\, dx = \frac{\sqrt{\pi}}{2} e^{i\pi/4}.
  \]
  i.e.,
  \[
    \int_0^\infty \cos(x^2)\, dx = \int_0^\infty \sin(x^2)\, dx = \sqrt{\frac{\pi}{8}}.
  \]
 
First recall that 
\[
  \int_0^\infty e^{ix^2}\, dx
\]
has a measning as a complex-valued improper integral (Exercise~\ref{fi}).

To improve the convergence, we insert $e^{-rx^2}$ with $r>0$ a positive parameter and
consider the continuous function 
\[
  H(r) = \int_0^\infty e^{-rx^2} e^{ix^2}\, dx = \frac{\sqrt{\pi}}{2\sqrt{r-i}}
\]
of $r>0$ with the branch of square root specified by $\text{Re}\,H(r) > 0$. 

Consequently,
\[
  H(0) \equiv \lim_{r \to +0} H(r) = \lim_{r \to +0}  \frac{\sqrt{\pi}}{2\sqrt{r-i}}
  = \frac{\sqrt{\pi}}{2} e^{i\pi/4}, 
\]
which is combined with Theorem~\ref{laplace} to obtain
\[
  \int_0^\infty e^{ix^2}\, dx = \frac{\sqrt{\pi}}{2} e^{i\pi/4}.
\]
\end{Example}



\medskip
Finally, as an interlude to the next section on Coulomb potentials,
we comment on $\R^n$-valued functions.

By the obvious identification $(\R^n)^X = (\R^X)^n$,
an $\R^n$-valued function $f$ is an $n$-tuple $(f_j)_{1 \leq j \leq n}$ of real functions $f_j$. 
When $X$ is furnished with an integral system $(L,I)$, we say that $f$ is $I$-integrable if each $f_j$ is $I$-integrable.
The set of $\R^n$-valued integrable functions is then a real vector space by pointwise operations.

From the complex lattice condition on $L$, $|f| = \sqrt{\sum_{j=1}^n f_j^2}$ is integrable and satisfies 
\[
  \sqrt{I(f_1)^2+ \dots + I(f_n)^2} \leq I(|f|). 
\]

In fact, the integrability is an easy induction on $n$ starting with $n=2$.
For the inequality part, $|\sum_j t_j f_j| \leq |t|\,|f|$ ($t \in \R^n$) is integrated to 
\[
  \Bigl|\sum_j t_j I(f_j)\Bigr| \leq |t| I(|f|), 
\]
which gives the assertion by choosing $t_j = I(f_j)$.

\section{Regularity on Coulomb Potentials}
Let $\gamma >0$ and $\rho$ be a locally integrable function on $\R^d$. Consider a function of $x \in \R^d$ described by 
\[
  \phi_\gamma(x) = \int_{\R^d} \frac{\rho(y)}{|x-y|^\gamma}\, dy
  = \int_0^\infty dr\, r^{d-\gamma -1} \int_{S^{d-1}} \rho(x+r\omega)\, d\omega, 
\]
which is well-defined if 
\[
  \int_0^\infty dr\, r^{d-\gamma -1} \int_{S^{d-1}} |\rho(x+r\omega)|\, d\omega < \infty. 
\]

Note that local integrability of $\rho$ is equivalent to the integrability of
$(0,R)\times S^{d-1} \ni (r,\omega) \mapsto r^{d-1} \rho(r\omega)$ for every $R>0$ and in that case
the above absolute-value integral has a meaning.

In what follows, 
we assume $\gamma < d$ to get the integrability near $r=0$ and $|\rho(y)| \leq M(1+|y|)^{-\beta}$
($y \in \R^d$) for some $\beta+\gamma > d$ and $M>0$ to get the integrability for a large $|y|$ and local boundedness of $\rho$. 


With this assumption, we can even show the continuity of $\phi_\gamma(x)$ at $x = a \in \R^d$.
To see this, use a singularity cut of the integral region near $a$ to control moving singularities: 
Set $\| \rho\|_{a,\delta} = \sup\{ |\rho(y)|; |y-a| \leq \delta\}$ for $\delta>0$. Then we have 
\begin{align*}
  \int_{|y-a| \leq \delta} \frac{|\rho(y)|}{|x-y|^\gamma}\, dy
  &\leq \| \rho\|_{a,\delta} \int_{|y-a| \leq \delta} \frac{1}{|x-y|^\gamma}\, dy\\
  &\leq \| \rho\|_{a,\delta} \int_{|y-x| \leq \delta + |x-a|} \frac{1}{|x-y|^\gamma}\, dy\\
  &= \| \rho\|_{a,\delta} |S^{d-1}| \frac{(\delta + |x-a|)^{d-\gamma}}{d-\gamma}, 
\end{align*}
which can be arbitrarily small if $|x-a| \leq \delta$ with $\delta$ sufficiently small.

The continuity is therefore reduced to that of
\[
  \int_{|y-a| > \delta} \frac{\rho(y)}{|x-y|^\gamma}\, dy
\]
at $x = a$. 
To see this, let $|x-a| \leq \delta/2$. Then $|\rho(y)|/|x-y|^\gamma \leq M/|y-a|^{\beta+\gamma}$ ($|y-a| \geq \delta$) for some $M > 0$, whence
the integrand is dominated by $M/|y-a|^{\beta+\gamma}$ with
\[
  \int_{|y-a| \geq \delta} \frac{1}{|y-a|^{\beta+\gamma}}\, dy
    = |S^{d-1}| \int_\delta^\infty r^{d-\beta-\gamma-1}\, dr = |S^{d-1}| \frac{\delta^{d-\beta-\gamma}}{\beta+\gamma-d} < \infty
  \]
in view of $\beta + \gamma > d$. The dominated convergence theorem is then applied to have
\[
  \lim_{x \to a} \int_{|y-a| > \delta} \frac{\rho(y)}{|x-y|^\gamma}\, dy
  = \int_{|y-a| > \delta} \frac{\rho(y)}{|a-y|^\gamma}\, dy. 
\]

\begin{Exercise}
Write down the above proof of continuity in an $\epsilon$-$\delta$ form. 
\end{Exercise}


\begin{Example}
  Let $\rho(y) = 1/|y|^\beta$ ($|y| > 1$) and $\rho(y) = 0$ ($|y| \leq 1$). If the assumption on $\gamma$ is strengthened to
  $\gamma < d-1$, then the continuous function $\phi_\gamma$ vanishes at $\infty$: 
  \begin{align*}
    \phi_\gamma(x) &= \int_1^\infty dr\, r^{d - \beta - 1} \int_{S^{d-1}} \frac{1}{|x-r\omega|^\gamma}\, d\omega\\
         &= |S^{d-2}| \int_1^\infty dr\, r^{d - \beta - 1} \int_0^\pi \frac{\sin^{d-2}\theta}{(|x|^2 + r^2 -2|x|r\cos\theta)^{\gamma/2}} \, d\theta\\
                   &= |S^{d-2}| \int_1^\infty dr\, \frac{r^{d - \beta - 1}}{(|x|^2+r^2)^{\gamma/2}}
                     \int_0^\pi \frac{\sin^{d-2}\theta}{(1 - s\cos\theta)^{\gamma/2}} \, d\theta,  
  \end{align*}
  where $s = 2|x|r/(|x|^2+ r^2) \in [0,1]$.
  
 Since the inner latitude integral is estimated by 
  \begin{align*}
    &\int_0^{\pi/2} \frac{\sin^{d-2}\theta}{(1-s\cos\theta)^{\gamma/2}} \, d\theta
    +   \int_0^{\pi/2} \frac{\sin^{d-2}\theta}{(1+s\cos\theta)^{\gamma/2}} \, d\theta\\
    \leq &\int_0^{\pi/2} \frac{\sin^{d-2}\theta}{(1-s\cos\theta)^{\gamma/2}} \, d\theta
      +   \int_0^{\pi/2} \sin^{d-2}\theta \, d\theta\\
= &\int_0^{\pi/2} \frac{\sin^{d-2}\theta}{(1-s\cos\theta)^{\gamma/2}} \, d\theta
    + \frac{\Gamma((d-1)/2) \Gamma(1/2)}{2\Gamma(d/2)}\\ 
    \leq &\int_0^{\pi/2} \frac{\sin^{d-2}\theta}{(1-\cos\theta)^{\gamma/2}} \, d\theta
    + \frac{\Gamma((d-1)/2) \Gamma(1/2)}{2\Gamma(d/2)},  
  \end{align*}
 with 
  \[
    \int_0^{\pi/2} \frac{\sin^{d-2}\theta}{(1-\cos\theta)^{\gamma/2}} \, d\theta < \infty \iff d-\gamma-1 > 0,  
  \]
 one sees that
  \[
    0 \leq \phi_\gamma(x) \leq |S^{d-2}| C_{d,\gamma} \int_1^\infty \frac{r^{d-\beta-1}}{(|x|^2+ r^2)^{\gamma/2}}\, dr, 
  \]
 where 
  \[
      C_{d,\gamma} = \int_0^{\pi/2} \frac{\sin^{d-2}\theta}{(1-\cos\theta)^{\gamma/2}} \, d\theta
    + \frac{\Gamma((d-1)/2) \Gamma(1/2)}{2\Gamma(d/2)}
  \]
  and
  \[
    \int_1^\infty \frac{r^{d-\beta-1}}{(|x|^2+ r^2)^{\gamma/2}}\, dr
    \leq \int_1^\infty \frac{r^{d-\beta-1}}{r^\gamma}\, dr = \frac{1}{\beta+\gamma-d} < \infty. 
  \]
 Now the dominated convergence theorem shows $\displaystyle \lim_{|x| \to \infty} \phi_\gamma(x) = 0$. 
\end{Example}


Next we move on to the differentiability of $\phi_\gamma$.
Consider the formal derivative (an $\R^d$-valued function)
\[
  \int_{\R^d} \left( \frac{\partial}{\partial x_i} \frac{1}{|x-y|^\gamma} \right)_{1 \leq i \leq d} \rho(y)\, dy
  = \gamma \int_{\R^d} \frac{y-x}{|x-y|^{\gamma+2}} \rho(y)\, dy,  
\]
whose integrability 
\[
  \int_{\R^d} \frac{|y-x|}{|x-y|^{\gamma+2}} |\rho(y)|\, dy
  = \int_{\R^d} \frac{1}{|x-y|^{\gamma+1}} |\rho(y)|\, dy < \infty
\]
is exactly the well-definedness of $\phi_{\gamma+1}$ and satisfied if $d - \gamma - 1 > 0$
(note that $\beta > d-\gamma -1$ follows from $\beta > d-\gamma$ then).

We shall show that this condition in turn implies
\[
  \phi_\gamma'(x) = \gamma \int_{\R^d} \frac{y-x}{|x-y|^{\gamma+2}} \rho(y)\, dy 
\]
and $\phi_\gamma'(x)$ is continuous in $x \in \R^d$.

Recall that the above equality at $x=a \in \R^d$ means that, given $\epsilon>0$, there exists $\delta>0$ such that $|x-a| \leq \delta$
implies
\[
  \left| \phi_\gamma(x) - \phi_\gamma(a) - \gamma \int_{\R^d} \frac{(x-a)\cdot(y-a)}{|a-y|^{\gamma+2}} \rho(y)\, dy \right|
  \leq \epsilon |x-a|. 
\]

To see this, we argue as in the proof of continuity of $\phi_\gamma$: 
For the differential term,
\begin{align*}
  \int_{|y-a| \leq \delta} \frac{|a-y|}{|a-y|^{\gamma+2}} |\rho(y)|\, dy
  &\leq \| \rho\|_{a,\delta} \int_{|y-a| \leq \delta} \frac{1}{|a-y|^{\gamma+1}}\, dy\\
  &= \| \rho\|_{a,\delta} |S^{d-1}| \frac{\delta^{d-\gamma - 1}}{d-\gamma - 1},
\end{align*}
which approaches $0$ as $\delta \to 0$.

To control the difference term, letting $x(t) = tx + (1-t)a$, we introduce
\[
  D(t,y) = \gamma \frac{y-x(t)}{|x(t) - y|^{\gamma+2}}
\]
so that 
\[
  \frac{1}{|x-y|^\gamma} - \frac{1}{|a-y|^\gamma}
  = \int_0^1\frac{\partial}{\partial t} |x(t) - y|^{-\gamma}\, dt = (x-a)\cdot \int_0^1 D(t,y)\, dt. 
\]
In view of $|y-a| \leq \delta \Longrightarrow |y-x(t)| \leq \delta + t|x-a|$, the singularity cut of the difference term is estimated by
\begin{align*}
  |x-a| &\int_0^1dt\, \int_{|y-a| \leq \delta} |D(t,y)| |\rho(y)|\, dy\\
                     &\leq |x-a| \| \rho\|_{a,\delta} \int_0^1dt\, \int_{|y-a| \leq \delta} |D(t,y)|\, dy\\
        &\leq \gamma |x-a| \| \rho\|_{a,\delta} \int_0^1dt\, \int_{|y-x(t)| \leq \delta + t|x-a|} \frac{1}{|x(t) - y|^{\gamma+1}}\, dy\\
        &= \gamma |x-a| \| \rho\|_{a,\delta} \frac{|S^{d-1}|}{d-\gamma-1} \int_0^1 (\delta + t|x-a|)^{d-\gamma-1}\, dt\\ 
  &\leq \gamma |x-a| \| \rho\|_{a,\delta} \frac{|S^{d-1}|}{d-\gamma-1} (\delta + |x-a|)^{d-\gamma-1}, 
\end{align*}
which approaches $0$ uniformly in $|x-a| \leq \delta$ as $\delta \to 0$.

The remaining part is given by
\[
  \int_0^1dt\, \int_{|y-a|>\delta} (x-a)\cdot (D(t,y) - D(0,y)) \rho(y)\, dy 
\]
and the validity of the differential formula is reduced to 
\[
  \lim_{x \to a} \int_0^1dt\, \int_{|y-a| > \delta} |D(t,y) - D(0,y)| |\rho(y)|\, dy = 0. 
\]

Note here that $D(t,y)$ depends on $x$ and $\lim_{x \to a} D(t,y) = D(0,y)$ for $0 \leq t \leq 1$ and $|y-a| \geq \delta$.
The above convergence is then a conseuqence of the dominated convergence theorem
once $|D(t,y)| |\rho(y)|$ is majorized uniformly in $|x-a| \leq \delta/2$
by an integrable function on $|y-a| \geq \delta$.

In fact, in view of $\rho(y) = O(|y|^{-\beta})$ and
\[
  |a-y| + |x-a| \geq |x(t) - y| \geq |a-y| -|x-a| \geq |a-y| - \frac{\delta}{2} \geq \frac{\delta}{2},  
\]
we can find $M>0$ satisfying 
\[
  |D(t,y)| |\rho(y)| = \frac{|\rho(y)|}{|x(t) - y|^{\gamma+1}}
  \leq M \frac{|y-a|^{-\beta}}{|y-a|^{\gamma+1}}
  = M \frac{1}{|y-a|^{\beta+\gamma+1}} 
  \]
  for $|y-a| \geq \delta$ in such a way that 
  \[
    \int_{|y-a| > \delta} \frac{1}{|y-a|^{\beta+\gamma+1}}\, dy = |S^{d-1}| \frac{\delta^{d-\beta-\gamma-1}}{\beta+\gamma + 1 - d} < \infty. 
  \]

  We shall now check the continuity of $\phi_\gamma'$, which can be done analogously with that of $\phi_\gamma$:
  The singularity part is estimated by 
  \begin{align*}
    \int_{|y-a| \leq \delta} \frac{|y-x|}{|x-y|^{\gamma+2}} |\rho(y)|\, dy
    &\leq \| \rho\|_{a,\delta} \int_{|y-a| \leq \delta} \frac{1}{|x-y|^{\gamma+1}}\, dy\\
    &\leq \| \rho\|_{a,\delta} \int_{|y-x| \leq \delta + |x-a|} \frac{1}{|x-y|^{\gamma+1}}\, dy\\
    &= \| \rho\|_{a,\delta} |S^{d-1}| \frac{\delta^{d-\gamma-1}}{d-\gamma - 1} < \infty, 
  \end{align*}  
  which can be arbitrarily small if $|x-a| \leq \delta$ with $\delta$ sufficiently small. 

  The continuity of $\phi_\gamma'(x)$ at $x=a$ is therefore reduced to that of
  \[
    \int_{|y-a| > \delta} \frac{y-x}{|x-y|^{\gamma+2}} \rho(y)\, dy,  
  \]
  which follows from the dominated convergence theorem: We can find $M>0$ satisfying 
  $|\rho(y)|/|x-y|^{\gamma+1} \leq M/|y-a|^{\beta+\gamma+1}$ ($|x-a|\leq \delta/2$, $|y-a| \geq \delta$) 
  so that $\int_{|y-a|>\delta} |y-a|^{-\beta-\gamma-1}\, dy < \infty$ due to $d-\beta-\gamma <0$. 
  
  We can repeat this process up to the $n$-th differential $\phi_\gamma^{(n)}$ as far as $d-\gamma - n > 0$.

  \begin{Theorem}
    Assume that $\rho$ is a locally bounded and locally integrable function satisfying $\rho(y) = O(1/|y|^\beta)$ with $\beta> d-\gamma$.
    Let $n \geq 0$ be the maximal integer satisfying $d-\gamma-n > 0$. 
    Then $\phi_\gamma$ is a $C^n$ function. 
  \end{Theorem}

  {\small
  \begin{Remark}
    The condition $\beta>d-\gamma$ is weaker than the integrability of $\rho$, i.e., $\beta>d$.
    Thus the source function $\rho$ of an infinite total charge may produce a finite and continuously differentiable potential. 
  \end{Remark}}

  The overall differentiability discussed so far is connected with the degree $\gamma$ of singularity of $1/|x-y|^\gamma$ and
  $\phi_\gamma'$ is well-defined if $d-\gamma-1>0$ but not for $\phi_\gamma''$ if $d-\gamma -2 \leq 0$.
  Even in that case, we can convert local differentiability of $\rho$ into the local differentiability of $\phi_\gamma'$. 
  
  Keep the condition including $d-\gamma-1>0$ which enables us to have an overall integral expression of $\phi_\gamma'$
  and assume that $\rho$ is continuously differentiable on a bounded open set $V$ in such a way that the boundary $\partial V$
  is piece-wise smooth and $\rho'$ on $V$ is continuously extended to the closure $\overline{V}$.
  


  \begin{Theorem}\label{surface} For $x \in V$, we have
  \begin{multline*}
      \frac{\partial \phi_\gamma}{\partial x_j}(x) = \gamma \int_{\R^d \setminus V} \frac{y_j-x_j}{|x-y|^{\gamma+2}} \rho(y)\, dy\\
      + \int_{V} \frac{1}{|x-y|^\gamma} \frac{\partial \rho}{\partial y_j}(y)\, dy
      - \int_{\partial V} \frac{\rho(y)}{|x-y|^\gamma} e_j\cdot (dy)_{\partial V}, 
    \end{multline*}
    which is continuously differentiable on $V$ and $\phi_\gamma''(x)$ is given by 
     \begin{multline*}
      \frac{\partial^2 \phi_\gamma}{\partial x_i \partial x_j}(x)
      = \int_{\R^d \setminus V} \frac{\partial^2}{\partial x_i\partial x_j} \left(\frac{1}{|x-y|^{\gamma}} \right) \rho(y)\, dy\\
      + \int_{V} \frac{\partial}{\partial x_i} \left( \frac{1}{|x-y|^\gamma}\right) \frac{\partial \rho}{\partial y_j}(y)\, dy
      - \int_{\partial V} \frac{\partial}{\partial x_i} \frac{\rho(y)}{|x-y|^\gamma} e_j\cdot (dy)_{\partial V}.  
    \end{multline*} 
  \end{Theorem}

  \begin{proof}
  In the expression
  \[
    \phi_\gamma'(x) = \gamma \int_{\R^d \setminus \overline{V}} \frac{y-x}{|x-y|^{\gamma+2}}\rho(y)\, dy
                     + \gamma \int_{\overline{V}} \frac{y-x}{|x-y|^{\gamma+2}}\rho(y)\, dy, 
  \]
  the first integral is infinitely differentiable in $x \in V$ with its differentials given by differentiating
  the integrand. To make the singularity mild in the second integral, we rewrite
  \begin{align*}
   \gamma \frac{y_j-x_j}{|x-y|^{\gamma+2}}\rho(y)
    &= - \rho(y) \frac{\partial}{\partial y_j}\frac{1}{|x-y|^\gamma}\\
    &= - \frac{\partial}{\partial y_j}\frac{\rho(y)}{|x-y|^\gamma} + \frac{1}{|x-y|^\gamma} \frac{\partial \rho}{\partial y_j}
  \end{align*}
  and apply the divergence theorem (Corollary~\ref{div}) to have an expression 
   \begin{multline*}
      \phi_\gamma'(x) = \gamma \int_{\R^d \setminus V} \frac{y-x}{|x-y|^{\gamma+2}} \rho(y)\, dy\\
      + \int_{V} \frac{\rho'(y)}{|x-y|^\gamma}\, dy - \int_{\partial V} \frac{\rho(y)}{|x-y|^\gamma} e_j\cdot (dy)_{\partial V}. 
    \end{multline*}
    which is valid for $x \in V$ and
    continuously differentiable in $x \in V$ with the second derivative $\phi_\gamma''(x)$ given by differentiating each integrand. 
  \end{proof}
  
  \begin{Corollary}
    If $\rho(y) = 0$ ($y \in V$), $\phi_\gamma(x)$ is infinitely differentiable in $x \in V$ with 
      \[
      \frac{\partial^{|\alpha|}}{\partial x^\alpha} \phi_\gamma(x)
      = \int_{\R^d \setminus V} \frac{\partial^{|\alpha|}}{\partial x^\alpha} \left( \frac{1}{|x-y|^\gamma} \right) \rho(y)\, dy. 
    \]
  \end{Corollary}

  We now restrict ourselves to the case $\gamma = d-2$ ($d \geq 3$) of \textbf{Coulomb potential}\footnote{Coulomb potential
  is also referred to as Newton potential.}\index{Coulomb potential}\index{Newton potential},
  for which $\phi = \phi_{d-2}$ is continuouly differentiable but $d-\gamma -2 = 0$.

  \begin{Theorem}
    Let $d \geq 3$ and $D$ be the domain of continuous differentiability of $\rho$,
    i.e., $a \in \R^d$ belongs to $D$ if $\rho(y)$ is continuously
    differentiable in a neighborhood of $a$.
    
    Then $\phi$ is $C^2$ on $D$ and satisifies the \textbf{Poisson equation}\index{Poisson equation} 
    \[
      -\Delta \phi(x) = |S^{d-1}| \rho(x) \quad(x \in D),  
    \]
    where $\displaystyle \Delta = \sum_{j=1}^d \frac{\partial^2}{\partial x_j^2}$ is the \textbf{Laplacian}\index{Laplacian} in $\R^d$.
  \end{Theorem}
  
  \begin{proof}
    Let $a \in D$ and $V$ be a small ball $|y-a| < \delta$ in Theorem~\ref{surface}. Then, for $|x-a| < \delta$, 
    \begin{multline*}
      \Delta \phi(x) = \int_{|y-a| > \delta} \Delta \frac{1}{|x-y|^{d-2}} \rho(y)\, dy\\
      + \int_{|y-a| < \delta} \sum_{i=1}^d \frac{y_i-x_i}{|x-y|^d} \frac{\partial \rho}{\partial y_i}(y)\, dy\\
      - \int_{|y-a|=\delta} \sum_{i=1}^d \frac{y_i-x_i}{|x-y|^d} \rho(y) e_j\cdot (dy)_{|y-a|=\delta}. 
    \end{multline*}
    Due to the identity $\Delta (1/|x-y|^{d-2}) = 0$, the first term vanishes.
    In the second integrand, an inner product inequality is used to have 
    \[
      \sum_{i=1}^d \frac{|y_i-x_i|}{|x-y|^d} \left|\frac{\partial \rho}{\partial y_i}(y)\right|
      \leq \frac{\|\rho'\|_{a,\delta}}{|x-y|^{d-1}} 
  \]
  and the second integral is estimated by
  \[
    \| \rho'\|_{a,\delta} \int_0^{\delta + |x-a|} dr\, r^{d-1} \frac{1}{r^{d-1}}\, |S^{d-1}|
    = \| \rho'\|_{a,\delta} |S^{d-1}| (\delta + |x-a|), 
  \]
  which vanishes as $\delta \to 0$ for the choice $x=a$.

  We evaluate the surface integral in the third term by putting $x = a$: 
  \begin{align*}
  \int_{|y-a|=\delta} &\sum_{j=1}^d \frac{y_j-a_j}{|a-y|^d} \rho(y) \frac{e_j\cdot (y-a)}{|y-a|}\, |dy|_{|y-a|=\delta}\\
    &=  \int_{|y-a|=\delta} \frac{\rho(y)}{|y-a|^{d-1}}\, |dy|_{|y-a|=\delta}\\
    &= \int_{S^{d-1}} \rho(a+\delta\omega)\, d\omega
    \xrightarrow{\delta \to 0}  
    |S^{d-1}| \rho(a).
  \end{align*}
\end{proof}
  


A formal expression for $d=2$ loses its meaning. Even in that case, we can work with the differential of $1/|x-y|^\gamma$ at 
$\gamma = 0$: Let $\rho(y)$ be a locally integrable function of $y \in \R^2$ and assume that 
$\rho(y) = O(1/(1+|y|)^\beta)$ for $\beta> 2$. Then 
\[
  \phi(x) = -\int_{\R^2} \rho(y) \log|x-y|\, dy.
\]
is continuously differentiable with
\[
  \phi'(x) = \int_{\R^2} \frac{y-x}{|x-y|^2} \rho(y)\, dy. 
\]
Moreover in the situation of the previous theorem, we can show that 
$\phi$ is $C^2$ on $D$ and satisfies
\[
  - \Delta \phi(x) = 2\pi \rho(x)
  \quad
  (x \in D).
\]

\begin{Exercise}
Check these assertions. 
\end{Exercise}

{\small
\begin{Remark}
  The substantial part $1/|x|^{d-2}$ (or $\log|x|$), which is the Coulomb potential produced by a point charge at $x=0$, 
  is known (up to a multiplicative constant) as a fundamental solution or a Green's function of the Laplacian in $\R^d$.
  \index{Green's function}\index{fundamental solution}
\end{Remark}}

\vfill 
\pagebreak
\appendix

\section{Compact Sets and Continuous Functions}

Recall that we use the notation $|x| = \sqrt{(x_1)^2 + \dots + (x_d)^2}$ for $x = (x_1,\dots,x_d) \in \R^d$.

The following is a sophisticated paraphrase of continuity (completeness) of real numbers.

\begin{Theorem}[Heine-Borel]\label{HB}\index{Heine-Borel}
  For a subset of $\R^d$, the following conditions are equivalent.
  \begin{enumerate}
  \item
    $K$ is bounded and closed.
  \item(finite covering property\index{finite covering property})
    Given a family $(U_i)$ of open subsets of $\R^d$ satisfying $K \subset \bigcup_i U_i$, we can find a finite set $J$ of indices
    so that $K \subset \bigcup_{j \in J} U_j$.
  \item(finite intersection property\index{finite intersection property}) Given a family $(F_i)$ of closed subsets of $\R^d$,
    if $\bigcap_{j \in J} (K \cap F_j) \not= \emptyset$ for any finite set $J$ of indices,
    then $\bigcap_i (K \cap F_i) \not= \emptyset$. 
  \end{enumerate}
\end{Theorem}

\begin{proof}
  (ii) and (iii) are equivalent because they are just in the relation of complements on open sets and closed sets.

  (i) $\Longrightarrow$ (ii): Assume that $K \not\subset \bigcup_{j \in J} U_j$ for any finite set $J$ of indices and we shall show that
  $K \not\subset \bigcup_i U_i$. 
  Choose a closed (and bounded) rectangle $R$ including $K$ and divide $R$ at middle coordinates into $2^d$ pieces of closed subrectangles.

  By the assumption, there exists at least one piece $R'$ for which $K \cap R'$ does not fulfill the finite-covering property.
  Next, dividing $R'$ likewise, we can find a subpiece $R''$ of $R'$ so that $K \cap R''$ does not fulfill the finite-covering property.

  The process is then repeated to get a decreasing sequence $R^{(n)}$ of closed renctangles so that each $R^{(n)}$ has the half-size width of
  $R^{(n-1)}$ and $K \cap R^{(n)}$ does not satisfy the finite covering property.

  By the nested interval property of real numbers, $\bigcap R^{(n)}$ is a one-point set $\{x\}$.
  Since $K \cap R^{(n)} \not= \emptyset$, $x$ belongs to $\overline{K} = K$.

  If there is any index $i$ satisfying $x \in U_i$, then $R^{(n)} \subset U_i$ for a sufficiently large $n$ and $K \cap R^{(n)}$ is covered by
  a single  open set $U_i$, which contradicts with our choice of $R^{(n)}$.
  Thus $x \not\in \bigcup_i U_i$, proving $K \not\subset \bigcup_i U_i$.

  (ii) $\Longrightarrow$ (i): If $K$ is not bounded, we can find a sequence $(x_n)$ in $K$ so that $|x_n| \uparrow \infty$.
  Clearly open balls $B_{|x_n|}(0)$ covers $K$ but not for any finitely many balls. If $K$ is not closed, there is $a \not\in K$ satisfying
  $B_r(a) \cap K \not= \emptyset$ ($r>0$) and an increasing sequence $\R^d \setminus \overline{B}_{1/n}(a)$ of open sets covers $K$ but
  not for any finite subfamily.   
\end{proof}

\begin{figure}[h]
  \centering
 \includegraphics[width=0.4\textwidth]{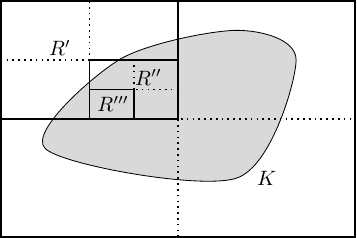}
 \caption{Nested Rectangles}
\end{figure}

Equivalent topological properties (ii) and (iii) are also referred to as being \textbf{compact}\index{compact} in other topological situations.

\begin{Exercise}[Bolzano-Weierstrass] \index{Bolzano-Weierstrass}
  Show that any bounded sequence in $\R^d$ has a convergent subsequence. 
\end{Exercise}

The following is immediate from finite covering property.

\begin{Proposition}~
  \begin{enumerate}
    \item
  Continuous images of a compact set are compact.
\item
  A continuous real function on a compact set $K$ attains maximum and minimum.
\end{enumerate}
\end{Proposition}

For a function $f$ defined on a subset $X$ of $\R^d$ and $\delta > 0$, set
\[
  C_f(\delta) = \sup \{ |f(x) - f(y)|; x, y \in X, |x-y| \leq \delta\} \in [0,\infty].
\]
Clearly $C_f(\delta)$ is an increasing function of $\delta$ and $f$ is said to be \textbf{uniformly continuous}\index{uniformly continuous}
if $\lim_{\delta \to 0} C_f(\delta) = 0$.
Notice that uniformly continuous functions are continuous.

\begin{Theorem}[Heine]\label{UC}\index{Heine}
A continuous function $f$ defined on a compact subset $K$ of $\R^d$ is uniformly continuous.   
\end{Theorem}

\begin{proof}
  Given $\epsilon>0$, we show that $C_f(\delta) \leq 2\epsilon$ for some $\delta>0$.
  Since $f$ is continuous, for any $a \in K$, we can choose $\delta(a)>0$ so that $|x-a| \leq 2\delta(a)$ implies
  $|f(x) - f(a)| \leq \epsilon$. Then $K \subset \bigcup_{a \in K} B_{\delta(a)}(a)$ and by the finite covering property
  we can find a finitely many points $a_1,\dots,a_n$ of $K$ so that $K \subset \bigcup_{1 \leq j \leq n} B_{\delta(a_j)}(a_j)$.

  Let $\delta = \delta(a_1) \wedge \dots \wedge \delta(a_n)$ and  $x,y \in K$ satisfy $|x-y| \leq \delta$.
  Since $x \in B_{\delta(a_j)}(a_j)$ for some $j$, 
  $|y-a_j| \leq |x-y| + |x-a_j| \leq \delta + \delta(a_j) \leq 2\delta(a_j)$ implies 
  $|f(x) - f(y)| \leq |f(x) - f(a_j)| + |f(y) - f(a_j)| \leq 2\epsilon$. 
\end{proof}

\begin{Theorem}[Tietze extension a la Riesz\footnote{See \cite{Gu} for more information.}]
  \label{tietze}\index{Tietze extension}
 Given a continuous positive function $h$ defined on a compact subset $K$ of $\R^d$,
  \[
    g(x) = d(x,K) \max\Bigl\{ \frac{h(y)}{|x-y|}; y \in K\Bigr\}
    \quad(x \in \R^d \setminus K)
  \]
  and $g(x) = h(x)$ ($x \in K$) give a continuous function $g$ on $\R^d$. 
\end{Theorem}

\begin{proof}
  Clearly $g = h$ is continuous on an open set $K \setminus \partial K$ and 
  we show that $g$ is continous on $\R^d \setminus K$ as well. Since $d(x,K)$ is a strictly positive continuous function of $x \not\in K$,
  this is equivalent to the continuity of $g(x)/d(x,K)$ at $x = a \in \R^d \setminus K$.
  
  Let $a_n \to a$ in $\R^d \setminus K$. 
  Since $h(y)/|y-a_n|$ and $h(y)/|y-a|$ are  continuous in $y \in K$
  with $K$ compact, we can find a sequence $c_n \in K$ and $c \in K$ satisfying 
  \[
    \frac{g(a_n)}{d(a_n,K)} = \frac{h(c_n)}{|c_n-a_n|},
    \quad 
    \frac{g(a)}{d(a,K)} = \frac{h(c)}{|c-a|}. 
  \]
  In the obvious inequality 
  \[
    \frac{h(c)}{|c-a_n|} \leq \frac{h(c_n)}{|c_n-a_n|}, 
  \]
  we move over to any accumulation point $c_\infty \in K$ of $(c_n)$ to get
  \[
    \frac{h(c)}{|c-a|} \leq \frac{h(c_\infty)}{|c_\infty-a|} \leq \frac{h(c)}{|c-a|}, 
  \]
  which implies 
  \[
    \lim_{n \to \infty} \frac{g(a_n)}{d(a_n,K)} = \lim_{n \to \infty} \frac{h(c_n)}{|c_n-a_n|}
    = \frac{h(c)}{|c-a|} = \frac{g(a)}{d(a,K)},
  \]
  proving the continuity of $g$ on $\R^d \setminus K$.
  
  Thus the whole problem is reduced to the continuity of $g(x)$ at $x = a \in \partial K \subset K$.
  Choose again a sequence $a_n \to a$, this time $a_n \in K$ or not. For a subsequence $a_{n'} \in K$, the continuity of $h$
  shows $\lim_{n \to \infty} g(a_{n'}) = \lim_{n \to \infty} h(a_{n'}) = h(a) = g(a)$.
  So we focus on the case $a_n \not\in K$ and choose this time $b_n \in K$ and $c_n \in K$ so that
  \[
    |a_n-b_n| = d(a_n,K),
    \quad
    \frac{h(c_n)}{|c_n - a_n|} = \max\left\{ \frac{h(y)}{|y-a_n|}; y \in K \right\}. 
  \]
  From $|b_n-a_n| = d(a_n,K) \leq |a_n-a|$, one sees that $\lim_{n \to \infty} b_n = a$.

 The obvious inequality
  \[
    \frac{h(b_n)}{|a_n-b_n|} \leq \frac{h(c_n)}{|a_n-c_n|}
    \iff |a_n-c_n| h(b_n) \leq |a_n-b_n| h(c_n), 
  \]
  then shows that any accumulation point $c_\infty$ of $(c_n)$ satisfies $|a-c_\infty| h(a) \leq 0$. Thus, if $h(a) > 0$,
  $c_n \to a$ and inequalities
  \[
    h(b_n) \leq \frac{|a_n-b_n|}{|a_n-c_n|} h(c_n) = \frac{d(a_n,K)}{|a_n-c_n|} h(c_n)
    \leq h(c_n)
  \]
  become equalities in the limit and we have 
  \[
    \lim_{n \to \infty} g(a_n) = \lim_{n \to \infty} \frac{|a_n-b_n|}{|a_n-c_n|} h(c_n) = h(a) = g(a). 
  \]

  Finally, if $h(a) = 0$ and there is any subsequence $(c_{n'})$ of $(c_n)$ which converges to $c_\infty \not= a$, 
  \[
    \lim_{n\to \infty} g(a_{n'}) = \lim_{n \to \infty} \frac{|a_{n'}-b_{n'}|}{|a_{n'} - c_{n'}|} h(c_{n'}) = 0 = h(a) = g(a). 
   \]  
\end{proof}

\section{Abstract Darboux Integral} 
Let $L$ be a linear lattice on a set $X$ and $I:L \to \R$ be a positive functional.
For a real-valued function $f$ on $X$, define its upper and lower integrals of Darboux by
\begin{align*}
  I^\circ(f) &= \inf\{ I(\psi); \psi \in L, f \leq \psi\} \in [-\infty,\infty], \\
  I_\circ(f) &= \sup\{ I(\varphi); \varphi \in L, \varphi \leq f\} \in [-\infty,\infty]
\end{align*}
so that $I_\circ(f) \leq I^\circ(f)$ and $I_\circ(-f) = - I^\circ(f)$.
We say that $f$ is Darboux integrable if $I_\circ(f) = I^\circ(f) \in \R$, i.e.,
\[
  \inf\{ I(\psi-\varphi); \varphi \leq f \leq \psi, \varphi, \psi \in L\} = 0, 
\]
with the common value called the Darboux integral of $f$.
It is immediate to see that the set $L^\circ$ of Darboux integrable functions is a linear lattice containing $L$ and 
the restriction of $I^\circ$ to $L^\circ$ is a positive linear functional extending $I$.
We call $(L^\circ,I^\circ)$ the Darboux extension of $(L,I)$. 

\begin{Exercise}
The Darboux extension of $(L^\circ,I^\circ)$ is $(L^\circ,I^\circ)$ itself. 
\end{Exercise}

When $I$ is continuous in addition, upper and lower integrals of Daniell, $\overline{I}$ and $\underline{I}$, satisfy
\[
  I_\circ(f) \leq \underline{I}(f) \leq \overline{I}(f) \leq I^\circ(f).
\]
Consequently $L^\circ \subset L^1$ and the Daniell integral $I^1$ extends the Darboux integral.

\begin{Lemma}
  For a real-valued function $f \in L_\downarrow$, $I_\downarrow(f) = I^\circ(f)$. 
\end{Lemma}

\begin{proof}
  From an expression $\psi_n \downarrow f$ with $\psi_n \in L$, we see that
  $I^\circ(f) \leq \lim I(\psi_n) = I_\downarrow(f)$. For $\psi \in L$ majorizing $f$, we have
  $(\psi\vee\psi_n)\downarrow \psi$ and 
  \[
    I_\downarrow(f) = \lim I(\psi_n) \leq \lim I(\psi\vee\psi_n) = I(\psi),
  \]
  where the continuity of $I$ is used in the last equality.
  Taking infimum on $\psi$, we obtain the reverse inequality $I_\downarrow(f) \leq I^\circ(f)$.   
\end{proof}

\begin{Corollary}
 We have $L_\uparrow \cap L_\downarrow \subset L^\circ$ and $I_\updownarrow$ coincides with $I^\circ$ on $L_\uparrow \cap L_\downarrow$. 
\end{Corollary}

In the following, we shall identify the Darboux extension of the volume integral on $S(\R^d)$ with the so-called Riemann-Darboux integral. 
Let $\Delta = \Delta_1\times\dots\times \Delta_d$ be a multiple partition of a closed rectangle $[a,b] \subset \R^d$
and $(R_i)$ be the open parts in $\Delta$.
Given a bounded function $f:[a,b] \to \R$, introduce upper and lower sums of Darboux by
\[
  I^\Delta(f) = \sum_i |R_i|\,(\sup f(\overline{R_i})),
  \quad
  I_\Delta(f) = \sum_i |R_i|\,(\inf f(\overline{R_i})). 
\]
By comparing these with the Darboux approximation (see \S 3) 
\[
  I(f^\Delta) = \sum_i |R_i|\,(\sup f(R_i)),
  \quad
  I(f_\Delta) = \sum_i |R_i|\,(\inf f(R_i)),  
\]
we observe that $I_\Delta(f) \leq I(f_\Delta) \leq I(f^\Delta) \leq I^\Delta(f)$.

According to Darboux, $f$ is said to be \textbf{Riemann integrable}
if $\inf\{ I^\Delta(f)\} = \sup\{ I_\Delta(f)\}$, i.e., $\inf\{ I^\Delta(f) - I_\Delta(f)\} = 0$, 
with the common value referred to as the \textbf{Riemann integral} of $f$.\index{Riemann integral} 

\begin{Proposition}
  Let $L = S(\R^d)$ and $I$ be the volume integral. 
  \begin{enumerate}
  \item
    $L^\circ$ consists of doubly bounded functions.
  \item
    A bounded function $f$ on $\R^d$ supported by a closed rectangle $[a,b]$ in $\R^d$
    belongs to $L^\circ$ if and only if $[a,b]f$ is Riemann integrable.
    In that case, Darboux and Riemann integrals coincide. 
  \item
The Lebesgue integral on $L^1(\R^d)$ extends the Riemann integral on $L^\circ$. 
\end{enumerate}
\end{Proposition}

\begin{proof} We focus on (ii) and check the equality $I^\circ(f) =  \inf\{ I^\Delta(f)\}$.
  From the inequality $I(f^\Delta) \leq I^\Delta(f)$, we have $I^\circ(f) \leq \inf\{ I^\Delta(f)\}$ and the problem is
  reduced to showing the reverse inequality. 
  
  Replacing $\Delta_j = \{ a_j = x_0 < \dots < x_{m_j} = b_j\}$ with
  \[
    \Delta_j^\epsilon = \{ a_j = x_0 < x_0+\epsilon < x_1 - \epsilon < x_1 < x_1 + \epsilon < \dots < x_{m_j} - \epsilon < x_{m_j}\}
  \]
  by adding $2m$ points near points in $\Delta_j$, one sees that 
  the multiple partition $\Delta^\epsilon = \Delta^\epsilon_1\times \dots \times \Delta^\epsilon_d$ for a small $\epsilon>0$ satisfies
  $\sup f(\overline{R_i^\epsilon}) \leq f^{\Delta}$ on $\sqcup_i \overline{R_i^\epsilon}$, where $R_i^\epsilon$ is shrinked from $R_i$ by $\epsilon$
  on the boundary and satisfies $\sum_i |R_i^\epsilon| = \prod_j(b_j-a_j-2m_j\epsilon)$.
  By simply estimating the value of $f$ outside $\sqcup_i R_i^\epsilon$ by $\| f\|$, we obtain 
  \[
  I^{\Delta^\epsilon}(f) \leq I(f^{\Delta}) + 2\| f\| \left(\prod_j(b_j-a_j) - \prod_j (b_j-a_j-2m_j\epsilon)\right), 
\]
whence $\inf \{I^\Delta(f)\} \leq \inf \{ I(f^\Delta)\} = I^\circ(f)$.
\end{proof}

\begin{Remark}
  By a similar idea in the proof of (ii), we can even show that
  $\lim_{|\Delta| \to 0} I^\Delta(f) = I^\circ(f)$ for a doubly bounded function $f$. 
\end{Remark}

The Darboux approximation is further generalized into the following form:
By a rectangular tiling,
we shall mean a finite family $T = (T_j)_{1 \leq j \leq n}$ of closed rectangles satisfying $T_j^\circ \not= \emptyset$
and $T_i^\circ \bigcap T_j^\circ = \emptyset$ ($i \not=j$) with its support defined by $[T] = \bigcup T_j$
and the mesh size by $d(T) = \max\{ d(T_j)\}$. Here $d(T_j)$ denotes the diameter of $T_j$.

Given a doubly bounded function $f$ on $\R^d$ supported by $[T]$, upper and lower Darboux sums are defined by
\[
  I^T(f) = \sum_{j=1}^n |T_j| \sup\{f(T_j)\},
  \quad
I_T(f) = \sum_{j=1}^n |T_j| \inf\{f(T_j)\}.
\]
Notice that these sums are of the form $I(\psi)$ and $I(\varphi)$ with $\varphi \leq f \leq \psi$, where
\begin{align*}
  \psi &= ([T] - \sum_j T_j^\circ) \| f\| + \sum_j T_j^\circ \sup f(T_j),\\
  \varphi &= - ([T] - \sum_j T_j^\circ) \| f\| + \sum_j T_j^\circ \inf f(T_j).
\end{align*}
Consequently $I^\circ(f) \leq I^T(f)$ and $I_\circ(f) \geq I_T(f)$ ($[f] \subset [T]$.

Furthermore, given a family $\xi = \{\xi_i\}$ of sample points $\xi_i \in T_i$, set
\[
  I_{T,\xi}(f) = \sum_{i=1}^m |T_i| f(\xi_i).
\]
Notice that $I_{T,\xi}(f)$ is a positive linear functional of $f$.

The following is a refinement of the classical theorem on Darboux approximation.

\begin{Theorem}[Darboux] 
Let $f:\R^d \to \R$ be a doubly bounded function. For a tiling $T$ supporting $f$, we then have 
  \[
    \lim_{d(T) \to 0} I^T(f) = I^\circ(f),
    \quad
    \lim_{d(T) \to 0} I_T(f) = I_\circ(f). 
  \]
 Moreover if $f$ is Riemann integrable, we have
  \[
    \lim_{d(T) \to 0} I_{T,\xi}(f) = \int_{\R^d} f(x)\, dx.
  \]
\end{Theorem}

\begin{proof}
  Let $[a,b]$ satisfy $[a,b]f = f$. 
  Given $\epsilon>0$, choose a multipartition $\Delta$ of $[a,b]$ with $(R_i)$ the family of open parts so that
  $I^\circ(f) \geq \sum_i |R_i| \sup f(R_i) + \epsilon$. Let $J$ be the totality of index $j$ for which $T_j$ has a non-empty intersection
  with the skelton $[a,b] \setminus (\sqcup_i R_i)$. For each $i$, let $J_i = \{ j; T_j \subset R_i\}$.
  
  Then the indices of $T$ are divided into $J$ and $\sqcup_i J_i$ in such a way that 
  \begin{align*}
    \sum_i |R_i| \sup f(R_i)
    &\geq \sum_i \sum_{j \in J_i} |T_j| \sup f(R_i) - \sum_{j \in J} |T_j| \| f\|\\
    &\geq \sum_i \sum_{j \in J_i} |T_j| \sup f(T_j) - \sum_{j \in J} |T_j| \| f\|\\
    &= \sum_{j \not\in J} |T_j| \sup f(T_j) - \sum_{j \in J} |T_j| \| f\|\\
    &\geq \sum_{j} |T_j| \sup f(T_j) - 2\sum_{j \in J} |T_j| \| f\|. 
  \end{align*}
  
  In the notation around the previous proposition, 
  we then see that
  \[
    \sum_{j \in J} |T_j| \leq \sum_i (|R_i| - |R_i^{d(T)}|) \to 0\quad
    (d(T) \to 0),
  \]
  which enables us to choose $\delta>0$ so that
  $d(T) \leq \delta$ implies $\sum_{j \in J} |T_j| \leq \epsilon$. Thus, for $T$ satisfying $d(T) \leq \delta$ and supporting $f$, we have
  $I^\circ(f) \geq I^T(f) + \epsilon - 2\| f\| \epsilon$, whence $I^\circ(f) \geq \lim_{d(T) \to 0} I^T(f)$. 
\end{proof}

\section{More on Integrability}\label{moreon}
We here introduce other definitions of integrability on real-valued functions, which turn out to be equivalent as seen below.
The following can be read after the section on null sets. 

A function $f: X \to \R$ is said to be \index{R-integrable}R-integrable\footnote{Named after F.~Riesz, not Riemann.}
if we can find $f_\updownarrow \in L_\updownarrow$ so that $I_\updownarrow(f_\updownarrow) \not= \pm \infty$ and
$f \;\mathring{=}\; f_\uparrow + f_\downarrow$. From Corollary~\ref{practical}, this implies that $f \in L^1$ and
$I^1(f) = I_\uparrow(f_\uparrow) + I_\downarrow(f_\downarrow)$.

Related to this, $f$ is said to be \index{M-integrable}M-integrable\footnote{Named after J.~Mikusi\'nski
  but a closely related (and equivalent) integrability was also discussed by M.~Stone in \cite{Sto}.} if
we can find sequences $(f_n)$ and $(\varphi_n)$ in $L$ satisfying $|f_n| \leq \varphi_n$, $\sum I(\varphi_n) < \infty$ and
$f(x) = \sum f_n(x)$ for $x \in [\sum \varphi_n < \infty] \iff \sum \varphi_n(x) < \infty$ (the condition is expressed by
  $f \stackrel{(\varphi_n)}{\simeq} \sum_n f_n$).
Then, letting $\varphi = \sum \varphi_n \in L_\uparrow^+$,  we have
\[
  [\varphi < \infty] f = \sum [\varphi < \infty] f_n = [\varphi< \infty] \sum (f_n \vee 0) + [\varphi<\infty] \sum (f_n \wedge 0). 
\]
Since $\sum (f_n \vee 0) \in L_\uparrow$, $\sum (f_n \wedge 0) \in L_\downarrow$ and 
$\sum |f_n| \leq \varphi$ with $I_\uparrow(\varphi) < \infty$, this implies that $f$ is R-integrable in view of negligibleness of
$[\varphi=\infty]f$, whence it belongs to $L^1$ and
\begin{multline*}
  I^1(f) = I_\uparrow(\sum f_n \vee 0) + I_\downarrow(\sum f_n \wedge 0)\\
  = \sum I(f_n \vee 0) + \sum I(f_n\wedge 0) = \sum I(f_n).
\end{multline*}
Note here that outer summations are absolutely convergent thanks to $\sum I(|f_n|) \leq \sum I(\varphi_n) < \infty$.

Henceforth we focus on the fact that integrability in turn implies M-integrability.
This is, however, not so obvious and we need to know more about M-integrability. 

From the definition, it is immediate to see that M-integrable functions constitute a linear subspace of $L^1$.
Moreover it is a sublattice of $L^1$:

\begin{Lemma}\label{absolute}
  For an M-integrable function $f$ with $f \stackrel{(\varphi_n)}{\simeq} \sum f_n$, $|f|$ is M-integrable and
$I^1(|f|) = \lim I(|f_1+ \dots + f_n|)$. 
\end{Lemma}

\begin{proof}
  Let $g_n = f_1 + \dots + f_n$ and $(h_n)$ be the difference of $(|g_n|)$ in $L^+$: $h_1 = |g_1|$ and
  $h_n = |g_n| - |g_{n-1}|$ for $n \geq 2$. Then $|h_n| \leq |g_n - g_{n-1}| = |f_n| \leq \varphi_n$ and
  $|f(x)| = \lim |g_n(x)| = \sum h_n(x)$ if $\sum \varphi_n(x) < \infty$.
  
  Thus $|f| \stackrel{(\varphi_n)}{\simeq} \sum_n h_n$ and $I^1(|f|) = \sum I(h_n) = \lim I(|g_n|)$. 
\end{proof}

\begin{Lemma}\label{NM}
Null functions are M-integrable. 
\end{Lemma}

\begin{proof}
  Let $f: X \to \R$ be a null function. Then we can find a sequence $h_m \in L_\uparrow^+$ so that $|f| \leq h_m$ and $I_\uparrow(h_m) \leq 1/2^m$.
  Write $h_m = \sum_n \varphi_{m,n}$ with $\varphi_{m,n} \in L^+$ so that $\sum_{m,n} I(\varphi_{m,n}) \leq 1$.
  Now $\sum_{m,n} \varphi_{m,n}(x) < \infty$ implies $h_m(x) < \infty$ ($m \geq 1$) and hence
  $\sum_m |f(x)| \leq \sum_m h_m(x) < \infty$, i.e., $f(x) = 0$. Thus
  $f \stackrel{(\varphi_{m,n})}{\simeq} \sum_{m,n} 0$ and $f$ is M-integrable. 
\end{proof}

\begin{Lemma}\label{PA} 
  Given $\epsilon> 0$ and an M-integrable positive function $h$, we can find sequences $(h_n)$ and $(\varphi_n)$ in $L$ so that
  $h \stackrel{(\varphi_n)}{\simeq} \sum h_n$ and $\sum I(\varphi_n) \leq I^1(h) + 3\epsilon$. 
\end{Lemma}

\begin{proof}
  Let $h \stackrel{(\psi_n)}{\simeq} \sum g_n$ in $L$ and choose $m'$ so that $\sum_{n>m'} I(\psi_n) \leq \epsilon$. 
  From Lemma~\ref{absolute} we can find $m''$ so that $I(|g_1 + \dots + g_n|) \leq I^1(h) + \epsilon$ ($n \geq m''$).
  
  Let $m = m' \vee m''$ and set $h_1 = g_1+ \dots+g_m$, $\phi_1 = \psi_1 + \dots + \psi_m$, $h_n = g_{m+n-1}$ and $\phi_n = \psi_{m+n-1}$
  for $n \geq 2$.
  With this arrangement, we have $h \stackrel{(\phi_n)}{\simeq} \sum h_n$ and
  \begin{align*}
    \sum_n I(|h_n|) &= I(|g_1+\dots+g_m|) + \sum_{n>m} I(|g_n|)\\
    &\leq I^1(h) + \epsilon + \sum_{n > m} I(\psi_n) \leq I^1(h) + 2\epsilon. 
  \end{align*}
  Finally set $\varphi_n = |h_n| + \epsilon' \phi_n$ with $\epsilon'>0$.
  Then, in view of $\sum \varphi_n(x) < \infty \iff \sum \phi_n(x) < \infty$, we have $h \stackrel{(\varphi_n)}{\simeq} \sum_n h_n$, whereas 
  \[
    \sum I(\varphi_n) = \sum I(|h_n|) + \epsilon' \sum I(\phi_n)
    \leq I^1(h) + 2\epsilon + \epsilon' \sum I(\phi_n).
  \]
  Thus, choosing $\epsilon'>0$ so that $\epsilon' \sum I(\phi_n) \leq \epsilon$, 
  $\sum I(\varphi_n) \leq I^1(h) + 3\epsilon$. 
\end{proof}

\begin{Corollary}[Monotone Continuity]\label{MC}\index{monotone continuity}
  Let $(f_n)$ be a decreasing sequence of M-integrable functions satisfying $f_n \downarrow f$ with $f:X \to \R$ and
  $\inf \{I^1(f_n)\} > -\infty$.
  Then $f$ is M-integrable and $I^1(f_n) \downarrow I^1(f)$. 
%
\end{Corollary}

\begin{proof}
  Define M-integrable positive functions by $\theta_n = f_n - f_{n+1}$, which satisfy
  $\sum_n \theta_n = f_1 - f$. Thus the problem is reduced to showing that $\sum_n \theta_n$ is M-integrable and
  $I^1(\sum_n \theta_n) = \sum I^1(\theta_n)$.
  
  Now express $\theta_n \stackrel{(\varphi_{n,k})}{\simeq} \sum_k \theta_{n,k}$ with $\theta_{n,k} \in L$ and $\varphi_{n,k} \in L^+$
  so that $\sum_k I(\varphi_{n,k}) \leq I^1(\theta_n) + 1/2^n$ (Lemma~\ref{PA}).
  Then 
  \[
    \sum_{n,k} I(\varphi_{n,k}) \leq \sum_n I^1(\theta_n) + \sum_n \frac{1}{2^n}
    = 1 + I^1(f_1) - \inf \{ I^1(f_n)\} < \infty. 
  \]

  Moreover, for $x$ satisfying $\sum_{n,k} \varphi_{n,k}(x) < \infty$, $\sum_k \varphi_{n,k}(x) < \infty$ implies
  $\theta_n(x) = \sum_k \theta_{n,k}(x)$ ($n \geq 1$) and hence
  $\sum_n \theta_n(x) = \sum_{n,k} \theta_{n,k}(x)$. Thus $\sum_n \theta_n \stackrel{(\varphi_{n,k})}{\simeq} \sum_{n,k} \theta_{n,k}$
  is M-integrable and
  \[
    I^1(\sum \theta_n) = \sum_{n,k} I(\theta_{n,k}) = \sum_n I^1(\theta_n).
  \]  
\end{proof}

\begin{Lemma}\label{monotone-up}
If $f, g \in L_\uparrow$ satisfy $f \leq g$ and $I_\uparrow(g) < \infty$,
then $[g<\infty]f$ is M-integrable and $I^1([g<\infty]f) = I_\uparrow(f)$.
\end{Lemma}

\begin{proof}
 Letting $f_n \uparrow f$ and $g_n \uparrow g$ with $f_n, g_n \in L$,
 $\varphi_k = f_k - f_{k-1} + g_k - g_{k-1}$ ($k \geq 2$) in $L^+$ majorizing $f_k - f_{k-1}$ are summed up to
 $\sum_{k \geq 2} \varphi_k = f - f_1 + g - g_1$, which satisfies
 \[
   g-g_1 \leq \sum_{k \geq 2} \varphi_k \leq 2g  - f_1 - g_1. 
 \]
 Consequently $[\sum_{k \geq 2} \varphi_k = \infty] = [g=\infty]$ and
 \[
   \sum_{k \geq 2} I(\varphi_k) = I_\uparrow(\sum_{k \geq 2} \varphi_k)
   \leq I_\uparrow(2g - f_1 - g_1) = 2I_\uparrow(g) - I(f_1) - I(g_1) < \infty.
 \]
 Thus, by adding $\varphi_1 = |f_1|$ as an initial term, if $\sum_n \varphi_n(x) < \infty \iff g(x) < \infty$, then 
   \[
     f(x) = f_1(x) + \sum_{k=2}^\infty (f_k(x) - f_{k-1}(x))
   \]
and one sees that
\[
  [g<\infty]f \stackrel{(\varphi_n)}{\simeq} f_1 + (f_2-f_1) + \cdots. 
\]
\end{proof}

Being prepared, we show that integrable functions are M-integrable.
Start with the fact that $\underline{I}(f) = \overline{I}(f) \in \R$ for $f \in L^1$.
Since $L_\updownarrow$ are lattices, we can choose a decreasing sequence
$(h_n)$ in $L_\uparrow$ and an increasing sequence $(g_n)$ in $L_\downarrow$ so that $g_n \leq f \leq h_n$, 
$I_\uparrow(h_n) \downarrow \overline{I}(f)$ and $I_\downarrow(g_n) \uparrow \underline{I}(f)$ by Lemma~\ref{DA} and Theorem~\ref{DE}.

Then limit functions $g = \lim g_n$ and $h = \lim h_n$ satisfy
$g \leq f \leq h$ and we see that
$h:X \to (-\infty,\infty]$, 
$g: X \to [-\infty,\infty)$ and $h-g:X \to [0,\infty]$ so that $(h_n-g_n) \downarrow (h-g)$.

Thus $\overline{I}(h-g) \leq \overline{I}(h_n-g_n) = I_\uparrow(h_n-g_n)$ implies that $h-g$ is a null function
due to $I_\uparrow(h_n-g_n) \downarrow 0$, whence $f$ is different from $h$ (or $g$) at most on a null set. 

We now apply Lemma~\ref{monotone-up} for $h_n \leq h_1$ to see that $[h_1 < \infty]h_n$ is M-integrable and
$I^1([h_1<\infty]h_n) = I_\uparrow(h_n)$.

Finally apply Corollary~\ref{MC} to $[h_1<\infty]h_n \downarrow [h_1<\infty]h$ 
with $I^1([h_1<\infty]h_n) = I_\uparrow(h_n)$ bounded below to conclude that $[h_1<\infty]h$ is M-integrable and
\[
  I^1([h_1<\infty]h) = \lim I^1([h_1<\infty]h_n) = \lim I_\uparrow(h_n).
\]

Since both $h-f$ and $h-[h_1<\infty]h$ are null functions (consequently M-integrable by Lemma~\ref{NM}), so is $f - [h_1<\infty]h$ and
$f$ is M-integrable as s sum of M-integrable finctions $[h_1<\infty]h$ and $f - [h_1<\infty]h$.

\bigskip
As an application of equivalent descriptions of integrability, we show that the complexified Daniell extension $(L_\C^1,I^1)$ of
an integral system $(L,I)$ satisfies $|f| \in L_\C^1$ ($f \in L_\C^1$) if $L_\C$ is a complex lattice.
This part can be read after the section on complex functions. 

\begin{Lemma}
  For $f \in L^1$, we can find a function $\varphi \in L^1$, a sequence $(f_n)$ in $L$ and a null set $N$ so that
  $|f_n(x)| \leq \varphi(x)$ and $\lim_n f_n(x) = f(x)$ for $x \not\in N$. 
\end{Lemma}

\begin{proof}
  Since $f$ is R-integrable, $f \mathring{=} f_\uparrow + f_\downarrow$ with $f_\updownarrow \in L_\updownarrow$ satisfying
  $I_\updownarrow(f_\updownarrow) \in \R$. Let $h_n \uparrow f_\uparrow$ and $g_n \downarrow f_\downarrow$ with $g_n, h_n \in L$.
  From finiteness of $I_\updownarrow(f_\updownarrow)$, $[f_\updownarrow = \pm \infty]$ are null sets (Theorem~\ref{MI}) and so is
  \[
    N = [f_\uparrow = \infty] \cup [f_\downarrow = -\infty] \cup
    \Bigl([f \not= f_\uparrow + f_\downarrow] \cap [f_\uparrow < \infty] \cap [f_\downarrow > -\infty]\Bigr).
  \]
  In view of $|g_n| \leq g_1 - g_n + |g_1| \leq g_1 - f_\downarrow + |g_1|$
  and $|h_n| \leq h_n - h_1 + |h_1| \leq f_\uparrow - h_1 + |h_1|$, define $\varphi \in L^1$ by
  \[
    \varphi = [f_\downarrow > - \infty] (g_1 - f_\downarrow + |g_1|) + [f_\uparrow <\infty] (f_\uparrow-h_1 + |h_1|). 
  \]
  Then $f_n = g_n + h_n$ satisfies all the requirements.    
\end{proof}

\begin{Proposition}\label{LC1}
  Let $(L,I)$ be an integral system with $L_\C$ a complex lattice.
  Then $L_\C^1 = L^1 + iL^1$ is a complex lattice and the complexified Daniell extension $I^1: L_\C^1 \to \C$ satisfies
\[
  |I^1(f)| \leq I^1(|f|)
  \quad
  (f \in L^1_\C). 
\] 
\end{Proposition}

\begin{proof}
  Consider $f+ig \in L_\C$ with $f,g \in L^1$. By the previous lemma, we can find a null set $N$, $\varphi, \psi \in L^1$,
  and sequences $(f_n)$, $(g_n)$ in $L$ so that $|f_n(x)| \leq \varphi(x)$, $|g_n(x)| \leq \psi(x)$ and
  $\lim_n f_n(x) = f(x)$, $\lim_n g_n(x) = g(x)$ for $x \not\in N$.

  Then $|f(x)+ig(x)| = \lim_n |f_n(x) + i g_n(x)| \leq \varphi(x) + \psi(x)$ ($x \not\in N$) and the dominated convergence theorem is
  applied to a sequence $|f_n+ig_n|$ in $L$ to conclude that $|f+ig|$ is integrable. Thus $L_\C^1$ is a complex lattice and
  hence the positivity of $I^1$ on $L^1$ gives rise to the integral inequality on $L_\C^1$. 
%
%
\end{proof}

\section{Measurable Sets and Functions}
Given an integral system $(L,I)$ on a set $X$ with $(L^1,I^1)$ its Daniell extension,
recall that a subset $A$ of $X$ is said to be \textbf{$I$-integrable}\index{+Iintegrable@$I$-integrable} or simply integrable
if it belongs to $L^1$ as an indicator function (Definition~\ref{D-integrable}) and
\textbf{$\bm\sigma$-integrable}\index{+sigma-integrable@$\sigma$-integrable}
if it is a countable union of integrable sets (Definition~\ref{sigma-integrable}).

We say that the integral system $(L,I)$ or the integral $I$ itself is 
\textbf{$\bm\sigma$-finite}\index{+sigma-finite@$\sigma$-finite} (finite)
if $X$ is $\sigma$-integrable (integrable). Notice that the Lebesgue integral is $\sigma$-finite (Proposition~\ref{m-sets} (iv)). 


If $I$ is $\sigma$-finite so that $X_n \uparrow X$ with $X_n$ integrable, then $1 \in L^1_\uparrow$ in view of $X_n \uparrow 1$. 

\begin{figure}[h]
  \centering
 \includegraphics[width=0.4\textwidth]{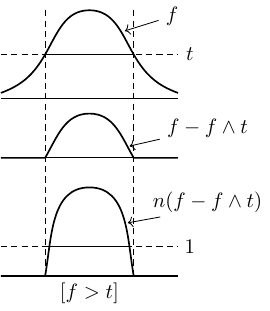}
 \caption{Push-up}
\end{figure}

\begin{Lemma}[Push-up]\label{pushup}\index{push-up}
    For $t \in \R$ and $f:X \to (-\infty,\infty]$, $1 \wedge (n(f-f\wedge t)) \uparrow [f>t]$ as $n \to \infty$. 
    Recall here that $[f>t] \subset X$ is identified with its indicator function. 
  \end{Lemma}

  \begin{Corollary}\label{ipushup}
    Assume that $1 \in L^1_\uparrow$ and let $f \in L^1_\uparrow$ satisfy $I^1_\uparrow(f) < \infty$.
    Then, for $r>0$,  $r\wedge f$, $[f>r]$ and $[f \geq r]$ are integrable. 
  \end{Corollary}

  \begin{proof}
    Express $h_n \uparrow 1$ and $f_n \uparrow f$ with $0 \leq h_n \in L^1$ and $f_n \in L^1$.
    Then $(rh_n) \wedge f_n \uparrow r\wedge f$ so that
    \[
      I^1((rh_n) \wedge f_n) \leq I^1(f_n) \leq \lim_{n \to \infty} I^1(f_n) = I^1_\uparrow(f) < \infty 
    \]
    and Theorem~\ref{MI} is applied to see $r\wedge f \in L^1$.

    Thus $f- r \wedge f \in L^1_\uparrow$ with $I^1_\uparrow(f - r\wedge f) = I^1_\uparrow(f) - I^1(r\wedge f) < \infty$ and
   $[f > r]$ is an increasing limit of $1\wedge (n(f-r\wedge f)) \in L^1$ by the lemma in such a way that $[f>r] \leq f/r$. 
   Theorem~\ref{MI} is again applied to observe that $[f>r]$ is integrable. 
\end{proof}

\begin{Proposition}
      $I$ is $\sigma$-finite if and only if $1 \in L^1_\uparrow$.
  \end{Proposition}

\begin{proof}
To see the if part, the corollary is applied for $h_n \in L^1$ to see $[h_n > 1/m] \in L^1$ ($m,n \geq 1$), which together with $h_n \uparrow 1$ 
implies that $X = \bigcup_{n \geq 1} [h_n > 0] = \bigcup_{m,n \geq 1} [h_n > 1/m]$ is $\sigma$-integrable.
\end{proof}



  



Given a positive function $h:X \to [0,\infty]$, we introduce a \textbf{level approximation}\index{level approximation}
$h_\varrho:X \to [0,\infty)$ of $h$ relative to a finite set $\varrho = \{ r_1 < \dots < r_n\}$ in $(0,\infty)$ by 
\begin{align*}
  h_\varrho
  &= \sum_{j=1}^{n-1} [r_j < h \leq r_{j+1}] r_j + [r_n < h] r_n\\
  &= [h>r_1]r_1 + [h > r_2](r_2-r_1) + \dots + [f>r_n] (r_n - r_{n-1}). 
\end{align*}

The correspondence $h \mapsto h_\varrho$ is increasing in $\varrho$ and monotone in $h$.
Moreover, if $\varrho_n$ is increasing in $n \geq 1$ in such a way that $|\varrho_n| \to 0$, then
$h_{\varrho_n} \uparrow h$. Here the mesh $|\varrho|$ of $\varrho$ is defined by
 \[
  |\varrho| = \max\{r_1, r_2-r_1,\dots, r_n - r_{n-1}, 1/r_n\}.
\]

\begin{Example}[binary partition]\label{bp}\index{binary partition}
  \[
    \varrho_n = \Bigl\{ \frac{1}{2^n}, \frac{2}{2^n},\dots, \frac{n2^n-1}{2^n}, \frac{n2^n}{2^n}\Bigr\},
    \quad
    |\varrho_n| = \frac{1}{n}. 
  \]
\end{Example}

From here on, $I$ is assumed to be $\sigma$-finite and
$\sigma$-integrable sets are then referred to as \textbf{$\bm I$-measurable}\index{measurable} (simply measurable) sets.
Let $\cL = \cL(I)$ be the collection of $I$-measurable sets.
Notice that $\cL$ is closed under taking countable unions, countable intersections and differences (Proposition~\ref{m-sets}).
Thus the $\sigma$-finiteness assumption means that $\cL$ is also closed under taking complements in $X$. 

\begin{Lemma}\label{m-lemma}
For $f \in L^1_\uparrow$ and $t \in \R$, $[f > t] \in \cL$ and $[f \geq t] \in \cL$. 
\end{Lemma}

\begin{proof}
  Since $[f \geq t] = \bigcap_{m \geq 1} [f > t-1/m]$, it suffices to check $[f > t] \in \cL$. 
  Letting $f_n \uparrow f$ with $f_n \in L^1$, $[f>t] = \bigcup_n [f_n>t]$ shows that
  the problem is further reduced to the case $f \in L^1$.
  
  If $t>0$, $[f>t]$ is integrable (Corollary~\ref{ipushup}) and then $[f > 0] = \bigcup_n [f > 1/n]$ is $\sigma$-integrable.

  If $t = -r < 0$, $[f>t] = [-f < r] = X \setminus [-f \geq r]$ is also $\sigma$-integrable as a complement of an integrable set. 
\end{proof}

\begin{Proposition}
  The following conditions on a function $f:X \to [-\infty,\infty]$ are equivalent.
  \begin{enumerate}
  \item
    Both $0\vee f$ and $0\vee(-f)$ belong to $L^1_\uparrow$.
  \item
   For each $t \in \R$, $[f > t] \in \cL$.  
  \item
   For each $t \in \R$, $[f \geq t] \in \cL$.  
  \item
   For each $t \in \R$, $[f < t] \in \cL$.  
  \item
   For each $t \in \R$, $[f \leq t] \in \cL$.  
  \end{enumerate}
\end{Proposition}

\begin{proof}
  The equivalence from (ii) to (v) is immediate.
  
  (i)$\Rightarrow$ (ii): Let $f^\pm = 0 \vee(\pm f)$ so that $f = f^+ - f^-$.
  Then $[f^\pm > t]$ is measurable by Lemma~\ref{m-lemma},  
  whence $[f > t] = [f^+ > t] \in \cL$ for $t>0$ and so is
  \[
    [f > 0] = [f^+ > 0] = \bigcup_n [f^+ > 1/n].
  \]
  For $t = -r < 0$,
  \[
    [f > t] = [f^- < r] = \bigcup_{n \geq 1} [f^- \leq r -1/n] = \bigcup_{n \geq 1} (X \setminus [f^- > r -1/n])
  \]
  is measurable as well.

  (ii)$\Rightarrow$ (i): From $\sigma$-finiteness, $X_m \uparrow X$ with $X_m$ integrable. 
    Let $f^\pm_{n,m}$ be the level approximation of $X_m (0\vee (\pm f))$ relative to the binary partition $\varrho_n$ in Example~\ref{bp},
    which is integrable as a linear combination of integrable sets 
    \begin{align*}
      X_m \cap [r_j < 0\vee (\pm f) \leq r_{j+1}]
      &= X_m \cap [r_j < \pm f \leq r_{j+1}]\\
      &= X_m \cap \Bigl([\pm f > r_j] \setminus [\pm f > r_{j+1}]\Bigr) 
    \end{align*}
  (cf.~Proposition~\ref{m-cut}) and satisfies
  $f^\pm_{n,m} \uparrow X_m (0 \vee (\pm f))$ for each $m \geq 1$, whence
  $0 \vee(\pm f)$ is in $L^1_\uparrow$ as an increasing limit of $f^\pm_{n,n} \in L^1$. 
\end{proof}

\begin{Definition}
  Under the assumption of $\sigma$-finiteness on $I$,
  a function $f:X \to [-\infty,\infty]$ is said to be \textbf{$\bm I$-measurable}\index{measurable} (or simply measurable) if it satisfies
  the equivalent conditions in the above proposition. When $I$ is the volume integral in $\R^d$,
  it is called \textbf{Lebesgue measurable}\index{Lebesgue measurable}. 
\end{Definition}

Here are basic properties of measurable functions.

\begin{Proposition}~
  \begin{enumerate}
    \item
  Measurable functions constitute a lattice so that they are closed under taking sequential limits in $[-\infty,\infty]^X$.
\item
  For real-valued measurable functions $f_1, \dots, f_m$ and a continuous function $\phi: \R^m \to \R$,
  $\phi(f_1,\dots,f_m)$ is measurable.
\end{enumerate}
\end{Proposition}

\begin{proof}
(i)  Let $(f_n)$ be a sequence of measurable functions. Then $[\bigvee f_n > t] = \bigcup_n [f_n > t] \in \cL$ and
  $[\bigwedge f_n \leq t] = \bigcap_n [f_n \leq t]$ show that $\sup f_n$ and $\inf f_n$ are measurable.
  Consequently, $\limsup f_n = \inf_m \sup_{n \geq m} f_n$ and $\liminf f_n = \sup_m \inf_{n \geq m} f_n$ are measurable as well.

  (ii) Since $\phi$ is continuous, $[\phi > t]$ is an open subset of $\R^m$ and we can find rectangles $(a_n,b_n]$ in $\R^m$ so that
  $[\phi>t] = \bigcup_n (a_n,b_n]$ (Proposition~\ref{CU} (i)). Then 
  \[
    [\phi(f_1,\dots,f_m) > t] = \bigcup_n \Bigl( [a_n^{(1)} < f_1 \leq b_n^{(1)}] \cap \dots \cap [a_n^{(m)} < f_m \leq b_n^{(m)}] \Bigr)
  \]
 belongs to $\cL$ in view of $[a_n^{(j)} < f_j \leq b_n^{(j)}] = [f_j > a_n^{(j)}] \setminus [f_j > b_n^{(j)}] \in \cL$. 
\end{proof}

\begin{Corollary}~ 
  \begin{enumerate}
  \item
    If $f,g:X \to \R$ are measurable, so are $f+g$ and $fg$.
  \item
    If $f:X \to \C$ is measurable in the sense that $\Re f$ and $\Im f$ are measurable, so is $|f|^r$ for any $r>0$. 
  \end{enumerate}
\end{Corollary}







\begin{Proposition}
A measurable function $f$ is integrable if and only if $|f| \leq g$ with $g$ an integrable function. 
\end{Proposition}

\begin{proof}
  If $|f| \leq g$, $0\vee(\pm f) \in L^1_\uparrow$ satisfies $0 \vee (\pm f) \leq g$ and then $0\vee(\pm f)$ is integrable. 
  Hence $f = (0\vee f) - (0\vee (-f))$ is integrable as well. 
\end{proof}

\begin{Corollary}
  If $f_1,\dots,f_m: X \to \R$ are integrable functions, so is $\left( \sum_{i=1}^m |f_i|^p \right)^{1/p}$ for $1 \leq p < \infty$.
  This for $m=2$ and $p=2$ means that $L^1(I) + iL^1(I)$ is a complex lattice if $I$ is $\sigma$-finite. 
\end{Corollary}

\begin{proof}
  This follows from $\left( \sum_{i=1}^m |f_i|^p \right)^{1/p} \leq \sum_{i=1}^m |f_i|$, which is a consequence of
  $\sum_{j=1}^m t_j^p \leq \sum_j t_j$ ($0 < t_j \leq 1$) for the choice $t_j = |f_j|/(|f_1|+ \dots + |f_m|)$. 
\end{proof}

\section{Determinant Formulas}\label{determinant}
Let $A$ be an $m\times n$ matrix and $B$ an $n\times m$ matrix.
\textbf{Sylvester's formula}\index{Sylvester's formula} is the identity
\[
  t^n \det(tI_m + AB) = t^m \det(tI_n + BA)
\]
as polynomials of indeterminate $t$. 
Here $I_m$ and $I_n$ denote unit matrices of size $m$ and $n$ respectively.

In the following identities\footnote{These are based on Gaussian eliminations.} on square matrices
\begin{align*}
  \begin{pmatrix}
    tI_m & A\\
    0 & I_n
  \end{pmatrix}
  \begin{pmatrix}
    I_m & -A\\
    B & tI_n
  \end{pmatrix}
  &=
  \begin{pmatrix}
    tI_m + AB & 0\\
    B & tI_n
  \end{pmatrix},\\
   \begin{pmatrix}
    I_m & 0\\
    -B & I_n
  \end{pmatrix}
  \begin{pmatrix}
    I_m & -A\\
    B & tI_n
  \end{pmatrix}
  &=
  \begin{pmatrix}
    I_m & -A\\
    0 & tI_n + BA
  \end{pmatrix}, 
\end{align*}
we take determinants to have
\begin{align*}
  t^m  \det\begin{pmatrix}
    I_m & -A\\
    B & tI_n
  \end{pmatrix}
  &= t^n \det(tI_m+AB),\\
  \det\begin{pmatrix}
    I_m & -A\\
    B & tI_n
   \end{pmatrix}
  &= \det(tI_n + BA)
\end{align*}
and the formula is obtained by eliminating the intermediate determinant. 

Now assume that $m < n$. The \textbf{Cauchy-Binet formula}\index{Cauchy-Binet formula} then states that
\[
  \det(AB) = \sum_{|J|=m} \det(BA)_{k,l \not\in J}. 
\]
Here the summation is taken over finite subsets $J$ of $\{1,2,\dots,n\}$ satisfying $|J| = m$ and
the determinants in the right hand side are for square matrices $(BA)_{k,l \not\in J}$ of size $n-m$. 

\begin{proof}
  For notational simplicity, we just check the case $m=n-1$. In Sylvester's formula, compare coefficients of $t^n$.
  From the left hand side, we have $\det(AB)$. The right hand side
 \[
    t^{n-1} \det(tI_n + BA)
    = t^{n-1} \sum_{\sigma \in S_n} \epsilon(\sigma) \prod_{k=1}^n (t\delta_{k,\sigma(k)} + (BA)_{k,\sigma(k)}) 
  \]
  ($\epsilon(\sigma)$ being the signature of a permutation $\sigma$) is expanded in $t$ to
  \[
    t^{n-1} \sum_{\sigma \in S_n} \epsilon(\sigma) (BA)_{k,\sigma(k)}
    + t^n \sum_{\sigma \in S_n} \epsilon(\sigma) \sum_{j=1}^n \delta_{j,\sigma(j)} \prod_{k\not= j} (BA)_{k,\sigma(k)} + \cdots
  \]
  and the coefficient of $t^n$ is given by
  \begin{align*}
    \sum_{\sigma \in S_n} \epsilon(\sigma) \sum_{j=1}^n \delta_{j,\sigma(j)} \prod_{k\not= j} (BA)_{k,\sigma(k)}
    &= \sum_{j=1}^n \sum_{\sigma \in S_n} \epsilon(\sigma)  \delta_{j,\sigma(j)} \prod_{k\not= j} (BA)_{k,\sigma(k)}\\
    &= \sum_{j=1}^n \sum_{\sigma(j) = j} \epsilon(\sigma) \prod_{k\not= j} (BA)_{k,\sigma(k)}\\
    &= \sum_{j=1}^n \det (BA)_{k\not=j, l\not=j}. 
  \end{align*}
\end{proof}

\section{Functional Tensor Products}\label{tensor}
Let $X$ and $Y$ be sets. For functions $f:X \to \R$ and $g:Y \to \R$, the function
$(x,y) \mapsto f(x)g(y)$ is denoted by $f\otimes g$ and called the tensor product of $f$ and $g$, which is linear in $f$ and $g$ seperately.
Notice that this is completely different from the product operation $(fg)(x) = f(x)g(x)$ even for $X=Y$ and $f\otimes g \not= g\otimes f$ generally.

As a basic property of tensor product, we remark here that, for linearly independent finite families $(f_i)_{1 \leq i \leq m}$ in $\R^X$ and
$(g_j)_{1 \leq j \leq n}$ in $\R^Y$, the coupled family $(f_i\otimes g_j)$ is linearly independent:
If $\sum \lambda_{i,j} f_i\otimes g_j = 0$ as a function on $X\times Y$, for each $x \in X$, we have 
\[
  \sum_j \Bigl(\sum_i \lambda_{i,j} f_i(x)\Bigr) g_j = 0
\]
as a function on $Y$, whence $\sum_i \lambda_{i,j} f_i(x) = 0$ ($1 \leq j \leq n$) by linear independence of $g_j$, i.e. 
$\sum_i \lambda_{i,j} f_i = 0$ as a function on $X$, and then $\lambda_{i,j} = 0$ by linear independence of $f_i$.

Tensor products of functions behave well in the identification of sets and indicator functions:
Let $A \subset X$ and $B \subset Y$. Then the tensor product function $A \otimes B$ indicates the product subset $A\times B$ of $X\times Y$.

For a subspace $V$ of $\R^X$ and a subspace $W$ of $\R^Y$, their tensor product $V\otimes W$ is defined to be 
the linear subspace of $\R^{X\times Y}$ generated by $\{ v\otimes w; v \in V, w \in W\}$, i.e., the set of linear combinations of
these tensor product functions. 

Tensor product operation is obviously extended in a multiple way:
Let $f_i:X_i \to \R$ ($1 \leq i \leq d$).
Then the multiple tensor product $f_1\otimes \dots \otimes f_d$ is a function on the product set $X = X_1\times \dots \times X_d$ defined by
\[
  (f_1\otimes \dots \otimes f_d)(x_1,\dots,x_d) = f_1(x_1) \cdots f_d(x_d),
\]
which is identified naturally with repeated tensor products.

For example, when $d=3$, $(f_1\otimes f_2)\otimes f_3$ and $f_1\otimes(f_2\otimes f_3)$ are identified with $f_1\otimes f_2\otimes f_3$ through
the natural identification among $(X_1\times X_2)\times X_3$, $X_1\times (X_2\times X_3)$ and $X_1\times X_2\times X_3$.

For subspaces $V_i \subset \R^{X_i}$ ($1 \leq i \leq d$), their tensor product $V_1\otimes \dots \otimes V_d$ is the linear space
generated by $\{ v_1\otimes \dots \otimes v_d; v_i \in V_i\}$.
Choose an index $i$ and elements $x_j \in X_j$ for $j \not= i$. Each $v \in V_1\otimes \dots \otimes V_d \subset \R^X$ then defines 
a function $v(x_1,\dots,x_i,\dots,x_d)$ of $x_i \in X_i$, which belongs to $V_i$, and we can evaluate it by a linear functional $\varphi$ of $V_i$
to have a function in
$V_{(i)} = V_1\otimes \dots \otimes V_{i-1}\otimes V_{i+1}\otimes \dots\otimes V_d$.

In other words, each $\varphi \in V_i^*$ gives
rise to a linear map $V_1\otimes \dots \otimes V_d \to V_{(i)}$, which is called a partial integration with respect to $\varphi$ and denoted by
\[
  \varphi^{(i)}
  = \overbrace{1\otimes \dots \otimes 1}^{\text{$(i-1)$-times}}\otimes \varphi \otimes \overbrace{1 \otimes \dots \otimes 1}^{\text{$(d-i)$-times}}. 
\]

Partial integrations are then repeated with respect to a choice $\varphi_i \in V_i^*$ ($1 \leq i \leq d$) and we obtain a linear
functional of $V_1\otimes \dots \otimes V_d$, which is independent of histories of repetition and called
the tensor product of $\varphi_1, \dots, \varphi_d$.

\printindex

\begin{thebibliography}{99}
 %
\bibitem{Fe}
 H.~Federer, Geometric Measure Theory, Springer, 1996. 
 \bibitem{F}
 O.~Forster, Analysis 3, Springer, 2017. 
%
%
 \bibitem{Gu}
 V.~ Gutev, Simultaneous extension of continuous and uniformly continuous functions, arXiv:2010.02955v2. 
\bibitem{Sa}
T.~Sakai, Riemannian Geometry, Amer.~Math.~Soc., 1996. 
\bibitem{Sch}
L.~Schwartz, Analyse IV, Hermann, 1997.
\bibitem{Schw}
  J.~Schwartz,
  The formula for change in variables in a multiple integral, Amer.~Math.~Monthly, 61(1954), 81--85.
\bibitem{Sto}
 M.~Stone, Notes on integration, I, PNAS, 34(1948), 336-342. 
\bibitem{St}
  D.W.~Stroock, Essentials of Integration Theory for Analysis, Springer, 2020. 
\end{thebibliography}
\end{document}